\documentclass[a4paper,leqno,11pt]{amsart}
\usepackage{etex} %added to cure the Latex error "! No room for a new \dimen"

\usepackage[top=1.5in,bottom=1.3in,left=1.3in,right=1.3in,marginparwidth=1.5cm]{geometry}

\usepackage{times}
\usepackage{tensor}
\usepackage[all]{xy}
\usepackage[new]{old-arrows}

\usepackage[bbgreekl]{mathbbol}

\DeclareSymbolFontAlphabet{\mathbbm}{bbold}
\DeclareSymbolFontAlphabet{\mathbb}{AMSb}%

\usepackage{type1cm} 
\usepackage{url}
\usepackage{amsmath, amssymb, amsfonts, latexsym, mdwlist, amsthm}
\usepackage{mathrsfs}
\usepackage{subfig}
\usepackage{graphicx}
\usepackage{longtable}
\usepackage{enumerate}

\usepackage{mathtools} %for \xhookrightarrow

\usepackage{enumitem}   
\usepackage{adjustbox}

\usepackage[unicode,colorlinks,plainpages=false,hyperindex=true,bookmarksnumbered=true,bookmarksopen=false,pdfpagelabels]{hyperref}
\hypersetup{urlcolor=cyan,linkcolor=blue,citecolor=red,colorlinks=true}
\usepackage{color}

\usepackage{tikz} 
\usepackage{tikz-cd}

\usetikzlibrary{arrows,arrows.meta,positioning}  % for fancier arrows in tikz + positioning

\usepackage{multirow}
\usepackage{diagbox}
\usepackage{arydshln}
\usepackage{bbold}
\usepackage{fix-cm} 
\usepackage{calligra}
\DeclareFontFamily{T1}{calligra}{}
\DeclareFontShape{T1}{calligra}{m}{n}{<-> s * [1.80] callig15}{}
\DeclareMathAlphabet{\mathcalligra}{T1}{calligra}{m}{n}
\DeclareFontFamily{OT1}{pzc}{}
\DeclareFontShape{OT1}{pzc}{m}{it}{<-> s * [1.30] pzcmi7t}{}
\DeclareMathAlphabet{\mathpzc}{OT1}{pzc}{m}{it}
\DeclareMathAlphabet{\pazocal}{OMS}{zplm}{m}{n}

\usepackage[english]{babel}
\usepackage[utf8x]{inputenc}
\usepackage{enumerate}
\usepackage{amsthm}
\newtheorem{theorem}{Theorem}[section]
\newtheorem{lemma}[theorem]{Lemma}
\newtheorem{corollary}[theorem]{Corollary}

\newtheorem{proposition}[theorem]{Proposition}

\newtheorem*{theorem*}{Theorem}
\newtheorem*{lemma*}{Lemma}

\theoremstyle{definition}
\newtheorem{definition}{Definition}
\newtheorem{example}{Example}
\newtheorem{problem}{Problem}

\usepackage{color}
\usepackage[colorinlistoftodos]{todonotes}

\usepackage{xfrac}    
\usepackage{faktor}

\usepackage{tikz}
\usetikzlibrary{calc}
\tikzset{point/.style = {fill=gray,circle,inner sep=2pt}}

\renewcommand\AA{{\mathbb A}}
\newcommand\BB{{\mathbb B}}

\newcommand\VV{{\mathbb V}}

\newcommand{\Aut}{\operatorname{Aut}\nolimits}

\newcommand{\End}{\operatorname{End}\nolimits}

\newcommand{\id}{\operatorname{id}\nolimits}

\newcommand{\Mat}{\operatorname{Mat}\nolimits}

\newcommand\Spec{\operatorname{Spec}\nolimits}

\newcommand{\tr}{\operatorname{tr}\nolimits}

\renewcommand\mod{{\,\mathrm{mod}\,}}

\renewcommand\tilde[1]{\widetilde{#1}}
\renewcommand\bar[1]{\overline{#1}}

\def\[[{[\![}
\def\]]{]\!]}
\def\(({(\!(}
\def\)){)\!)}

\newlength{\rrrr}

\newcommand{\intoo}[1]{\:
\xymatrix@1{\ar@{^(->}[r]^{#1}&}\:}

\newcommand{\ootni}[1]{\:
\xymatrix@1{&\ar@{_(->}[l]_(.3){#1}}\:}

\newcommand{\loeqs}[1]{\:
\xymatrix@1{\ar@{=}[rr]^{\ #1\ }&&}\:}

\def\geq{\geqslant}
\def\leq{\leqslant}

\usepackage{amsthm}
\newtheorem*{main-theo}{Main Theorem}

\theoremstyle{remark}
\newtheorem{remark}[theorem]{Remark}

\numberwithin{equation}{section}

\newcommand{\Z}{\ensuremath{\mathbb{Z}}}
\newcommand{\kk}{\ensuremath{\Bbbk}}
\newcommand{\Hom}{\operatorname{Hom}}
\newcommand{\Der}{\operatorname{Der}}
\newcommand{\Id}{\operatorname{Id}}
\newcommand{\Gl}{\operatorname{GL}}
\newcommand{\Rep}{\operatorname{Rep}}
\newcommand{\res}{\operatorname{res}}
\newcommand{\Vect}{\operatorname{Vect}}
\newcommand{\gl}{\ensuremath{\mathfrak{gl}}}
\newcommand{\GG}{\ensuremath{\mathbf{G}}}
\newcommand{\HH}{\ensuremath{\mathbf{H}}}
\newcommand{\g}{\ensuremath{\mathfrak{g}}}
\newcommand{\h}{\ensuremath{\mathfrak{h}}} 
\newcommand{\del}{\ensuremath{\partial}}
\newcommand{\Jinf}{\ensuremath{\mathcal{J}_\infty}}

\newcommand{\red}{\ensuremath{\mathrm{red}}}
\newcommand\br[1]{\{ #1 \}}
\newcommand\dgal[1]{  \left\{\!\!\left\{#1\right\}\!\!\right\} }

\newcommand{\PA}{\ensuremath{\mathtt{PA}}}
\newcommand{\PVA}{\ensuremath{\mathtt{PVA}}}
\newcommand{\DPA}{\ensuremath{\mathtt{DPA}}}
\newcommand{\DPVA}{\ensuremath{\mathtt{DPVA}}}
\newcommand{\HoP}{\ensuremath{\mathtt{H}_0\mathtt{P}}}
\newcommand{\HoPV}{\ensuremath{\mathtt{H}_0\mathtt{PV}}}

\def\restriction#1#2{\mathchoice
              {\setbox1\hbox{${\displaystyle #1}_{\scriptstyle #2}$}
              \restrictionaux{#1}{#2}}
              {\setbox1\hbox{${\textstyle #1}_{\scriptstyle #2}$}
              \restrictionaux{#1}{#2}}
              {\setbox1\hbox{${\scriptstyle #1}_{\scriptscriptstyle #2}$}
              \restrictionaux{#1}{#2}}
              {\setbox1\hbox{${\scriptscriptstyle #1}_{\scriptscriptstyle #2}$}
              \restrictionaux{#1}{#2}}}
\def\restrictionaux#1#2{{#1\,\smash{\vrule height .8\ht1 depth .85\dp1}}_{\,#2}} 

\title[Functorial constructions related to double Poisson vertex algebras]{Functorial constructions related to \\ double Poisson vertex algebras}

\author{Tristan Bozec}
\author{Maxime Fairon}
\author{Anne Moreau}

\address{Univ Angers, CNRS, LAREMA, SFR MATHSTIC, F-49000 Angers, France}
\email{tristan.bozec@univ-angers.fr}

\address{Université Bourgogne Europe, CNRS, IMB UMR 5584, F-21000 Dijon, France}
\email{maxime.fairon@u-bourgogne.fr}

\address{Universit\'{e} Paris-Saclay, CNRS, Laboratoire de Math\'{e}matiques d'Orsay, 
Rue Michel Magat, B\^{a}t. 307, 
91405 Orsay, France}
\email{anne.moreau@universite-paris-saclay.fr}

\begin{document}

\begin{abstract} 
For any double Poisson algebra, we produce a double Poisson vertex algebra 
using the jet algebra construction. We show that this construction is compatible 
with the representation functor which associates to any double Poisson (vertex) 
algebra and any positive integer a Poisson (vertex) algebra. 
We also consider related constructions, such as Poisson reductions and 
Hamiltonian reductions, with the aim of comparing the different corresponding categories. 
This allows us to provide various interesting examples of double Poisson vertex algebras, 
in particular from double quivers. 
\end{abstract}

\maketitle

\selectlanguage{english}

\tableofcontents

\section{Introduction}
This article is based on the jet construction for double Poisson algebras, with the aim 
of comparing different categories: the category of {\em Poisson algebras}, 
of {\em double Poisson algebras}, of {\em Poisson vertex algebras} and 
of {\em double Poisson vertex algebras}. 
This allows us to provide various interesting examples of double Poisson vertex algebras.

\subsection{}
Let $\AA$ denote an associative unital algebra of finite type over a field $\kk$ of characteristic $0$.  
It was realised by Van den Bergh \cite{VdB} that if we endow $\AA$ 
with a \emph{double Poisson bracket}, that is a $\kk$-linear map 
$$\dgal{-,-} \colon \AA \otimes \AA \longrightarrow \AA\otimes \AA\,, \quad a \otimes b \longmapsto \dgal{a, b}\,,$$
satisfying some rules of skewsymmetry, derivation and Jacobi identity (see \eqref{Eq:DA}, \eqref{Eq:Dl}, \eqref{Eq:Dr} and \eqref{Eq:DJ}), then the pair $(\AA,\dgal{-,-})$ induces a structure of Poisson algebra on the $N$-th representation algebra $\AA_N:=\kk[\Rep(\AA,N)]$, where $N\in \Z_{\geq0}$ (cf.~\S\ref{sss:RepAlg} for the definition). 
Slightly rephrasing Van den Bergh's result, we see that each representation functor $(-)_N \colon\AA \mapsto \AA_N:=\kk[\Rep(\AA,N)]$ from associative algebras to commutative algebras restricts to a functor between the category of double Poisson algebras $\DPA$ and that of Poisson algebras~$\PA$.  
We should emphasize that double brackets are compatible with Hamiltonian reduction 
(see \S\ref{ss:Intro-HRed}), which makes them particularly convenient to study the Poisson structure of families of moduli spaces of representations. This is the case of quiver varieties, and it played a crucial role 
 in the study of new integrable systems of Calogero-Moser type (see~\cite{F22,FG}), 
or in Van den Bergh's construction of a Poisson structure on multiplicative quiver varieties\footnote{An experienced reader should remark that one needs to use double \emph{quasi-}Poisson brackets in the multiplicative case, which are not considered in the present work, cf. \S\ref{ss:qPVA}.} \cite{CBS}.

Let us now assume that $(\VV,\del)$ is a differential associative unital algebra over $\kk$. 
Note that the derivation $\del\in \Der(\VV)$ naturally induces a derivation $\del$ on $\VV_N$. 
Thus $(\VV_N,\del)$ is a commutative (associative unital) differential algebra, that is to say, 
a {\em commutative vertex algebra} (\cite{Bo}). 
Building on the previous situation, De Sole, Kac and Valeri \cite{DSKV} introduced the notion of \emph{double Poisson vertex algebras} which are based on a \emph{double $\lambda$-bracket}, that is a $\kk$-linear map 
$$\dgal{-_\lambda-} \colon  \VV \otimes \VV \longrightarrow (\VV\otimes \VV)[\lambda]\,, \quad a \otimes b \longmapsto \dgal{a_\lambda b}\,,$$
compatible with $\del$ and satisfying properties similar to the double Poisson case (see 
Definition~\ref{def:DPVA}). Again, the definition is chosen so that each representation functor $(-)_N$ restricts to a functor between the category of double Poisson vertex algebras $\DPVA$ and that of Poisson vertex algebras $\PVA$. 

In a nutshell, the present paper aims at comparing the two functors $\DPA\to \PA$ and $\DPVA \to \PVA$. 
We shall subsequently describe several constructions which are standard in $\DPA$, $\DPVA$, $\PA$ or $\PVA$ (or some of these categories) but are yet lacking analogues in the other categories. 
Thus, our goal is to extend the theory of (double) Poisson vertex algebras from that of (double) Poisson algebras, and vice-versa, by adopting several \emph{a priori} separate point of views which are well-known to experts from some communities, but are unheard of in the other fields.

\subsection{}

There are two standard approaches to Poisson vertex algebras. 
With the first approach, given a commutative differential algebra $(V,\del)$, we encode the extra structure in terms of a $\lambda$-bracket 
$$\br{-_\lambda-} \colon  V \otimes V \to V[\lambda]\,, \quad a \otimes b \mapsto \br{a_\lambda b}\,,$$
where $\lambda$ is a formal variable and $V[\lambda]:=V\otimes \kk[\lambda]$. The $\lambda$-bracket is a derivation in its second argument, which is required to satisfy rules analogous to that of a Lie algebra; the main difference is that the variable $\lambda$ serves to have an operation compatible with the differential~$\del$, cf.~Definition \ref{Def:PVA-br}.  
That point of view is, for example, particularly suited to the study of integrable PDEs \cite{BDSK,Ka}. 
The introduction of double Poisson vertex algebras by De Sole, Kac and Valeri \cite{DSKV} follows this idea. As we mentioned earlier, their definition is based on endowing an associative differential algebra with a \emph{double} $\lambda$-bracket, and this theory allows to study non-abelian integrable PDEs efficiently. More recently, this was also used to upgrade the one-to-one correspondence between some (graded) Poisson vertex algebras and  Courant-Dorfman algebras to the associative setting \cite{AFH}.

With the second approach, we are given a family of `products', which are $\kk$-bilinear maps 
\begin{equation} \label{Eq:I-mult}
    -_{(n)}-\colon V\times V \to V\,, \qquad n\in \Z_{\geq0},
\end{equation}
satisfying compatibility conditions between themselves and with the differential, see Definition \ref{Def:PVA-end}. 
The advantage of these products is their appearance within the theory of vertex algebras. 

It is a standard result that we can go back and forth between these two definitions of Poisson vertex algebra,  as we recall in Proposition \ref{Pr:CorPVA}. 
One can therefore wonder if such a correspondence exists at the associative level. 
This requires the introduction of an alternative definition of double Poisson vertex algebras based on `double products'  
\begin{equation} \label{Eq:I-Dmult}
    -_{(\!(n)\!)}-\colon \VV\times \VV \to \VV\otimes \VV\,, \qquad n\in \Z_{\geq0}.
\end{equation}
A first result of the present paper is to provide the alternative Definition \ref{def:DPVA-b}, and prove its equivalence with the definition of \cite{DSKV}, cf. Proposition \ref{Pr:CorDPVA}. 
As opposed to the commutative setting, this is not entirely straightforward. 
We shall also explain that both definitions of a double Poisson vertex algebras induce the analogous (commutative) structures of Poisson vertex algebras after passing to representation algebras under each functor $(-)_N$, $N\geq 1$.  
In particular, we recover the usual equivalence between the two definitions in the commutative setting from that in the associative setting. 
%as depicted in Figure \ref{Fig:A}. 

Another original motivation for the present work was to bridge the gap between double Poisson vertex algebras and Van den Bergh's original notion of double Poisson algebras \cite{VdB}. 
If we glance at the commutative theory, we can observe that 
\begin{itemize}
    \item a Poisson vertex algebra leads to a Poisson algebra by quotienting with respect to the ideal generated by the image of the differential;
    \item a Poisson vertex algebra can be defined from the jet algebra of a Poisson algebra~\cite{A12}.
\end{itemize}
These observations are at the basis of functors $\mathtt{Q}\colon\PVA\to \PA$ and $\mathtt{J}\colon\PA\to \PVA$, respectively. 
We shall construct the analogues of these functors between the categories $\mathtt{DP}(\mathtt{V})\mathtt{A}$ of double Poisson (vertex) algebras in Section \ref{Sec:JQ}. 
Namely, we associate to each associative algebra $\AA$ its {\em jet algebra} $\Jinf \AA$ 
%(cf.~\S\ref{ss:JetAlg}) 
and show that, if $\AA$ is a double Poisson algebra, then 
$\Jinf \AA$ has a unique double Poisson vertex structure extending the double Poisson structure on $\AA$ 
(Lemma~\ref{Lem:DPAtoDPVA}). 
It will be shown that these functors are compatible with their commutative versions when one goes to representation algebras.

\begin{theorem}[see Theorem \ref{Thm:QJ-Rep}]  
The following diagrams commute: 
\begin{center}
\begin{tikzpicture}
 \node  (TopLeft) at (-1,1) {$\DPVA$};
 \node  (TopRight) at (1,1) {$\DPA$};
 \node  (BotLeft) at (-1,-1) {$\PVA$};
 \node  (BotRight) at (1,-1) {$\PA$};
\path[->,>=angle 90,font=\small]  
   (TopLeft) edge node[above] {$\mathtt{Q}$}  (TopRight) ;
   \path[->,>=angle 90,font=\small]  
   (BotLeft) edge node[above] {$\mathtt{Q}$}  (BotRight) ;
\path[->,>=angle 90,font=\small]  
   (TopLeft) edge node[left] {$(-)_N$}  (BotLeft) ;
\path[->,>=angle 90,font=\small]  
   (TopRight) edge node[right] {$(-)_N$}  (BotRight) ;
   \end{tikzpicture}
   \qquad
   \begin{tikzpicture}
   \node at (0,0) {and};
   \node at (0,-1.2) {~};
      \end{tikzpicture}
\qquad
   \begin{tikzpicture}
 \node  (TopLeft) at (-1,1) {$\DPVA$};
 \node  (TopRight) at (1,1) {$\DPA$};
 \node  (BotLeft) at (-1,-1) {$\PVA$};
 \node  (BotRight) at (1,-1) {$\PA$};
\path[->,>=angle 90,font=\small]  
   (TopRight) edge node[above] {$\mathtt{J}$}  (TopLeft) ;
   \path[->,>=angle 90,font=\small]  
   (BotRight) edge node[above] {$\mathtt{J}$}  (BotLeft) ;
\path[->,>=angle 90,font=\small]  
   (TopLeft) edge node[left] {$(-)_N$}  (BotLeft) ;
\path[->,>=angle 90,font=\small]  
   (TopRight) edge node[right] {$(-)_N$}  (BotRight) ;
   \end{tikzpicture}
\end{center}
\end{theorem}

\subsection{}

We assume that $\kk=\overline{\kk}$ in addition to $\textrm{char}(\kk)=0$. 
Let $A$ be a commutative algebra of finite type over $\kk$ and $\GG$ be an algebraic group acting on $A$. 
Assume that $A$ is equipped with a Poisson bracket and that each element  $g\in \GG$ acts on $A$ through a Poisson automorphism. This guarantees that the subalgebra $A^\GG \subset A$ of $\GG$-invariant elements inherits the Poisson bracket from $A$; we call $A^\GG$ the \emph{Poisson reduction} of $A$ by $\GG$. 

Fix $N\geq 1$.  
The $N$-th representation algebra $\AA_N$ of a double Poisson algebra $\AA$ admits an action of the algebraic group $\Gl_N$ by Poisson automorphism (for the induced Poisson structure given by Theorem \ref{Thm:RepdP}). Thus, we can form the Poisson reduction $\AA_N^{\Gl_N}$. 
Interestingly, Van den Bergh has also explained that the Poisson bracket on $\AA_N^{\Gl_N}$ is completely determined by the data of a Lie bracket on the vector space $H_0(\AA):=\AA/[\AA,\AA]$ with lifts to derivations on $\AA$; these data form a \emph{$H_0$-Poisson structure} according to Crawley-Boevey's terminology \cite{CB}, which is induced by the double Poisson bracket on $\AA$. 
One can observe that everything can be assembled functorially in order to form a commutative square  
\begin{center}
    \begin{tikzpicture}
 \node  (TopLeft) at (-1,1) {$\DPA$};
 \node  (TopRight) at (1,1) {$\PA^{\Gl_N}$};
 \node  (BotLeft) at (-1,-1) {$\HoP$};
 \node  (BotRight) at (1,-1) {$\PA$};
\path[->,>=angle 90,font=\small]  
   (TopLeft) edge (TopRight) ;
   \path[->,>=angle 90,font=\small]  
   (BotLeft) edge (BotRight) ;
\path[->,>=angle 90,font=\small]  
   (TopLeft) edge  (BotLeft) ;
\path[->,>=angle 90,font=\small]  
   (TopRight) edge  (BotRight) ;
   \end{tikzpicture}
\end{center} 
between the categories of double Poisson algebras $\DPA$, $H_0$-Poisson structures $\HoP$ and Poisson algebras $\PA$ (with a $\Gl_N$ action by Poisson automorphisms $\PA^{\Gl_N}$), see Proposition \ref{Pr:ComFront}. 

In view of our previous motivation, we can ask whether this square can be extended to a cube by adding a `vertex direction', i.e. by adding analogues of the jet functors $\mathtt{J} \colon \PA\to \PVA$ and $\mathtt{J}\colon\DPA\to \DPVA$. Applying this idea, we obtain the jet algebras $\Jinf(\AA_N)$ and $\Jinf(\AA_N^{\Gl_N})$ as the new vertices at the top and bottom of the right side of the (not-yet-defined) cube. 
This is where technicalities appear: to relate these two algebras we should consider the $\Jinf(\Gl_N)$-action on $\Jinf(\AA_N)$ induced by the $\Gl_N$-action on $\AA_N$, but if we consider invariants we only have a natural 
differential algebra 
morphism $\Jinf(\AA_N^{\Gl_N})\to \Jinf(\AA_N)^{\Jinf(\Gl_N)}$, which is not an isomorphism in general. 
Moreover, $\Jinf(\Gl_N)$ does not act on $\Jinf(\AA_N)$ by Poisson vertex automorphisms.  

However, we have the following general statement, of independent interest.

\begin{theorem}[Theorem \ref{Th:invariants_are_PVA}]
Let $\GG$ be an affine algebraic group 
acting on a Poisson algebra $A$ by Poisson automorphism. 
The invariant algebra $(\Jinf A)^{\Jinf \GG}$ 
is a Poisson vertex subalgebra of $\Jinf A$, 
and the natural morphism 
$$j_A \colon \Jinf (A^{\GG}) \longrightarrow (\Jinf A)^{\Jinf \GG}$$  
induced by the inclusion $A^{\GG} \hookrightarrow A$ 
is a  Poisson vertex algebra morphism. 
\end{theorem}

We are led to introduce the new notion of $H_0$-Poisson vertex structures in \S\ref{ss:HOPconf}, which are analogues of Crawley-Boevey's $H_0$-Poisson structures \cite{CB} in the presence of a differential. 
We obtain the following result. 

\begin{theorem}[Theorem \ref{Thm:ComPoissonRed}]
\label{Th:main_cube_commute}
Fix $N\geq 1$.  
The cube in Figure~\ref{Fig:PoiRed} is commutative. 
The precise categories in which each object at a node belongs, 
and the precise functors between them, are specified in Section \ref{S:PoiRed}.

\medskip

{\tiny
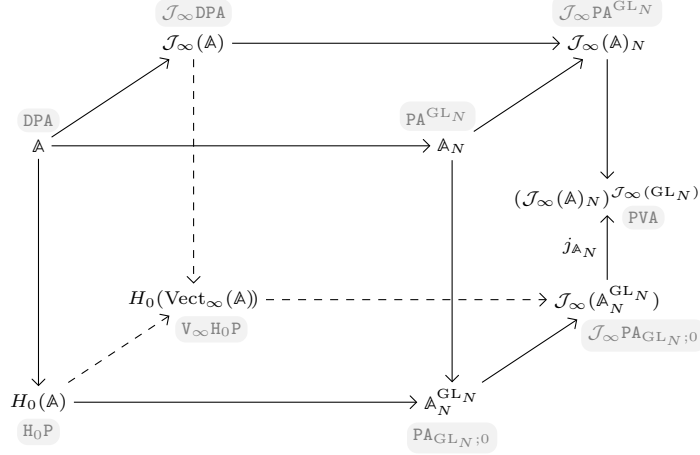
\begin{figure}[!h]
\centering
\begin{tikzpicture}[scale=.68]
 \node  (TopLeft) at (-5,2) {$\AA$};
 \node  (TopRight) at (3,2) {$\AA_{N}$};
 \node  (BotLeft) at (-5,-3) {$H_0(\AA)$}; 
 \node  (BotRight) at (3,-3) {$\AA_N^{\Gl_N}$};
% back
 \node  (BackTopLeft) at (-2,4) {$\Jinf(\AA)$};
 \node  (BackTopRight) at (6,4) {$\Jinf(\AA)_N$};
 \node   (BackBotLeft) at (-2,-1) {$H_0(\Vect_\infty(\AA)\!)$}; 
 \node  (JVJG) at (6,1) {$(\Jinf(\AA)_N)^{\Jinf(\Gl_N)}$};
 \node  (BackBotRight) at (6,-1) {$\Jinf(\AA_N^{\Gl_N})$};
%%%% categories 
\node[gray,fill=gray!10,rounded corners, above=0cm of TopLeft] (CatTL) {$\DPA$};
\node[gray,fill=gray!10,rounded corners, above left=0cm and -0.6cm of TopRight] (CatTR) {$\PA^{\Gl_N}$};
\node[gray,fill=gray!10,rounded corners, below=-0cm of BotLeft] (CatBL) {$\HoP$};
\node[gray,fill=gray!10,rounded corners, below=-0cm of BotRight] (CatBR) {$\PA_{\Gl_N;0}$};
% back
\node[gray,fill=gray!10,rounded corners, above=0cm of BackTopLeft] (CatBTL) {$\Jinf\DPA$};
\node[gray,fill=gray!10,rounded corners, above=0cm of BackTopRight] (CatBTR) {$\Jinf\PA^{\Gl_N}$};
\node[gray,fill=gray!10,rounded corners, below right=-0cm and -1.2cm of BackBotLeft] (CatBBL) {$\mathtt{V}_\infty\HoP$};
\node[gray,fill=gray!10,rounded corners] (CatBBR) at (6.7,-1.7) {$\Jinf \PA_{\Gl_N;0}$};
\node[gray,fill=gray!10,rounded corners] (CatJVJG) at (6.7,0.6)  {$\PVA$};
% front arrows
\path[->,>=angle 90,font=\small]  
   (TopLeft) edge (TopRight) ;
   \path[->,>=angle 90,font=\small]  
   (BotLeft) edge (BotRight) ;
\path[->,>=angle 90,font=\small]  
   (TopLeft) edge (BotLeft) ;
\path[->,>=angle 90,font=\small]  
   (TopRight) edge (BotRight) ;
% back arrows
\path[->,>=angle 90,font=\small]  
   (BackTopLeft) edge (BackTopRight) ;
   \path[->,>=angle 90,dashed,font=\small]  
   (BackBotLeft) edge (BackBotRight) ;
\path[->,>=angle 90,dashed,font=\small]  
   (BackTopLeft) edge (BackBotLeft) ;
\path[->,>=angle 90,font=\small]  
   (BackTopRight) edge (JVJG) ;
\path[->,>=angle 90,font=\small]  
   (BackBotRight) edge node[left] {{\tiny $j_{\AA_N}$}} (JVJG) ;
% back to front arrows
\path[<-,>=angle 90,font=\small]  
   (BackTopLeft) edge (TopLeft) ;
   \path[<-,>=angle 90,dashed,font=\small]  
   (BackBotLeft) edge (BotLeft) ;
\path[<-,>=angle 90,font=\small]  
   (BackTopRight) edge (TopRight) ;
\path[<-,>=angle 90,font=\small]  
   (BackBotRight) edge (BotRight) ;
   \end{tikzpicture}
   \caption{We indicate next to a node (on a light gray background) to which category an object belongs.}  
    \label{Fig:PoiRed}
\end{figure}
}
\end{theorem}

\subsection{Example} 
\label{ss:double_quiver} 
Let $Q$ be a (finite) quiver with vertex set $S$.
The double quiver
$\bar{Q}$ is obtained by adjoining a new opposite arrow $a^\ast$ for each arrow $a$ originally in $Q$.
Let us consider the path algebra $\kk \bar{Q}$ of $\bar{Q}$ (cf. Remark \ref{Rem:ConvQuiver} for our convention).
The subalgebra $B=\oplus_{s\in S}\kk e_s$ of $\kk \bar{Q}$ formed by trivial paths may be identified
with $\kk S$. 
We get an involution $(-)^\ast\colon \bar{Q}\to \bar{Q}$ by setting $(a^\ast)^\ast=a$ if $a\in Q$. 
There is also a map $\epsilon\colon \bar{Q}\to \{\pm 1\}$ given by $\epsilon(a)=+1$ if $a\in Q$ and $\epsilon(a)=-1$ if $a\in \bar{Q}\setminus Q$.  
There is a unique $B$-linear double Poisson bracket on $\kk \bar{Q}$  satisfying (\cite{VdB}): 
\begin{equation} \label{Eq:dbr-quiver}
   \dgal{a,b}= \left\{
\begin{array}{ll}
\epsilon(a) \, e_{h(a)}\otimes e_{t(a)}  &\text{ if } b=a^\ast  \\
0     & \text{ otherwise,}
\end{array}
   \right.
\end{equation}
for each arrows $a,b \in \bar{Q}$. Here, $t,h\colon\bar{Q}\to S$ denote the tail and head maps which satisfy $t(a^\ast)=h(a)$ with respect to the involution $(-)^\ast$. 

Given a quiver $Q$, we construct $Q_\infty$ as the quiver whose vertex set is the same as $Q$, while its arrow
set is $Q_\infty=\{a^{(\ell)} \mid a\in Q,\, \ell \in \Z_{\geq0} \}$ with
$t(a^{(\ell)})=t(a)$ and $h(a^{(\ell)})=h(a)$ for each $a$ and $\ell$.
Then $\Jinf(\kk Q)$ is isomorphic to $\kk Q_\infty$ as a differential algebra
equipped with $\del\colon a^{(\ell)}\mapsto a^{(\ell+1)}$, $e_s\mapsto 0$. 
This construction is compatible with going to the doubles $\bar{Q}$, $\bar{Q}_\infty$.

Path algebras of double quivers are particular examples of {\em noncommutative cotangent algebras} introduced by Crawley-Boevey, Etingof and Ginzburg: 
for any smooth $B$-algebra $\AA$, it is defined 
to be $T^*\AA := T_\AA (\Der_B \AA)$, the tensor algebra of 
the $\AA$-bimodule of $B$-linear double derivations $\Der_B \AA$.  
Its $N$-th representation algebra is 
the coordinate ring of the cotangent bundle of $\Rep(\AA,N)$.
Path algebras associated to quivers are
typical examples of smooth algebras, and we have 
$\kk \bar{Q} = T^* \AA$ where $\AA$ is the path algebra $\kk Q$ of the original quiver. 
Applying Hamiltonian reduction to noncommutative cotangent algebras gives
an interesting class of associative algebras that includes preprojective
algebras associated with quivers. 
We refer to \cite{CBEG} for more details about this topic.

Motivated by this family of examples, we consider {\em semisimple versions} 
of our results. A double Poisson algebra {\em relative to a $\kk$-algebra $B$} 
is an algebra $\AA$ which contains $B$ as subalgebra equipped with a double Poisson 
bracket such that $\dgal{a,B}=0$ identically for any $a \in \AA$. 
Then all the previous statements can be adapted to this setting 
(cf.~Theorem \ref{Thm:rel-ComPoissonRed}), 
where the integer $N$ has to be replaced by $\underline{n}=(n_s)\in \Z_{\geq0}^S$ 
and $\Gl_N$ by $\Gl_{\underline{n}}:=\prod_{s\in S}\Gl_{n_s}$;
see Section~\ref{Sec:SemiS}. 
We refer to \S\ref{ss:SemiS-example} for a slight generalisation of the motivating example of a double quiver,
where we compute explicitly the (Lie, Poisson, Poisson vertex, double Poisson, $\ldots$)
bracket defined at each node of the cube in Figure~\ref{Fig:PoiRed} attached to 
the double Poisson algebra $\kk Q_{p,q}$
(which is $\kk\bar{Q}_{p,0}$ if $p=q$).

\subsection{} \label{ss:Intro-HRed}
We are also willing to adapt the discussion to the case of \emph{Hamiltonian reduction}. In that situation, we assume that $\GG$ is an algebraic group, and we consider the subcategory of $\PA^\GG$ whose objects admit a \emph{moment map} (which is respected by morphisms). Such a map $\tilde{\mu} \colon \kk[\g^\ast]\to A$ encodes the infinitesimal action on $A$ of the Lie algebra $\g$ of $\GG$ (with dual $\g^\ast$). 
After fixing an element $\xi\in \g^\ast$ invariant under the coadjoint action by $\GG$, we can define  the Hamiltonian reduction $A_{\red;\xi}$ of $A$ at $\xi$ by considering $\GG$-invariant elements of  $A/I_{\xi}$, where $I_{\xi}\subset A$ is the ideal generated by requiring that $\tilde{\mu}$ equals the map $ev_{\xi}\colon\kk[\g^\ast]\to \kk$ evaluating at $\xi$, see 
Section~\ref{Sec:HamRed} for a precise definition. 
Again, Van den Bergh remarked that we can understand moment maps on double Poisson algebras, which leads to a commutative square 
\begin{center}
    \begin{tikzpicture}
 \node  (TopLeft) at (-1,1) {$\DPA_\mu$};
 \node  (TopRight) at (1,1) {$\PA^{\Gl_{\underline{n}}}_\mu$};
 \node  (BotLeft) at (-1,-1) {$\HoP$};
 \node  (BotRight) at (1,-1) {$\PA$};
\path[->,>=angle 90,font=\small]  
   (TopLeft) edge (TopRight) ;
   \path[->,>=angle 90,font=\small]  
   (BotLeft) edge (BotRight) ;
\path[->,>=angle 90,font=\small]  
   (TopLeft) edge  (BotLeft) ;
\path[->,>=angle 90,font=\small]  
   (TopRight) edge  (BotRight) ;
   \end{tikzpicture}
\end{center} 
involving this time the categories of double Poisson algebras with a (noncommutative) moment map $\DPA_\mu$ and Poisson algebras with a $\Gl_{\underline{n}}$ action by Poisson automorphisms and a moment map $\PA^{\Gl_{\underline{n}}}_\mu$, cf.~\S\ref{ss:HamRedDPA}. 
For example, $\mu=\sum_{a\in \bar{Q}}\epsilon(a) aa^\ast$ is a noncommutative moment map 
for the path algebra $\kk \bar{Q}$ as in \S\ref{ss:double_quiver}, and the above commutative square yields an affine quiver variety associated with $Q$ as the object in $\PA$.  
Let us point out that we can not simply rely on the cube obtained in the Poisson reduction picture (cf. Figure \ref{Fig:PoiRed}) because we have to set up the general theory of noncommutative moment maps in the vertex case.  

In this setting, we obtain the following result.

\begin{theorem} [Theorem \ref{Th:cube_Hamiltonian}]
The cube in Figure~\ref{Fig:HamRed} is commutative. 
The precise categories in which each object at a node belongs, 
and the precise functors between them, are specified in Sections \ref{Sec:HamRed} 
and \ref{Sec:Momap}. 

{\tiny
\begin{figure}[!h]
\begin{center}
\begin{tikzpicture} [scale=.8]
 \node  (TopLeft) at (-4,2) {$(\AA,\bbmu)$};
 \node  (TopRight) at (3,2) {$(\AA_{\underline n},X(\bbmu))$};
 \node  (BotLeft) at (-4,-2) {$H_0(\mathbb{P})$}; 
 \node  (BotRight) at (3,-2) {$\mathbb{P}_{\underline n}^{\Gl_{\underline n}}$};
% back
 \node  (BackTopLeft) at (-1,4) {$(\Jinf(\AA),\bbmu)$};
 \node  (BackTopRight) at (6,4) {$(\Jinf(\AA)_{\underline n},X(\bbmu))$};
 \node   (BackBotLeft) at (-1,0) {$H_0(\Vect_\infty(\mathbb{P}))$}; 
 \node  (BackBotRight) at (6,0) {$\Jinf\big(\mathbb{P}_{\underline n}^{\Gl_{\underline n}}\big)$};
 \node (BackMidRight) at (6,2) {$\Jinf(\mathbb{P}_{\underline n})^{\Jinf(\Gl_{\underline n})}$};
% front arrows
\path[->,>=angle 90,font=\small]  
   (TopLeft) edge (TopRight) ;
   \path[->,>=angle 90,font=\small]  
   (BotLeft) edge (BotRight) ;
\path[->,>=angle 90,font=\small]  
   (TopLeft) edge (BotLeft) ;
\path[->,>=angle 90,font=\small]  
   (TopRight) edge (BotRight) ;
% back arrows
\path[->,>=angle 90,font=\small]  
   (BackTopLeft) edge (BackTopRight) ;
   \path[->,>=angle 90,dashed,font=\small]  
   (BackBotLeft) edge (BackBotRight) ;
\path[->,>=angle 90,dashed,font=\small]  
   (BackTopLeft) edge (BackBotLeft) ;
\path[->,>=angle 90,font=\small]  
   (BackTopRight) edge (BackMidRight) ;
   \path[->,>=angle 90,font=\small]  
   (BackBotRight) edge (BackMidRight) ;
% back to front arrows
\path[<-,>=angle 90,font=\small]  
   (BackTopLeft) edge (TopLeft) ;
   \path[<-,>=angle 90,dashed,font=\small]  
   (BackBotLeft) edge (BotLeft) ;
\path[<-,>=angle 90,font=\small]  
   (BackTopRight) edge (TopRight) ;
\path[<-,>=angle 90,font=\small]  
   (BackBotRight) edge (BotRight) ;
%%%%% categories 
\node[gray,fill=gray!10,rounded corners, above left=0cm and -0.8cm of TopLeft] (CatTL) {$\DPA_\mu$};
\node[gray,fill=gray!10,rounded corners, above left=0cm and -1cm of TopRight] (CatTR) {$\PA_\mu^{\Gl_{\underline n}}$};
\node[gray,fill=gray!10,rounded corners, below=-0cm of BotLeft] (CatBL) {$\HoP$};
\node[gray,fill=gray!10,rounded corners, below=-0cm of BotRight] (CatBR) {$\PA_{\Gl_{\underline n};0}$};
% back
\node[gray,fill=gray!10,rounded corners, above=0cm of BackTopLeft] (CatBTL) {$\Jinf\DPA_\mu$};
\node[gray,fill=gray!10,rounded corners, above=0cm of BackTopRight] (CatBTR) {$\Jinf\PA_\mu^{\Gl_{\underline n}}$};
\node[gray,fill=gray!10,rounded corners, below right=-0cm and -1.2cm of BackBotLeft] (CatBBL) {$\mathtt{V}_\infty\HoP$};
\node[gray,fill=gray!10,rounded corners, below right=-0cm and -1.2cm of BackBotRight] (CatBBR) {$\Jinf \PA_{\Gl_{\underline n};0}$};
\node[gray,fill=gray!10,rounded corners, right=-0.1cm of BackMidRight] (CatBBR) {$\PVA$};
   \end{tikzpicture}
\end{center}
\caption{We indicate next to a node (on a light gray background) to which category an object belongs. 
Here, $\mathbb{P}=\AA/\bbmu-\zeta$ and thus $\mathbb{P}_{\underline n}=\AA_{\underline n}/\bbmu_{s,ij}-\zeta_{s,ij}$.}
\label{Fig:HamRed} 
\end{figure}
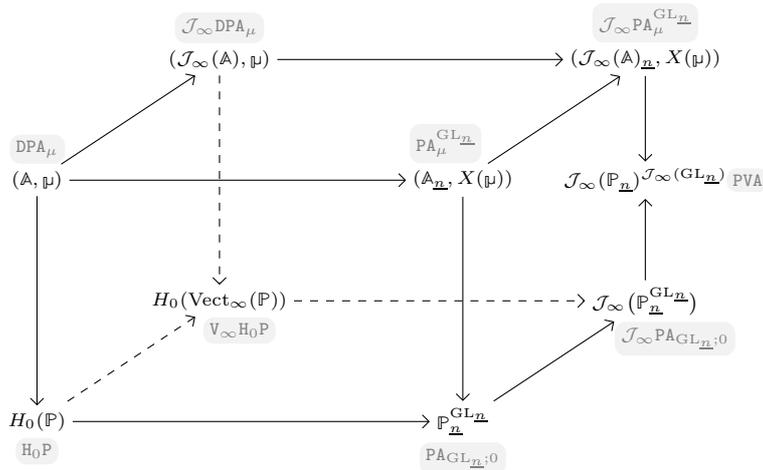}
\end{theorem}

\subsection{}
\label{ss:IntroVA}
The motivation to construct double Poisson vertex algebras also comes from the theory of vertex algebras. 
Any vertex algebra $\mathcal{V}$ produces in a canonical way two interesting Poisson vertex algebras as follows~\cite{Li1,A12}. 
First, the graded algebra ${\rm gr}\,\mathcal{V}$ with respect to the Li 
filtration is a commutative vertex algebra, equipped with a Poisson vertex 
structure. 
One the other hand, the {\em Zhu $C_2$-algebra} $R_{\mathcal{V}}$ 
 is a certain quotient of $\mathcal{V}$ that inherits a Poisson 
structure. So 
its jet algebra has a Poisson vertex algebra structure. 
In fact, there is always a surjective 
Poisson vertex algebra morphism $\Jinf R_{\mathcal{V}} \twoheadrightarrow {\rm gr}\,\mathcal{V}$. 
In the case where the isomorphism 
\begin{align}
\label{eq:arc_iso}
\Jinf R_{\mathcal{V}}\cong {\rm gr}\,\mathcal{V}
\end{align} 
holds, and where the Poisson structure on $R_{\mathcal{V}}$ comes from a double Poisson 
structure, our jet algebra construction gives to $\mathcal{V}$ a richer structure, 
${\rm gr}\,\mathcal{V}$ being the representation algebra  
of some double Poisson vertex algebra.

Let us consider an illustrating example for which the isomorphism \eqref{eq:arc_iso} holds. 
If $\AA$ is a smooth associative algebra, then $\Rep(\AA,N)$ is an affine smooth scheme and  
one can consider the global section $\mathscr{D}_{\Rep(\AA,N)}^{ch}$ 
of the {\em chiral differential operators} 
%in many situations
%\footnote{If $X$ is a smooth 
%scheme, there is a non-trivial obstruction of cohomological nature to the
%construction of $\mathscr{D}_{X}^{ch}$ 
%which can be expressed in terms of a certain homotopy
%Lie algebra.} 
%the vertex algebra of {\em chiral differential operators} 
on it introduced independently by Malikov,
Schechtman and Vaintrob \cite{MSV} and Beilinson and Drinfeld \cite{BD2}. 
This is a vertex algebra which is a natural chiralization 
of the algebra  $\mathcal{D}_{\Rep(\AA,N)}$ of differential operators 
on $\Rep(\AA,N)$ 
in the sense that its {\em Zhu's algebra}  
is $\mathcal{D}_{\Rep(\AA,N)}$ (\cite{ACM}).  
Recall that $\mathcal{D}_{\Rep(\AA,N)}$ is naturally filtered 
by the degree of operators. 
We have the following isomorphisms: 
\begin{align*}
 {\rm gr} \,\mathcal{D}_{\Rep(\AA,N)}  &\cong \kk[T^* \Rep(\AA,N)] = (T^*\AA)_N,& \\ 
{\rm gr}\,\mathscr{D}_{\Rep(\AA,N)}^{ch}  = \Jinf \kk [T^*\Rep(\AA,N)] 
&= \Jinf (T^*\AA)_N, \quad  
R_{\mathscr{D}_{\Rep(\AA,N)}^{ch}} \cong \kk [T^*\Rep(\AA,N)],&
\end{align*}
with the Poisson structure on $R_{\mathscr{D}_{\Rep(\AA,N)}^{ch}}$ 
coming from the symplectic structure on $T^*\Rep(\AA,N)$. 
In particular, the isomorphism \eqref{eq:arc_iso} holds.

For example, 
if $\BB$ is the free algebra $\kk\langle a,b \rangle$ equipped with the unique double Poisson bracket 
such that $\dgal{a,a}=0=\dgal{b,b}$, $\dgal{b,a}=1 \otimes 1$ 
(see Example \ref{Exmp:Symp}), then  
$\BB$ is isomorphic to the path algebra $\kk\bar{Q}$ of  
the double of the quiver 
$\raisebox{-0.375\height}{
\begin{adjustbox}{max totalsize={.055\textwidth}}
\begin{tikzpicture} 
\node(3)at(4,0){$\bullet$};
\draw[-latex](3)edge[out=-30,in=30,looseness=8]node{}(3); 
\end{tikzpicture}
\end{adjustbox}}$
having only one vertex and only one arrow, with 
the double Poisson structure as in \S\ref{ss:double_quiver}. 
In this case $\BB_N \cong \kk[T^* \gl_N ] \cong 
(T^* \kk Q)_N$, with $\gl_N=\Mat_N$ 
the Lie algebra of $\Gl_N$. 
Then $\mathcal{D}_{\Rep(\kk Q,N)}$ is the Weyl 
algebra of rank $N^2$,   
and $\mathscr{D}_{\Rep(\kk Q,N)}^{ch}$ the $\beta\gamma$-system  
generated by the fields $\beta_{i}(z),\gamma_j(z)$, 
for $i,j\in\{1, \ldots,N^2\}$, with OPEs:
\begin{align*}
\beta_i(z) \gamma_j(w) \sim \dfrac{\delta_{i,j}}{z-w}, \quad 
\gamma_j(z) \beta_i(w) \sim -\dfrac{\delta_{i,j}}{z-w}, \quad  
\beta_i(z) \beta_j(w) \sim 0,  \quad  \gamma_i(z) \gamma_j(w) \sim 0.&
\end{align*}

\medskip

{\tiny
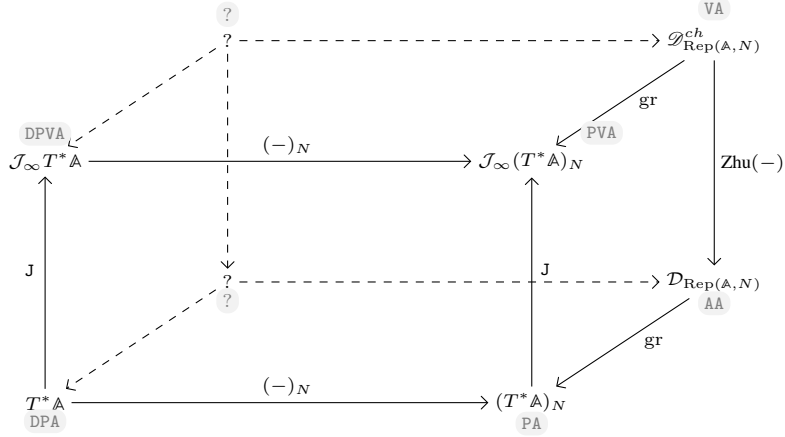
\begin{figure}
\centering
\begin{tikzpicture}[scale=.8]
 \node  (TopLeft) at (-5,2) {$\Jinf T^*\AA$};
 \node  (TopRight) at (3,2) {$\Jinf (T^*\AA)_N$};
 \node  (BotLeft) at (-5,-2) {$T^* \AA$}; 
 \node  (BotRight) at (3,-2) {$(T^*\AA)_N$};
% back
 \node  (BackTopLeft) at (-2,4) {$?$};
 \node  (BackTopRight) at (6,4) {$\mathscr{D}_{\Rep(\AA,N)}^{ch}$};
 \node   (BackBotLeft) at (-2,0) {$?$}; 
 \node  (BackBotRight) at (6,0) {$\mathcal{D}_{\Rep(\AA,N)}$}; 
\path[->,>=angle 90,font=\small]  
   (TopLeft) edge node[above] {\;\;{\tiny $(-)_N$}} (TopRight) ;
   \path[->,>=angle 90,font=\small]  
   (BotLeft) edge node[above] {\;\;{\tiny $(-)_N$}} (BotRight) ;
\path[<-,>=angle 90,font=\small]  
   (TopLeft) edge node[left,above] {\!\!\!\!\!\!\!{\tiny $\mathtt{J}$}} (BotLeft) ;
\path[<-,>=angle 90,font=\small]  
   (TopRight) edge node[right,above] {\quad{\tiny $\mathtt{J}$}} (BotRight) ;
% back arrows
\path[->,>=angle 90,dashed,font=\small]  
   (BackTopLeft) edge (BackTopRight) ;
   \path[->,>=angle 90,dashed,font=\small]  
   (BackBotLeft) edge (BackBotRight) ;
\path[->,>=angle 90,dashed,font=\small]  
   (BackTopLeft) edge (BackBotLeft) ;
\path[->,>=angle 90,font=\small]  
   (BackTopRight) edge node[right] {{\tiny $\text{Zhu}(-)$}} (BackBotRight) ;
% back to front arrows
\path[->,>=angle 90,dashed,font=\small]  
   (BackTopLeft) edge (TopLeft) ;
   \path[->,>=angle 90,dashed,font=\small]  
   (BackBotLeft) edge (BotLeft) ;
\path[->,>=angle 90,font=\small]  
   (BackTopRight) edge node[right] {\,\;{\tiny ${\rm gr}$}} (TopRight) ;
\path[->,>=angle 90,font=\small]  
   (BackBotRight) edge node[right] {\;\;{\tiny ${\rm gr}$}} (BotRight) ;
%%%% categories 
\node[gray,fill=gray!10,rounded corners, above=0cm of TopLeft] (CatTL) {$\DPVA$};
\node[gray,fill=gray!10,rounded corners, above left=0cm and -2cm of TopRight] (CatTR) {$\PVA$};
\node[gray,fill=gray!10,rounded corners, below=-0.1cm of BotLeft] (CatBL) {$\DPA$};
\node[gray,fill=gray!10,rounded corners, below=-0.1cm of BotRight] (CatBR) {$\PA$};
% back
\node[gray,fill=gray!10,rounded corners, above=0cm of BackTopLeft] (CatBTL) {$?$};
\node[gray,fill=gray!10,rounded corners, above=0cm of BackTopRight] (CatBTR) {$\mathtt{VA}$};
\node[gray,fill=gray!10,rounded corners, below=-0.1cm of BackBotLeft] (CatBBL) {$?$};
\node[gray,fill=gray!10,rounded corners, below=-0.1cm of BackBotRight] (CatBBR) {$\mathtt{AA}$};
   \end{tikzpicture}
   \caption{Noncommutative cotangent algebra and relative objects: 
   here $\mathtt{VA}$ is the category of vertex algebras, 
   and $\mathtt{AA}$ that of unital associative filtered algebras.} 
   \label{Fig:cdo}
\end{figure}
}

The diagram depicted in 
Figure \ref{Fig:cdo}, 
where the front and the right cube commute, 
summarizes the discussion. 
A natural, but hard, problem would be to construct {\em double} versions 
of $\mathcal{D}_{\Rep(\AA,N)}$ 
and $\mathscr{D}_{\Rep(\AA,N)}^{ch}$, 
corresponding to $T^*\AA$ and 
$\Jinf (T^*\AA)$, respectively. 
More generally, a similar problem can be raised for vertex 
algebras for which the Poisson structure 
on the Zhu $C_2$-algebras comes from a double Poisson structure 
and where the isomorphism \eqref{eq:arc_iso} holds, see Problem \ref{Pb:vertex}. 
More examples of such vertex algebras are discussed in Section~\ref{Sec:Open}.

\subsection{Organization of the article}
Section \ref{Sec:preliminary} is about the standard notions 
of Poisson, double Poisson, Poisson vertex and double Poisson vertex 
algebras, and the representation functor. 
In Section \ref{S:Alt-PVA}, we give an alternative definition of double Poisson vertex 
algebras. 
We prove in Section \ref{Sec:JQ} the commutativity of the top face of the cube depicted 
in Figure \ref{Fig:PoiRed}. In particular, the jet algebra functor for double Poisson algebras 
is introduced in this section. 
Section \ref{Sec:Inv_PVA} is of independent interest, and deals with the structure 
of invariants of jet algebras in the commutative setting. The main result of this section 
is Theorem \ref{Th:invariants_are_PVA}. 
We prove the commutativity (cf.~Theorem~\ref{Thm:ComPoissonRed}) of the cube depicted 
in Figure \ref{Fig:PoiRed} in Section \ref{S:PoiRed}. 
The categories involved in the figure are introduced in this section. 
Section \ref{Sec:SemiS} is about the semisimple version of Poisson reduction. 
An illustrating example from quivers is presented in \S\ref{ss:SemiS-example}. 
We consider the Hamiltonian reduction in the commutative setting 
(in particular, for Poisson vertex algebras) in Section~\ref{Sec:HamRed}. 
The Hamiltonian reduction in the noncommutative setting is discussed 
in Section~\ref{Sec:Momap}. The commutativity of the cube depicted in Figure
\ref{Fig:HamRed} is proved in this section (cf.~Theorem \ref{Th:cube_Hamiltonian}). 
Finally, problems related with vertex algebras and other topics are 
discussed in Section~\ref{Sec:Open}. 

Appendix \ref{App:cot} gathers calculations about the left/right (noncommutative) moment maps 
on the cotangent bundle $T^\ast\Gl_N\simeq \Rep(\kk\langle a,b^{\pm1}\rangle,N)$ considered in \S\ref{ss:Pb_red}.

\subsection{Acknowledgements}
The authors thank Michael Bulois, Reimundo Heluani, Thibault Juillard, Andrew Linshaw and Daniele Valeri for useful discussions. 
T.B. has received funding from the European Research Council (ERC) under the European Union's Horizon 2020 research and innovation programme (Grant Agreement No. 768679).
The authors wish to thank Damien Calaque, Johan Leray and Bruno Vallette for organizing a workshop on double Poisson structures, funded by the aforementioned grant. 
T.B. is partially supported by the Agence Nationale de la Recherche (project n°\! ANR-24-CE40-2583).
M.F. has received funding from the European Union’s Horizon 2020 research and innovation programme under the Marie 
Sk\l{}odowska-Curie grant agreement No.~101034255. 
A.M. is partially supported by the European Research 
Council (ERC) under the European Union's Horizon 2020 research innovation programme
(ERC-2020-SyG-854361-HyperK) and the ANR project ANR-24-CE40-3389 (GRAW).

%%%%%%%%%%%%%%%%%% NEW SECTION %%%%%%%%%%%%%%%%%%%%%
%%%%%%%%%%%%%%%%%% NEW SECTION %%%%%%%%%%%%%%%%%%%%%
%%%%%%%%%%%%%%%%%% NEW SECTION %%%%%%%%%%%%%%%%%%%%%
%%%%%%%%%%%%%%%%%% NEW SECTION %%%%%%%%%%%%%%%%%%%%%

\section{Preliminary notions}
\label{Sec:preliminary}

We fix a field $\kk$ of characteristic $0$. All unadorned tensor products are taken over $\kk$. Unless otherwise stated, an algebra is a unital associative algebra over $\kk$. 
Below, we recall a collection of conventions that come mainly from \cite{DSKV,VdB}.

\subsection{Basic operations}

%\subsubsection{Homomorphisms on tensor products} \label{ss:HomTP}
Assume that $\AA$ is a vector space over $\kk$. 
Given $n \geq 2$, we form the tensor product $\AA^{\otimes n}$.  
When $n=2$, we use the following Sweedler's notation 
$d= \sum_l d_l'\otimes d_l'' =:d'\otimes d'' \,\in \AA\otimes \AA$ 
to denote elements. 
The permutation endomorphism $(-)^\sigma$ on $\AA\otimes \AA$ is given by  
$(a\otimes b)^\sigma=b\otimes a$.  
More generally for $n\geq 2$, 
we introduce the cyclic permutation 
\begin{equation}
  \AA^{\otimes n}\to \AA^{\otimes n} \colon 
a_1 \otimes \ldots \otimes a_n \mapsto (a_1 \otimes \ldots \otimes a_n)^\sigma:=
a_n \otimes a_1 \otimes \ldots \otimes a_{n-1}\,. 
\end{equation} 

If $\del \in \End(\AA)$, 
we extend $\del$ to an element of $\End(\AA^{\otimes n})$, which by abuse of notation we also denote by $\del$, as follows: 
\begin{equation} \label{Eq:DelExtend}
  \del(a_1 \otimes \ldots \otimes a_n)=\sum_{i=1}^n a_1 \otimes \ldots \otimes a_{i-1} \otimes \del(a_i) \otimes a_{i+1} \otimes \ldots \otimes a_n\,.
\end{equation} 
We also consider the induced linear maps acting as $\del$ on the leftmost and rightmost factors which are respectively denoted by $\del_L$ and $\del_R$: 
\begin{equation} \label{Eq:DelLRExtend}
\begin{aligned}
\del_L(a_1 \otimes \ldots \otimes a_n)&=\del(a_1) \otimes a_2 \otimes  \ldots \otimes a_n\,, \\ 
  \del_R(a_1 \otimes \ldots \otimes a_n)&=a_1 \otimes \ldots \otimes a_{n-1} \otimes \del(a_n)\,.
\end{aligned}
\end{equation}

%\subsubsection{Operations on an algebra}

Assume now that $\AA$ is a (unital associative) algebra over $\kk$.  
Given $n \geq 2$, the tensor product $\AA^{\otimes n}$ is seen as an associative algebra for 
\begin{equation} \label{Eq:prodtens}
    (a_1\otimes \ldots\otimes a_n)(b_1\otimes \ldots \otimes b_n)
    =a_1b_1\otimes \ldots \otimes a_nb_n\,.
\end{equation}

We define left and right $\AA$-module structures on $\AA^{\otimes n}$ as follows. For $0\leq i \leq n-1$, 
\begin{equation}
  \begin{aligned}
    b \ast_i (a_1 \otimes \ldots \otimes a_n)=& a_1 \otimes \ldots a_i \otimes b a_{i+1} \otimes a_{i+2} \otimes \ldots \otimes a_n \,, \\
(a_1 \otimes \ldots \otimes a_n) \ast_i b=& a_1 \otimes \ldots \otimes a_{n-i-1} \otimes  a_{n-i}b \otimes a_{n-i+1} \otimes \ldots \otimes a_n\,.
  \end{aligned}
\end{equation}
For $i=0$, we have the \emph{outer $\AA$-bimodule structure} on $\AA^{\otimes n}$ given by multiplication on the left of the left-most component, and on the right of the right-most component. In that case, we omit to write $\ast_0$, so that $b d=b \ast_0 d$ and $d b = d \ast_0 b$ for any $d\in \AA^{\otimes n}$.  
We set $\ast_{i+n}=\ast_i$ to define the operation on $\AA^{\otimes n}$ for any $i \in \Z$. 

We introduce tensor product rules in a similar way as maps 
$\AA\otimes \AA^{\otimes n}\to \AA^{\otimes (n+1)}$ and $\AA^{\otimes n}\otimes \AA \to \AA^{\otimes (n+1)}$. For $0\leq i \leq n-1$, 
\begin{equation}
  \begin{aligned}
    b \otimes_i (a_1 \otimes \ldots \otimes a_n)=& a_1 \otimes \ldots a_i \otimes b  \otimes a_{i+1} \otimes \ldots \otimes a_n \,, \\
(a_1 \otimes \ldots \otimes a_n) \otimes_i b=& a_1 \otimes \ldots\otimes  a_{n-i} \otimes b \otimes a_{n-i+1} \otimes \ldots \otimes a_n\,,
  \end{aligned}
\end{equation}
hence $b \otimes_i (-) = (-) \otimes_{n-i}b \colon \AA^{\otimes n}\to \AA^{\otimes (n+1)}$. 
We omit the subscript for $i=0$.

Let us spell out the case $n=2$. We have the outer ($i=0$) and inner ($i=1$ with $\ast := \ast_1$) bimodule structures on $\AA\otimes \AA$ defined for any $a,b\in \AA$ and $d\in \AA\otimes\AA$ by 
\begin{equation}
    a\,d \,b=ad'\otimes d''b\,, \quad 
    a\ast d\ast b=d'b\otimes ad''\,.
\end{equation}
Furthermore, $a\otimes_1 d=d'\otimes a \otimes d''=d \otimes_1 a$.

%\subsubsection{Derivations on tensor products}
The previous constructions can be used for a differential algebra $(\VV,\del)$, that is $\VV$ is an associative algebra equipped with a derivation 
$$\del \in \Der(\VV)=\{\delta \in \Hom(\VV,\VV) \mid \delta(ab)= a \delta(b) + \delta(a) b\}\,.$$
In particular, we note that the extension \eqref{Eq:DelExtend}  of $\del$  to $\VV^{\otimes 2}$ is the sum of the extensions $\del_{L,R}\colon \VV^{\otimes 2}\to \VV^{\otimes 2}$ \eqref{Eq:DelLRExtend} which satisfy ($a_1,a_2,b_1,b_2\in \VV$)
\begin{equation*}
\begin{aligned}
 \del_L(a_1b_1\otimes a_2b_2)&=\del(a_1)b_1\otimes a_2b_2+ a_1\del(b_1)\otimes a_2b_2 \,, \\
 \del_R(a_1b_1\otimes a_2b_2)&=a_1 b_1\otimes \del(a_2)b_2 + a_1b_1 \otimes a_2\del(b_2)\,.
\end{aligned}
\end{equation*}

%\subsubsection{Symbols of operators}
If $(\VV,\del)$ is a differential algebra, we call $\VV[\del]$ the algebra of scalar differential operators. 
An element $P(\del)=\sum_{n\geq 0}p_n \del^n\in \VV[\del]$ acts on $b\in \VV$ by the obvious formula $P(\del)b=\sum_{n\geq 0}p_n \del^n(b)$. Its symbol is the polynomial $P(\lambda)=\sum_{n\geq 0}p_n \lambda^n\in \VV[\lambda]$.   
We set\footnote{We should see $\res_\lambda \lambda^{-n-1}$ as a linear map $\VV[\lambda]\to \VV$. 
It is obtained from the usual residue $\res_\lambda\colon \VV[\lambda^{\pm 1}]\to \VV$, $\sum_{n\in \Z} p_n \lambda^n\mapsto p_{-1}$.} $\res_\lambda \lambda^{-n-1} P(\lambda)=p_n$ for all $n\in\Z_{\geq0}$. 
Given such a symbol of a scalar differential operator, we recall the standard notation (see e.g. \cite{DSKV}) 
\begin{align*}
\big|_{x=\del}\, P(\lambda+x) &=P(\lambda+\del) = \sum_{n\geq 0}\sum_{k=0}^n \binom{n}{k} \del^k(p_n) \lambda^{n-k}\in \VV[\lambda] \,, \\
P(\lambda+x)\, (\big|_{x=\del} b) &= P(\lambda+\del)_{\to}b = \sum_{n\geq 0} p_n \sum_{k=0}^n \binom{n}{k}  \del^k(b) \lambda^{n-k}\in \VV[\lambda]\,,   \\
(\big|_{x=\del} b)\, P(\lambda+x) &= b{}_{\leftarrow}P(\lambda+\del) = \sum_{n\geq 0}\sum_{k=0}^n \binom{n}{k}  \del^k(b)p_n \lambda^{n-k}\in \VV[\lambda]\,, 
\end{align*}
with $b\in \VV$.  
Similarly, we consider $\VV^{\otimes 2}[\del]$ as the algebra of scalar differential operators on $\VV\otimes \VV$ and use an analogous notation for the symbols  $P(\lambda)\in\VV^{\otimes 2}[\lambda]$ acting on  $\VV\otimes \VV$. 

\subsection{Representation algebra} 
\label{sss:RepAlg} 
Let $N \in \Z_{>0}$. For an algebra $\AA$, define $\AA_N$ as the commutative algebra with generators $\{a_{ij} \mid a\in \AA,\, 1\leq i,j \leq N\}$, subject to the relations 
\begin{equation} \label{Eq:RepRel}
  1_{ij}=\delta_{ij}\,, \quad (ab)_{ij}=\sum_{1\leq k \leq N}a_{ik}b_{kj}\,, \quad 
(\gamma a+\gamma' b)_{ij}=\gamma a_{ij}+\gamma' b_{ij}\,,
\end{equation}
for any $a,b \in \AA$, $\gamma,\gamma' \in \kk$, $1\leq i,j \leq N$. We call $\AA_N$ the \emph{$N$-th representation algebra} of~$\AA$. Clearly, $\AA_N$ is finitely generated if $\AA$ has this property. 
Note that $(-)_N$ is a functor, hence a morphism $\theta\colon \AA_1\to\AA_2$ gives rise to a morphism of commutative algebras $\theta_N \colon (\AA_1)_N\to (\AA_2)_N$, which we may simply denote $\theta:=\theta_N$. 
For example, if $\del \in \Der(\AA)$, it induces a unique derivation $\del\in \Der(\AA_N)$ given on generators by 
\begin{equation} \label{Eq:DerRep}
   \del(a_{ij}):=(\del(a)\!)_{ij}\,, \qquad  a\in \AA,\,\,\, 1\leq i,j\leq N\,.
\end{equation}

The algebra $\AA_N$ is the coordinate ring of the \emph{representation scheme} $\Rep(\AA,N)$ parametrised by representations of $\AA$ on $\kk^N$. 
There is a natural left action of $\Gl_N:=\Gl_N(\kk)$ on $\Rep(\AA,N)$ by conjugation of matrices. In turn, this induces a left action on $\AA_N$ by automorphisms through 
\begin{equation} \label{Eq:ActRep}
    g\cdot a_{ij}:=\sum_{1\leq k,l\leq N} (g^{-1})_{ik} a_{kl} g_{lj}\,, \qquad 
    g\in \Gl_N,\,\, a\in \AA,\,\, 1\leq i,j\leq N\,.
\end{equation}
The Lie algebra $\gl_N:=\gl_N(\kk)$ inherits an infinitesimal action by derivations on $\AA_N$ through 
\begin{equation} \label{Eq:ActInfRep}
    \xi \cdot a_{ij}:=\sum_{1\leq k \leq N} (a_{ik} \xi_{kj} - \xi_{ik} a_{kj})\,, \qquad 
    \xi \in \gl_N,\,\, a\in \AA,\,\, 1\leq i,j\leq N\,.
\end{equation} 
In terms of the matrix-valued element $X(a):=(a_{ij})_{1\leq i,j\leq N}$, these actions read $g\cdot X(a)=g^{-1}X(a) g$ and $\xi \cdot X(a)=-[\xi,X(a)]$. 

When performing reduction on representation algebras in Sections \ref{S:PoiRed}, \ref{Sec:SemiS} and \ref{Sec:Momap}, we shall assume that $\kk$ is algebraically closed in addition to the running assumption $\operatorname{char}(\kk)=0$.

\subsection{Commutative reminder: Poisson (vertex) algebras} 
\label{ss:Remind}

%\subsubsection{Definitions} 
Let $A$ be a commutative algebra. 
A Poisson bracket on $A$ is a linear map 
$$\br{-,-} \colon A \otimes A \to A\,, \quad a\otimes b \mapsto \br{a,b}$$
such that for all $a,b,c\in A$ 
\begin{subequations}
  \begin{align}
& \quad \br{a,b}=- \br{b,a}\,, \quad & \text{(skewsymmetry)}  \\
& \quad \br{a, bc}=\br{a, b}c+b \br{a, c}\,, \quad & \text{(Leibniz rule)}  \\
& \quad \br{a,\br{b,c}}-\br{b,\br{a,c}}-\br{\br{a,b},c}=0\,.  
\quad & \text{(Jacobi identity)}  \label{Eq:J1}
  \end{align}
\end{subequations}
We say that $(A,\br{-,-})$ is a Poisson algebra. The category of Poisson algebras over $\kk$ is denoted by $\PA$. Its morphisms $(A_1,\br{-,-}_1)\to (A_2,\br{-,-}_2)$ are the morphisms of algebras $\phi\colon 
A_1\to A_2$ satisfying $\br{\phi(a),\phi(b)}_2=\phi(\br{a,b}_1)$ for any $a,b\in A_1$.

\begin{definition}[PVA from a bracket]  \label{Def:PVA-br} 
Let $V$ be a commutative differential algebra for the derivation $\del\in \Der(V)$. 
    A $\lambda$-bracket on $V$ is a linear map 
$$\br{-_\lambda-} \colon V \otimes V \to V[\lambda]\,, \quad a\otimes b \mapsto \br{a_\lambda b}$$
such that for all $a,b,c\in V$ 
\begin{subequations}
  \begin{align}
& \quad \br{\del(a)_\lambda b}=-\lambda \br{a_\lambda b}\,, \quad 
\br{a_\lambda \del(b)}=(\lambda+\del)\br{a_\lambda b}\,,
& \text{(sesquilinearity)}  \label{Eq:A1}\\
& \quad \br{a_\lambda bc}=\br{a_\lambda b}c+b \br{a_\lambda c}\,, \quad & \text{(left Leibniz rule)} \label{Eq:lL} \\
& \quad \br{ab_\lambda c}=\br{a_{\lambda+\del} c}_{\to}b 
+  \br{b_{\lambda+\del} c}_{\to}a \,. \quad & \text{(right Leibniz rule)}  \label{Eq:rL}
  \end{align}
\end{subequations}
The couple $(V,\br{-_\lambda-})$ is a Poisson vertex algebra if in addition 
\begin{subequations}
  \begin{align}
& \quad \br{a_{\lambda}b}=- \br{b_{-\lambda-\del}a}\,, \quad & \text{(skewsymmetry)}  \label{Eq:A2} \\
& \quad \br{a_\lambda \br{b_\mu c}}-\br{b_\mu\br{a_\lambda c}}-\br{\br{a_\lambda b}_{\lambda+\mu}c}=0\,.  
\quad & \text{(Jacobi identity)} \label{Eq:A3} 
  \end{align}
\end{subequations}
\end{definition}

\begin{definition}[PVA from endomorphisms] \label{Def:PVA-end}
A vector space $V$ is a Lie vertex algebra if it is equipped with a linear operator $\del\in \End(V)$ and $\kk$-bilinear products  
$$V \times V \to V\,, \quad (a,b) \mapsto a_{(n)}b\,, \quad n\in \Z_{\geq0}\,,$$
such that for fixed $a,b\in V$, $a_{(n)}b=0$  for $n$ sufficiently large, and for all $a,b,c\in V$, $m,n\in \Z_{\geq0}$: 
\begin{subequations}
  \begin{align}
& \quad (\del a)_{(n)}b=-n a_{(n-1)}b\,, \quad 
a_{(n)}(\del b)=\del(a_{(n)}b)+n a_{(n-1)}b \,,& \text{(sesquilinearity)}  
\label{Eq:A1-b} \\
& \quad a_{(n)}b =\sum_{i\geq 0}\frac{(-1)^{n+i+1}}{i!} \del^i(b_{(n+i)}a)\,, & \text{(skewsymmetry)}  \label{Eq:A2-b} \\
& \quad a_{(m)}b_{(n)} c - b_{(n)}a_{(m)}c=\sum_{i=0}^m \binom{m}{i} (a_{(i)}b)_{(m+n-i)}c\,. & \text{(Jacobi identity)} \label{Eq:A3-b}
  \end{align}
\end{subequations}
If in addition to the Lie vertex algebra structure, $(V,\del)$ is a commutative differential algebra, we say that it is a Poisson vertex algebra when 
\begin{subequations}
  \begin{align}
& \quad a_{(n)} bc=(a_{(n)} b)\,c+b \, (a_{(n)} c)\,,   & \text{(left Leibniz rule)} \label{Eq:L-b} \\
& \quad (ab)_{(n)}c=
\sum_{i \geq 0} \frac{1}{i!} \del^i(a)\, (b_{(n+i)}c) 
 +\sum_{i \geq 0} \frac{1}{i!}  (a_{(n+i)}c)\, \del^i(b)  \,.   & \text{(right Leibniz rule)}   \label{Eq:R-b} 
  \end{align}
\end{subequations}
\end{definition}

\begin{proposition}[\cite{Ka,AM}]
\label{Pr:CorPVA}
The two definitions are equivalent.
\end{proposition}
\begin{proof}
It suffices to translate the axioms from each definition into one another using the correspondence $\br{a_\lambda b}:=\sum_{n\geq 0}\frac{\lambda^n}{n!} a_{(n)}b$ and $a_{(n)}b:=n!\, \res_{\lambda}\lambda^{-n-1}\br{a_\lambda b}$ for any $a,b\in V$. 
\end{proof}

\begin{remark} \label{Rem:LVA}
By repeating the same argument, one sees that a Lie vertex algebra is a vector space $V$ equipped with a linear operator $\del\in \End(V)$ and a linear map $\br{-_\lambda-}\colon V\otimes V\to V[\lambda]$ satisfying \eqref{Eq:A1}, \eqref{Eq:A2} and \eqref{Eq:A3}. We say that $\br{-_\lambda-}$ is a \emph{Lie vertex bracket}. 
\end{remark}

The category of Poisson vertex algebras over $\kk$ is denoted by $\PVA$. Its morphisms are the morphisms of differential algebras $\phi\colon (V_1,\del_1)\to (V_2,\del_2)$ (hence $\phi(\del_1(a)\!)=\del_2(\phi(a)\!)$) satisfying $\br{\phi(a)_\lambda\phi(b)}_2=\phi(\br{a_\lambda b}_1)$ for any $a,b\in V_1$. Equivalently, the later condition is $\phi(a)_{(n)_2}\phi(b)=\phi(a_{(n)_1}b)$ for any $n\geq 0$ if one defines $(V_i,\del_i,-_{(n)_i}-)$, $i=1,2$, using Definition \ref{Def:PVA-end}.

%\subsubsection{Functorial constructions} \label{ss:Com-Functor}

Write $a\mapsto \bar{a}$ for the linear map $V\to V/\del V$ from $V$ to the vector space obtained as the quotient of $V$ by $\operatorname{im}(\del)$. We denote in the same way the algebra homomorphism  $V\to V/\langle\del V\rangle$ where $\langle\del V\rangle$ is the ideal generated by $\operatorname{im}(\del)$ in $V$.

\begin{lemma} \label{Lem:PVAtoPA}
Assume that $(V,\del, \br{-_\lambda-})$ is a Poisson vertex algebra. 
Then $V/\langle\del V\rangle$ equipped with the bilinear map 
\begin{equation*}   
 \{-,-\} \colon  V/\langle\del V\rangle \times V/\langle\del V\rangle \to V/\langle\del V\rangle\,, \quad 
\{ \bar{a},\bar{b}\} :=\overline{\br{a_\lambda b}}|_{\lambda =0} =\overline{a_{(0)}b}\,,
\end{equation*}
is a Poisson algebra.     
Moreover, this construction extends to a functor $\mathtt{Q}\colon \PVA\to \PA$. 
\end{lemma}
\begin{proof} 
The proof is similar to showing that the vector space $V/\del V$ equipped with $ [-,-]$ is a Lie algebra, see \cite[\S5.3]{Ka} or  \cite[\S4.1]{FBZ}. The details are left to the reader.  
\end{proof}

\subsection{Non-commutative reminder: double Poisson (vertex) algebras} 
\label{ss:NCRemind}

%\subsubsection{Double Poisson algebras} 
Let $\AA$ be a (unital associative) algebra of finite type over $\kk$.

\begin{definition}[\cite{VdB}]
 \label{def:DBr}
  A \emph{double bracket}  (or \emph{2-fold bracket}) on $\AA$ is a linear map 
\begin{equation*}
\dgal{-,-} \colon   \AA \otimes \AA \longrightarrow \AA\otimes \AA\,, \quad a \otimes b \longmapsto \dgal{a, b}\,,
\end{equation*}
such that for all $a,b,c\in \AA$
\begin{subequations}
  \begin{align}
& \quad \dgal{a,b}=- \dgal{b,a}^\sigma\,, \quad & \text{(cyclic skewsymmetry)} \label{Eq:DA} \\
& \quad \dgal{a, bc}=\dgal{a, b}c+b \dgal{a, c}\,, 
\quad & \text{(left Leibniz rule)} \label{Eq:Dl} \\
& \quad \dgal{ab, c}=\dgal{a,c}\ast_1 b 
+a \ast_1 \dgal{b,c}\,.  
\quad & \text{(right Leibniz rule)} \label{Eq:Dr}
  \end{align}
\end{subequations}
\end{definition}
Given a double bracket, we introduce the maps 
\begin{equation}
\begin{aligned}
    &\dgal{a,b' \otimes b''}_L=\dgal{a,b'}\otimes b''\,, \quad 
\dgal{a,b' \otimes b''}_R=b' \otimes \dgal{a,b''}\,, \\
&\dgal{a' \otimes a'',b}_L=\dgal{a',b}\otimes_1 a''\,, \quad 
\dgal{a' \otimes a'',b}_R=a' \otimes_1\dgal{a'',b}\,.
\end{aligned}
\end{equation}

\begin{definition}[\cite{VdB}]
\label{def:DPA}
  A \emph{double Poisson algebra} is an algebra $\AA$ endowed with a double bracket such that for all $a,b,c\in \AA$
\begin{equation}
    \quad \dgal{a, \dgal{b,c}}_L-\dgal{b,\dgal{a,c}}_R-\dgal{\dgal{a,b},c}_L=0
\,. \quad \text{(Jacobi identity)} \label{Eq:DJ} 
\end{equation}
In that case, we say that $\dgal{-,-}$ is a \emph{double Poisson bracket}.
\end{definition}
\begin{remark}
In Definition \ref{def:DPA}, we chose the condition \eqref{Eq:DJ} which is given in \cite{DSKV} and is equivalent to the original condition of Van den Bergh: 
\begin{equation*}
    \quad \dgal{a, \dgal{b,c}}_L+(\dgal{b,\dgal{c,a}}_L)^\sigma
+(\dgal{c,\dgal{a,b}}_L)^{\sigma^2}=0
\,.
\end{equation*}
\end{remark}

Double Poisson brackets can be seen as a noncommutative version of Poisson brackets due to the next result. 
%where we use the notations from \S\ref{sss:RepAlg}. 

\begin{theorem}[\cite{VdB}]
\label{Thm:RepdP}
Assume that $\dgal{-,-}$ is a double bracket on $\AA$. Then there is a unique skewsymmetric biderivation on $\AA_N$ which satisfies for any $a,b \in \AA$, $1 \leq i,j \leq N$, 
\begin{equation} \label{Eq:relPA}
  \br{a_{ij},b_{kl}}= \dgal{a, b}'_{kj} \dgal{a, b}''_{il}\,.
\end{equation}
Furthermore, if $\dgal{-,-}$ is a double Poisson bracket, then $(\AA_N,\br{-,-})$ is a  Poisson algebra. 
\end{theorem} 
Note that the Poisson bracket obtained through \eqref{Eq:relPA} 
is invariant under the $\Gl_N$ action given in \S\ref{sss:RepAlg}. 
That is, for any $g\in \Gl_N$, the map $g\cdot (-) \in \Aut(\AA_N)$ is a Poisson morphism. 
We shall denote by $\gl_N$ the Lie algebra of $\Gl_N$ which is, as a vector space, 
the set $\Mat_N:=\Mat_N(\kk)$ of $\kk$-valued $N\times N$ matrices.

\begin{example}[\cite{Po,VdB}]  \label{Ex:ku}
Let $\AA=\kk[a]$. Then a double bracket $\dgal{-,-}$ on $\AA$ is a double Poisson bracket if and only if there exists $\alpha,\beta,\gamma\in \kk$ satisfying $\beta^2=\alpha \gamma$ so that   
\begin{equation*}
  \dgal{a,a}=\alpha(a \otimes 1 - 1 \otimes a) + \beta (a^2 \otimes 1 - 1 \otimes a^2) 
+ \gamma (a^2 \otimes a - a \otimes a^2)\,.
\end{equation*}    
For any $N\geq 1$, we have  $\AA_N=\kk[a_{ij} \mid 1\leq i,j\leq N]$. 
Assuming that $\beta=\gamma=0$, the Poisson bracket on $\AA_N$ induced by Theorem \ref{Thm:RepdP} is uniquely determined by 
\begin{equation} \label{Eq:KKS}
    \br{a_{ij},a_{kl}}=a_{kj}  \delta_{il} - \delta_{kj} a_{il}\,.
\end{equation}
Under the identification $\AA_N \simeq \kk[\Mat_N] \simeq \kk[\gl_N^\ast]$ 
through the trace pairing, we can see $a_{ij}$ as a linear map $\xi\mapsto \xi_{ij}$ returning the $(i,j)$-entry of the $N\times N$ matrix $\xi$.  
Therefore the Poisson bracket on $\AA_N$ satisfying \eqref{Eq:KKS} is the Kirillov-Kostant-Souriau Poisson bracket on $\kk[\gl_N^\ast]$.  
\end{example}

\begin{example}   
\label{Exmp:Symp}
Let $\AA=\kk\langle a,b \rangle$. 
Then   $\dgal{a,a}=0=\dgal{b,b}$, $\dgal{b,a}=1 \otimes 1$ defines a unique double Poisson bracket on $\AA$, 
  and the induced Poisson structure on 
  $$\AA_N= \kk[\Mat_N \times \Mat_N]\cong 
  \kk[\gl_N] \otimes \kk[\gl_N^\ast]\cong \kk[T^\ast \gl_N],$$  
 is that coming from the symplectic structure on $T^\ast \gl_N$.  
  \end{example}
  
\begin{example}  
\label{Exmp:SympGL}
Let $\AA=\kk\langle a,b^{\pm1}\rangle$ 
with the unique double Poisson bracket such that 
\begin{equation*}
 \dgal{b,b}=0,\quad \dgal{a,b}=b\otimes 1, \quad \dgal{a,a}=a\otimes 1 - 1\otimes a. 
\end{equation*} 
The induced Poisson structure 
on $\AA_N\cong \kk[\Gl_N] \otimes \kk[\gl_N^\ast]$ is 
that coming from the symplectic structure on 
of the cotangent bundle $T^\ast \Gl_N$. 
\end{example}

%\subsubsection{Double Poisson vertex algebras}
Let $(\VV,\del)$ be a differential (unital associative) algebra.

\begin{definition}[\cite{DSKV}] \label{def:DlamBr} 
  A \emph{double $\lambda$-bracket} (or \emph{$2$-fold $\lambda$-bracket}) on $\VV$ is a linear map 
\begin{equation*}
\dgal{-_\lambda -} \colon  \VV \otimes \VV \longrightarrow (\VV\otimes \VV)[\lambda]\,, 
\quad a \otimes b \longmapsto \dgal{a_\lambda b}
\end{equation*}
such that for all $a,b,c\in \VV$
\begin{subequations}
  \begin{align}
&  \dgal{\del(a)_\lambda b}=-\lambda \dgal{a_\lambda b}\,, \quad 
\dgal{a_\lambda \del(b)}=(\del+\lambda)\dgal{a_\lambda b}\,,  & \text{(sesquilinearity)} \label{Eq:DA1} \\
& \dgal{a_\lambda bc}=\dgal{a_\lambda b}c+b \dgal{a_\lambda c}\,,  & \text{(left Leibniz rule)} \label{Eq:DL} \\
& \dgal{ab_\lambda c}=\dgal{a_{\lambda+x}c}\ast_1\restriction{\Big(}{x=\del} b\Big) 
+\restriction{\Big(}{x=\del} a\Big) \ast_1  \dgal{b_{\lambda+x}c}\,.   & \text{(right Leibniz rule)} \label{Eq:DR}
  \end{align}
\end{subequations}
\end{definition}

Given a double $\lambda$-bracket, we introduce the maps 
\begin{equation}
\begin{aligned}
   & \dgal{a_\lambda b' \otimes b''}_L=\dgal{a_\lambda b'}\otimes b''\,, \quad 
\dgal{a_\lambda b' \otimes b''}_R=b' \otimes \dgal{a_\lambda b''} \\
&\dgal{a' \otimes a''{}_\lambda b}_L=\dgal{a'_{\lambda+x} b}\otimes_1 \restriction{\Big(}{x=\del} a''\Big) \,, \\ 
&\dgal{a' \otimes a''{}_\lambda b}_R=\restriction{\Big(}{x=\del} a'\Big)  \otimes_1\dgal{a''_{\lambda+x}b}\,.
\end{aligned}
\end{equation}

\begin{definition}[\cite{DSKV}]
 \label{def:DPVA} 
  A \emph{double Poisson vertex algebra} is a differential algebra $\VV$ endowed with a double $\lambda$-bracket such that 
\begin{subequations}
  \begin{align}
&\quad \dgal{a_\lambda b}=-\big|_{x=\del} \dgal{b_{-\lambda-x}a}^\sigma
\,, & \text{(skewsymmetry)} \label{Eq:DA2} \\
&\quad \dgal{a_\lambda \dgal{b_\mu c}}_L-\dgal{b_\mu \dgal{a_\lambda c}}_R-\dgal{\dgal{a_\lambda b}_{\lambda+\mu}c}_L=0
\,.  & \text{(Jacobi identity)} \label{Eq:DA3} 
  \end{align}
\end{subequations}
\end{definition}

\begin{remark}
The original definition in \cite{DSKV} (considered in \cite{AFH}) uses the equivalent form of the right Leibniz rule \eqref{Eq:DR} given by ($a,b,c\in \VV$)
\begin{equation}
    \dgal{ab_\lambda c}=\dgal{a_{\lambda+\del}c}_{\to} \ast_1 b
+ (e^{\del \frac{d}{d\lambda}} a ) \ast_1  \dgal{b_\lambda c}\,, 
\end{equation}
for  $(e^{\del \frac{d}{d\lambda}} a ) \ast_1  (b'\otimes b'')\lambda^j:=\sum_{k=0}^j \frac{1}{k!} b'\otimes \del^k(a)b''\, \lambda^{j-k}$.  
\end{remark}

The next result is the analogue of Theorem \ref{Thm:RepdP}.

\begin{theorem} \emph{(\cite{DSKV})} \label{Thm:RepdPV}
Assume that $\dgal{-_\lambda-}$ is a double  $\lambda$-bracket on $\VV$. Then there is a unique $\lambda$-bracket on $\VV_N$ which satisfies for any $a,b \in \VV$, $1 \leq i,j \leq N$, 
\begin{equation} \label{Eq:relDPVA}
  \br{a_{ij\, \lambda}b_{kl}}=\sum_{n \geq 0} \dgal{a_{(n)} b}'_{kj} \dgal{a_{(n)} b}''_{il} \lambda^n, \, 
\text{ for }\dgal{a_\lambda b}=\sum_{n \geq 0} (\dgal{a_{(n)} b}' \otimes \dgal{a_{(n)} b}'')\, \lambda^n\,.
\end{equation}
Furthermore, if $(\VV,\dgal{-_\lambda -})$ is a double Poisson vertex algebra, then $(\VV_N,\br{-_\lambda-})$ is a Poisson vertex algebra on which $\Gl_N$ acts by Poisson vertex automorphisms. 
\end{theorem}
\begin{proof}
The main statement is proved in \cite[\S3.7]{DSKV}. 
The fact that $\Gl_N$ acts through \eqref{Eq:ActRep} by Poisson vertex automorphisms amounts to requiring  
\begin{equation*}
   \del(g\cdot F_1) = g\cdot \del(F_1)\,, \quad 
 \br{(g\cdot F_1){}_\lambda (g\cdot F_2)}=g\cdot   \br{F_1{}_\lambda F_2}\,, \qquad g\in\Gl_N,\,\, F_1,F_2\in \VV_N\,.
\end{equation*}
These equalities follow by checking them when $F_1$ and $F_2$ are generators of $\VV_N$; this is readily obtained from \eqref{Eq:DerRep} and \eqref{Eq:relDPVA} after plugging \eqref{Eq:ActRep}. 
\end{proof}

\begin{example}   
\label{Ex:kuPVA} 
  Let $\VV=\kk\langle u^{(0)},u^{(1)},u^{(2)},\ldots\rangle$ with $\del u^{(i)}=u^{(i+1)}$. Set $u=u^{(0)}$. If 
\begin{equation} \label{Eq:DPVAuu}
  \dgal{u_\lambda u}=1\otimes u - u \otimes 1 + \epsilon \, (1 \otimes 1) \lambda \in \VV^{\otimes 2}[\lambda]\,,
\end{equation}
for some fixed $\epsilon \in \kk$, we can uniquely define a double Poisson vertex algebra structure on $\VV$ satisfying \eqref{Eq:DPVAuu}. On $\VV_N$, the induced $\lambda$-bracket is given by 
\begin{equation}
    \br{u_{ij}{}_\lambda u_{kl}} = \delta_{kj}u_{il} -\delta_{il}u_{kj} + \epsilon\, \delta_{kj}\delta_{il} \lambda\,.
\end{equation}
\end{example}

%%%%%%%%%%%%%%%%%% NEW SECTION %%%%%%%%%%%%%%%%%%%%%
%%%%%%%%%%%%%%%%%% NEW SECTION %%%%%%%%%%%%%%%%%%%%%
%%%%%%%%%%%%%%%%%% NEW SECTION %%%%%%%%%%%%%%%%%%%%%
%%%%%%%%%%%%%%%%%% NEW SECTION %%%%%%%%%%%%%%%%%%%%%

\section{Alternative definition of double Poisson vertex algebras} 
\label{S:Alt-PVA}

\begin{definition} \label{def:DPVA-b}
A vector space $\VV$ is a \emph{double Lie vertex algebra} if it is equipped with a linear operator $\del\in \End(\VV)$ and $\kk$-bilinear double products  
$$\VV \times \VV \longrightarrow \VV\otimes \VV\,, \quad (a,b) \longmapsto a_{(\!(n)\!)}b\,, \quad n\in \Z_{\geq0}\,,$$
such that for fixed $a,b\in \VV$, $a_{(\!(n)\!)}b=0$  for $n$ sufficiently large, 
and for all $a,b,c\in \VV$, $m,n\in \Z_{\geq0}$: 
\begin{subequations}
  \begin{align}
&  \del(a)_{(\!(n)\!)}b=-n a_{(\!(n-1)\!)}b\,, \quad 
a_{(\!(n)\!)}\del(b)=\del(a_{(\!(n)\!)}b)+n a_{(\!(n-1)\!)}b \,, & \text{(sesquilinearity)} \label{Eq:DA1-b}  \\
&  a_{(\!(n)\!)}b =\sum_{i\geq 0}\frac{(-1)^{n+i+1}}{i!} \del^i(b_{(\!(n+i)\!)}a)^\sigma \,,   & \text{(skewsymmetry)}  \label{Eq:DA2-b}  \\
&  a_{(\!(m;L)\!)}b_{(\!(n)\!)} c - b_{(\!(n;R)\!)}a_{(\!(m)\!)}c=\sum_{i=0}^m \binom{m}{i} (a_{(\!(i)\!)}b)_{(\!(m+n-i;L)\!)}c\,,   & \text{(Jacobi identity)}  \label{Eq:DA3-b} 
  \end{align}
\end{subequations}
where we use the extended maps 
\begin{align*}
&a_{(\!(n;L)\!)}-\colon \VV^{\otimes 2}\to \VV^{\otimes 3},   \qquad  a_{(\!(n;L)\!)} (b'\otimes b''):= a_{(\!(n)\!)}b' \otimes b''\,, \\ 
&a_{(\!(n;R)\!)}-\colon \VV^{\otimes 2}\to \VV^{\otimes 3},   \qquad  a_{(\!(n;R)\!)} (b'\otimes b''):= b'\otimes a_{(\!(n)\!)} b'' \,, \\ 
&(a'\otimes a'')_{(\!(n;L)\!)}-\colon \VV\to \VV^{\otimes 3},   \qquad  (a'\otimes a'')_{(\!(n;L)\!)}b := \sum_{j\geq 0}\frac{1}{j!}  (a'_{(\!(n+j)\!)} b) \otimes_1 \del^j(a'') \,.
\end{align*}
If in addition to the double Lie vertex algebra structure, $(\VV,\del)$ is a differential algebra, we say that it is a \emph{double Poisson vertex algebra} when  
\begin{subequations}
  \begin{align}
& \quad a_{(\!(n)\!)} bc=(a_{(\!(n)\!)} b)\,c+b \, (a_{(\!(n)\!)} c)\,,  & \text{(left Leibniz rule)}   \label{Eq:DL-b} \\
& \quad (ab)_{(\!(n)\!)}c=
\sum_{i \geq 0} \frac{1}{i!} \del^i(a)\ast_1 \, (b_{(\!(n+i)\!)}c) 
 + \sum_{i \geq 0} \frac{1}{i!}  (a_{(\!(n+i)\!)}c)\, \ast_1 \del^i(b)   \,. 
 & \text{(right Leibniz rule)}   \label{Eq:DR-b} 
  \end{align}
\end{subequations}
\end{definition}

\begin{remark}
Due to skewsymmetry \eqref{Eq:DA2-b}, the two sesquilinearity rules in \eqref{Eq:DA1-b} are equivalent, and the left/right Leibniz rules \eqref{Eq:DL-b}--\eqref{Eq:DR-b} are equivalent.  Hence, as part of the definition, we can simply pick the two easier identities $\del(a)_{(\!(n)\!)}b=-n a_{(\!(n-1)\!)}b$ and \eqref{Eq:DL-b}. 
A similar remark holds for the original Definition \ref{def:DPVA}; this can also be observed for Poisson vertex algebras (cf. Definitions \ref{Def:PVA-br} and \ref{Def:PVA-end}) and double Poisson algebras (cf. Definition \ref{def:DBr}).  
\end{remark}

\begin{lemma} \label{Lem:dA3-Equiv}
In a double Lie vertex algebra, the Jacobi identity \eqref{Eq:DA3-b} for arbitrary $a,b,c\in \VV$ and $m,n\in \Z_{\geq0}$ can be replaced by 
\begin{equation}
    a_{(\!(m;L)\!)}b_{(\!(n)\!)} c - b_{(\!(n;R)\!)}a_{(\!(m)\!)}c=-\sum_{j=0}^n \binom{n}{j} (b_{(\!(j)\!)}a)^\sigma_{(\!(m+n-j;L)\!)}c\,.  \label{Eq:DA3-Equiv}
\end{equation}
\end{lemma}
\begin{proof}
We first note using \eqref{Eq:DA1-b} that we have the following sesquilinearity rule for the left extended map  
$\del(a'\otimes a'')_{(\!(m;L)\!)}c=-m\, (a'\otimes a'')_{(\!(m-1;L)\!)}c$.  
Combining this observation with skewsymmetry \eqref{Eq:DA2-b}, 
we easily prove the statement. 
%the right-hand side of \eqref{Eq:DA3-b} reads 
%\begin{align*}
%\sum_{i=0}^m \binom{m}{i} (a_{(\!(i)\!)}b)_{(\!(m+n-i;L)\!)}c &=
%\sum_{i=0}^m \binom{m}{i} \sum_{s\geq 0}\frac{(-1)^{i+s+1}}{s!} 
%(\partial^s (b_{(\!(i+s)\!)}a)^\sigma )_{(\!(m+n-i;L)\!)}c \\ 
%&=\sum_{i=0}^m (-1)^{i+1}\binom{m}{i} \sum_{s=i}^{m+n} \binom{m+n-i}{s-i} (\!(b_{(\!(s)\!)}a)^\sigma )_{(\!(m+n-s;L)\!)}c\,.
%\end{align*}     
%The binomial expansion in $\kk[x,y]$ of $(x+y)^n y^m = \sum_{i=0}^m\binom{m}{i} (-1)^i x^i (x+y)^{m+n-i}$ yields  
%\begin{align*}
%    \sum_{j=0}^n \binom{n}{j} x^j y^{m+n-j} =
%    \sum_{i=0}^m (-1)^{i}\binom{m}{i} \sum_{j=i}^{m+n} \binom{m+n-i}{j-i} x^j y^{m+n-j}\,, 
%\end{align*}
%thus we can write the above expression as $-\sum_{j=0}^n \binom{n}{j} (b_{(\!(j)\!)}a)^\sigma_{(\!(m+n-j;L)\!)}c$ and conclude. 
\end{proof}

%\subsection{Equivalence of definitions}

\begin{lemma} \label{Lem:DLVA}
Assume that $\VV$ is a vector space endowed with a linear operator $\del \in \End( \VV)$. 
A structure of double Lie vertex algebra on $\VV$ is equivalent to having a linear map 
\begin{equation*}
\dgal{-_\lambda -} \colon  \VV \otimes \VV \longrightarrow (\VV\otimes \VV)[\lambda]\,, \quad 
a \otimes b \longmapsto \dgal{a_\lambda b}
\end{equation*}
such that \eqref{Eq:DA1}, \eqref{Eq:DA2} and \eqref{Eq:DA3} hold. 
\end{lemma}
\begin{proof}
We claim that the equivalence is realised by defining  
\begin{equation} \label{Eq:DLVA-1}
\dgal{-_\lambda -} \colon  \VV \otimes \VV \longrightarrow (\VV\otimes \VV)[\lambda]\,, a \otimes b \longmapsto \dgal{a_\lambda b}:=\sum_{n\geq 0} \frac{\lambda^n}{n!} a_{(\!(n)\!)}b\,,
\end{equation}
from a structure of double Lie vertex algebra on $\VV$. To go backwards it suffices to define the $\kk$-bilinear double products  
\begin{equation} \label{Eq:DLVA-2}
\VV \times \VV \longrightarrow \VV\otimes \VV\,, \quad (a,b) \mapsto a_{(\!(n)\!)}b:=n!\,\res_{\lambda} \lambda^{-n-1} \dgal{a_\lambda b} \,, \quad n\in \Z_{\geq0}\,.
\end{equation}
Then we leave to the reader the verifications.  
\end{proof}

The following proposition is a direct consequence of Lemma \ref{Lem:DLVA}. 

\begin{proposition} \label{Pr:CorDPVA}
If $(\VV,\del)$ is a differential algebra, the definitions of double Poisson vertex algebras according to Definition \ref{def:DPVA} and Definition \ref{def:DPVA-b} are equivalent. 
\end{proposition}
%\begin{proof}
%It suffices to extend the correspondence of Lemma \ref{Lem:DLVA} given by  \eqref{Eq:DLVA-1} and \eqref{Eq:DLVA-2} so that it is compatible with the Leibniz rules. By skewsymmetry, we only need to check that the left Leibniz rules \eqref{Eq:DL} and \eqref{Eq:DL-b} correspond to one another under \eqref{Eq:DLVA-1} and \eqref{Eq:DLVA-2}; this is straightforward.
%\end{proof}

\subsection{Relation to representation spaces}

\begin{theorem}  \label{Thm:RepdLVA}
Assume that $(\VV,\del)$ is a double Poisson vertex algebra according to  Definition~\ref{def:DPVA-b}. 
Then there is a unique structure of Poisson vertex algebra in the sense of Definition \ref{Def:PVA-end} on $(\VV_N,\del)$  
which satisfies for any $a,b \in \VV$, $1 \leq i,j \leq N$, 
\begin{equation} \label{Eq:relLVA}
  a_{ij}{}_{(n)}b_{kl}=(a_{(\!(n)\!)} b)'_{kj} (a_{(\!(n)\!)} b)''_{il}\,.
\end{equation} 
Moreover, the action of $\Gl_N$ on $\VV_N$ is by Poisson vertex automorphisms. 
\end{theorem}
\begin{proof} 
The second part of the statement is proved as in Theorem \ref{Thm:RepdPV}. 
For the first part, we adapt the strategy of \cite[\S3.7]{DSKV}. 
Details are left to the reader, noting that Lemma \ref{Lem:dA3-Equiv} is needed to check Jacobi identity. 
\end{proof}

\begin{corollary}
The equivalence of definitions of double Poisson vertex algebras given in Proposition \ref{Pr:CorDPVA} induces on representation spaces the equivalence of the corresponding Poisson vertex algebras as in Proposition \ref{Pr:CorPVA}. 
%In other words, the diagram depicted in Figure \ref{Fig:A} is commutative. 
\end{corollary}

\section{Relating double Poisson algebras 
and double Poisson vertex algebras}

\label{Sec:JQ}

%\subsection{Morphisms}   

If $\theta\colon \AA_1\to\AA_2$ is a morphism of algebras, let $\theta^{\otimes 2}\colon  \AA_1^{\otimes 2}\to\AA_2^{\otimes 2}$ be the unique map obtained from the universal property of $\AA_1^{\otimes 2}$ that makes the following diagram commutes 
\begin{center}
\begin{tikzpicture}
 \node  (TopLeft) at (-2,1) {$\AA_1\times \AA_1$};
 \node  (TopRight) at (2,1) {$\AA_1\otimes \AA_1$};
 \node  (BotLeft) at (-2,-1) {$\AA_2\times \AA_2$};
 \node  (BotRight) at (2,-1) {$\AA_2\otimes \AA_2$};
\path[->,>=angle 90,font=\small]  
   (TopLeft) edge node[above] {$- \otimes -$}  (TopRight) ;
   \path[->,>=angle 90,font=\small]  
   (BotLeft) edge node[above] {$- \otimes -$}  (BotRight) ;
\path[->,>=angle 90,font=\small]  
   (TopLeft) edge node[left] {$\theta^{\times 2}$}  (BotLeft) ;
\path[->,dashed,>=angle 90,font=\small]  
   (TopRight) edge node[right] {$\theta^{\otimes 2}$}  (BotRight) ;
   \end{tikzpicture}
\end{center}

\begin{definition}
For $i=1,2$, let $(\AA_i,\dgal{-,-}_i)$  be a double Poisson algebra. A morphism of algebras $\theta\colon \AA_1\to \AA_2$ is said to be a \emph{morphism of double Poisson algebras} if for any $a,b\in \AA_1$, 
\begin{equation}
    \dgal{\theta(a),\theta(b)}_2=\theta^{\otimes 2}(\dgal{a,b}_1)\,.
\end{equation}
We denote by $\DPA$ the category of double Poisson algebras (over $\kk$). 
\end{definition}

We can introduce the above notation for a morphism $\theta\colon \VV_1\to\VV_2$ of differential algebras. 
We denote in the same way its extension as a $\kk[\lambda]$-linear map $\theta\colon \VV_1[\lambda]\to\VV_2[\lambda]$.

\begin{definition}
For $i=1,2$, let $(\VV_i,\del_i,\dgal{-_\lambda-}_i)$  be a double Poisson vertex algebra. A morphism of differential algebras $\theta\colon \VV_1\to \VV_2$ is said to be a \emph{morphism of double Poisson vertex algebras} if for any $a,b\in \VV_1$, 
\begin{equation}\label{defmorphdpvas}
    \dgal{\theta(a)_\lambda\theta(b)}_2=\theta^{\otimes 2}(\dgal{a_\lambda b}_1)\,.
\end{equation}
Equivalently, if we see $(\VV_i,\del_i,-_{(\!(n)\!)_i}-)$ as a double Poisson vertex algebra according to Definition \ref{def:DPVA-b}, we require that for any $a,b\in \VV_1$ and $n\in \Z_{\geq0}$, 
\begin{equation}
    \theta(a)_{(\!(n)\!)_2}\theta(b)=\theta^{\otimes 2}(a_{(\!(n)\!)_1} b)\,.
\end{equation}
We denote by $\DPVA$ the category of double Poisson vertex algebras (over $\kk$). 
\end{definition}

\subsection{From Poisson vertex to Poisson}
Denote by $\mathtt{(Comm)Alg}$, respectively $\mathtt{(Comm)DiffAlg}$, the category of unital associative (commutative), respectively differential, algebras over $\kk$.
Given a differential algebra $\VV$, denote by $\langle\del \VV\rangle$ the two-sided ideal generated by $\operatorname{im}(\del)$ in $\VV$. We have the following universal property, where we omit the reference to $\kk$.

\begin{lemma} \label{Lem:Univ-Del} 
The algebra $\VV/\langle\del \VV\rangle$ is unique up to isomorphism such that 
$$
\Hom_\mathtt{Alg}(\VV/\langle\del \VV\rangle,\AA)\simeq\Hom_\mathtt{DiffAlg}(\VV,\AA)
$$
for any algebra $\AA$, endowed with the trivial differential on the right hand side. 

We denote by $a\mapsto \bar{a}$ the algebra homomorphisms  $\VV\to \VV/\langle\del \VV\rangle$ and $\VV^{\otimes 2}\to (\VV/\langle\del \VV\rangle)^{\otimes 2}$ where we quotient out by the ideal generated by 
$\operatorname{im}(\del_L)+\operatorname{im}(\del_R)$ in $\VV^{\otimes 2}$. 
If $(\VV, \dgal{-_\lambda-})$ is a double Poisson vertex algebra, 
 then  $\VV/\langle\del \VV\rangle$ equipped with the linear map 
\begin{equation*}   
 \dgal{-,-}\colon  \VV/\langle\del \VV\rangle \times \VV/\langle\del \VV\rangle \to \VV/\langle\del \VV\rangle\otimes \VV/\langle\del \VV\rangle\,, \quad 
 \dgal{\bar{a},\bar{b}}:=\overline{\dgal{a_\lambda b}}|_{\lambda =0} \,,
\end{equation*}
is a double Poisson algebra.        
This construction extends to a functor $\mathtt{Q}\colon \DPVA\to \DPA$.  
\end{lemma}

The second part of the lemma is an analogue to Lemma \ref{Lem:PVAtoPA}. 
The proof does not pose any major difficulty, 
so we omit it.

\begin{remark}
The result generalises the observation from \cite{DSKV} that for an associative algebra $\AA$ equipped with the zero differential $\del:= 0_{\AA}$, a structure of double Poisson vertex algebra with $\lambda$-bracket $\dgal{-_\lambda-}$ turns $(\AA,\dgal{-_\lambda-}|_{\lambda=0})$ into a double Poisson algebra. 
\end{remark}

\subsection{From Poisson to Poisson vertex}

%\subsubsection{Jet algebras (commutative setting)} \label{ss:JetAlg}
Let $A$ be a unital commutative algebra. 
There exists a unique (up to isomorphism)
differential algebra $\Jinf A$, called the {\em jet algebra} of $A$ 
such that:  
\begin{align*}
{\rm Hom}_{\mathtt{CommDiffAlg}} (\Jinf A, B)
\cong 
{\rm Hom}_{\mathtt{CommAlg}} (A,B) 
\end{align*}
for any differential algebra $(B,\partial)$. 
We refer to \cite[Lemma 1.1]{AM} for a proof. 
An explicit construction of $\Jinf A$ is presented below %
%in \S\ref{sss:JetAlg} 
(where we do not require the algebra to be commutative).

\begin{lemma}[{\cite[\S2.3]{A12},\cite[\S4.2]{AM}}] 
\label{Lem:PAtoPVA}
    Assume that $(A, \br{-,-})$ is a Poisson algebra. Then  there exists a unique $\lambda$-bracket $\br{-_\lambda -}\colon  \Jinf A \otimes \Jinf A \to \Jinf A[\lambda]$ on the jet algebra $\Jinf A$ which satisfies for any $a,b\in A$, 
\begin{equation}  \label{Eq:PA-PVA}
\br{a_\lambda b}:= \br{a,b}  
\, 
\text{ i.e. } \, 
a_{(n)} b := \delta_{n,0}\,\br{a,b} \,, \quad \forall\,n \geq 0\,. 
\end{equation}   
Moreover, this construction extends to a functor $\mathtt{J}\colon \PA\to \PVA$.
\end{lemma}

%\subsubsection{Jet algebras (associative setting)} \label{sss:JetAlg}

We are going to derive an associative analogue of Lemma \ref{Lem:PAtoPVA}. 
Fix an associative unital algebra $\AA$. 
We form $\Jinf\AA$ as the algebra generated by symbols 
\begin{equation*}
    \del^k (a)\,, \qquad a\in \AA,\,\, k\in \Z_{\geq0}\,,
\end{equation*}
subject to the following relations for any $a,b\in \AA$, $\alpha,\beta\in \kk$ and $k\in \Z_{\geq0}$, 
\begin{equation} \label{Eq:Jinf-Rel}
\del^k(\alpha a + \beta b)=\alpha \del^k(a) + \beta \del^k(b)\,, \quad 
\del^k(ab)=\sum_{j=0}^k \binom{k}{j}\, \del^j(a) \,\del^{k-j}(b)\,.
\end{equation}
When $k=0$, the relation $\del^0(ab)=\del^0(a)\del^0(b)$ for any $a,b\in \AA$ yields an embedding of $\AA$ into $\Jinf\AA$ through $a\mapsto a:=\del^0(a)$ for each $a\in \AA$. 
By construction, if $\sum_{i\in I}a_i=0$ in $\AA$ for some elements $a_i$ of $\AA$, we remark that $\sum_{i\in I}\del^k(a_i)=0$ in $\Jinf\AA$ for each $k\in \Z_{\geq0}$. We also see that $\del^k (\alpha)=0$ for any $\alpha\in \kk\subset \AA$ whenever $k\geq 1$.  
There naturally exists a derivation $\del\in \Der(\Jinf \AA)$ which is uniquely determined by 
\begin{equation}
 \del(\del^k(a)\!):=\del^{k+1}(a)  \,, \quad \text{ for all }a\in \AA,\,\, k\in \Z_{\geq0}\,.
\end{equation}
Its expression on an arbitrary element of $\Jinf\AA$ is obtained recursively using the derivation rule 
$\del(fg)=f\del(g)+\del(f)g$ for any $f,g\in \Jinf\AA$. 
The differential algebra $(\Jinf\AA,\del)$ is called the \emph{jet algebra} of $\AA$. 

\begin{example} \label{Ex:Jet-Free}
Consider the free algebra $F_\ell=\kk\langle x_1,\ldots,x_\ell\rangle$ on $\ell\geq 1$ generators. Then the jet algebra of $F_\ell$ is 
$\Jinf F_\ell=\langle \del^k(x_1),\ldots,\del^k(x_\ell) \mid k\in \Z_{\geq0}  \rangle$ equipped with the derivation satisfying $\del\colon \del^k(x_i)\mapsto \del^{k+1}(x_i)$ for any $1\leq i \leq \ell$. Thus $\Jinf F_\ell$ is the differential algebra freely generated (as a differential algebra) by $\ell$ generators. 
\end{example}

\begin{example}  \label{Ex:Jet-Gen}
Consider $\AA=\kk\langle x_1,\ldots,x_\ell\rangle/\langle f_1,\ldots,f_r\rangle$, where we take the quotient of $F_\ell$ by the (2-sided) ideal generated by $f_1,\ldots,f_r\in F_\ell$. 
By construction, the jet algebra of $\AA$ is 
$$\Jinf \AA=\langle \del^k(x_1),\ldots,\del^k(x_\ell) \mid k\in \Z_{\geq0}  \rangle / \langle \del^k(f_1),\ldots,\del^k(f_r)\mid k \in \Z_{\geq0} \rangle$$
which is obtained by quotienting out $\Jinf F_\ell$ by the ideal generated by the elements $\del^k(f_s)\in \Jinf F_\ell$ with $k\in \Z_{\geq0}$ and $1\leq s \leq r$. For $k\geq 1$, $\del^k(f_s)$ is understood as being obtained by applying $k$ times the derivation $\del \in \Der(\Jinf F_\ell)$ to the image $\del^0(f_s)$ of $f_s$ under the embedding $F_\ell \hookrightarrow \Jinf F_\ell$.  
\end{example}

We now establish the universal character of $\Jinf\AA$. Recall that a  morphism of differential algebras $\theta\colon (\VV_1,\del_1)\to (\VV_2,\del_2)$ is such that $\theta(\del_1(v)\!)=\del_2(\theta(v)\!)$ for any $v\in \VV_1$.

\begin{lemma} [Analogue of Lemma 1.1 in \cite{AM}]
For any unital associative algebra 
$\AA$, there exists a unique (up to isomorphism) differential algebra $\Jinf\AA$ such that
\begin{equation}
    \Hom_\mathtt{DiffAlg}(\Jinf\AA,\VV)  = \Hom_\mathtt{Alg}(\AA,\VV)\,,
\end{equation} 
for any (associative) differential algebra $\VV$.    
\end{lemma}
\begin{proof}
The uniqueness of $\Jinf\AA$ follows by applying Yoneda's lemma. It remains to check that the definition of $\Jinf\AA$ given above satisfies the desired universal property. 
 
For a fixed differential algebra $(\VV,\del_\VV)$, a  morphism of algebras $\theta_0\colon \AA\to \VV$ uniquely extends to a morphism of differential algebras $\theta\colon (\Jinf\AA,\del)\to (\VV,\del_\VV)$. Indeed, as $\AA$ embeds into $\Jinf\AA$, compatibility with the derivations completely determines $\theta$ from its value on $\AA$ by requiring $\theta(\del^k(a)\!)=\del_\VV^k(\theta(a)\!)$ for any $a\in \AA$.  
Given $\theta\colon (\Jinf\AA,\del)\to (\VV,\del_\VV)$, we get a morphism of algebras $\theta_0\colon \AA\to \VV$ by considering the restriction of $\theta$ to the subalgebra $\AA\subset \Jinf\AA$. It is direct that $\theta$ and $\theta_0$ correspond to one another. 
\end{proof}

%\subsubsection{Double Poisson vertex algebra structure on jet algebras}
Recall that we embed $\AA$ in the jet algebra $\Jinf\AA$ through $a\mapsto a:=\del^0(a)$.

\begin{lemma} \label{Lem:DPAtoDPVA}
Assume that $(\AA, \dgal{-,-})$ is a double Poisson algebra. The jet algebra  $\Jinf\AA$ admits a unique structure of double Poisson vertex algebra satisfying for any $a,b\in \AA$ 
\begin{equation} \label{Eq:PtoPA-a}
   \dgal{a_\lambda b}=\dgal{a,b} \, 
\text{ i.e. } \, 
a_{(\!(n)\!)} b := \delta_{n,0}\,\dgal{a,b} \,, \quad \forall\,n \geq 0\,. 
\end{equation} 
\end{lemma}
\begin{proof}
By Proposition \ref{Pr:CorDPVA}, it suffices to establish the result for the double $\lambda$-bracket  
before writing the double products using  \eqref{Eq:DLVA-2}. 
By sesquilinearity \eqref{Eq:DA1}, if the double $\lambda$-bracket is well-defined, it is given on generators of $\Jinf\AA$ by 
\begin{equation} \label{Eq:PtoPA-c}
   \dgal{\del^k(a)_\lambda \del^l(b)}=(-\lambda)^k(\del+\lambda)^l(\dgal{a,b})\,, 
\quad a,b\in \AA,\quad k,l\in \Z_{\geq0},
\end{equation}
and then it is extended to $\Jinf\AA$ through the left and right Leibniz rules \eqref{Eq:DL}--\eqref{Eq:DR}. In order to have a well-defined $\lambda$-bracket, we have to check compatibility of the expression \eqref{Eq:PtoPA-c} with the defining relations \eqref{Eq:Jinf-Rel} of $\Jinf\AA$. It is clearly compatible with the first relation of linearity since the double bracket $\dgal{-,-}$ on $\AA$ is linear. 
For the second relation, we compute for any $v\in \Jinf(\AA)$, $a,b\in \AA$ and $l\geq0$,
\begin{equation} \label{Eq:PtoPA-d}
    \begin{aligned} 
\dgal{v_\lambda \del^l(ab)} 
&=\sum_{j=0}^l \binom{l}{j}  \dgal{v_\lambda \del^j(a) \del^{l-j}(b)} \\ 
&=\sum_{j=0}^l \binom{l}{j} \del^{l-j}(a) (\lambda+\del)^{j}(\dgal{v_\lambda b}) 
+ \sum_{j=0}^l \binom{l}{j} (\lambda+\del)^j(\dgal{v_\lambda a }) \del^{l-j}(b)  \\
&=\sum_{i=0}^{l} \sum_{j=i}^l\binom{l}{j}\binom{j}{i} \left(\del^{l-j}(a) \del^{j-i}(\dgal{v_\lambda b})  
+ \del^{j-i}(\dgal{v_\lambda a }) \del^{l-j}(b) \right) \lambda^i \,,
\end{aligned}
\end{equation}
after successively using the second relation in \eqref{Eq:Jinf-Rel}, the left Leibniz rule \eqref{Eq:DL},  then sesquilinearity \eqref{Eq:DA1}.  Meanwhile, we compute 
\begin{equation} \label{Eq:PtoPA-e}
   \begin{aligned}
\dgal{v_\lambda \del^l(ab)}     
&=\sum_{i=0}^l \binom{l}{i} \del^{l-i}(a\dgal{v_\lambda b}+\dgal{v_\lambda a}b)  \lambda^i \\
&=\sum_{i=0}^{l} \sum_{j'=0}^{l-i}\binom{l}{i}\binom{l-i}{j'} \left(\del^{j'}(a) \del^{l-j'-i}(\dgal{v_\lambda b})  
+ \del^{l-j'-i}(\dgal{v_\lambda a }) \del^{j'}(b) \right) \lambda^i\,, 
\end{aligned} 
\end{equation}
after successively using sesquilinearity \eqref{Eq:DA1}, the left Leibniz rule \eqref{Eq:DL}, then the fact that $\del\in \Der(\AA^{\otimes 2})$. 
After setting $j=l-j'$ in the last expression of \eqref{Eq:PtoPA-e}, we recover the one of  \eqref{Eq:PtoPA-d}. 
An easier computation yields that  $\dgal{ \del^l(ab)_\lambda v}$ is independent of the way we compute it; 
this also follows from our previous computation and the skewsymmetry property.

We prove that skewsymmetry \eqref{Eq:DA2} and the Jacobi identity \eqref{Eq:DA3} hold. 
Thanks to the Leibniz rules, it suffices to check skewsymmetry on generators of $\Jinf\AA$. 
This is direct by looking at \eqref{Eq:PtoPA-c} as this identity implies 
\begin{align*}
-\big|_{x=\del}\dgal{\del^l(b)_{-\lambda-x} \del^k(a)}^\circ  
=-(-\lambda)^k(\lambda+\del)^l(\dgal{b,a})^\circ
=(-\lambda)^k(\lambda+\del)^l(\dgal{a,b})
\end{align*} 
by the skewsymmetry \eqref{Eq:DA} of the double Poisson bracket. 
For the Jacobi identity \eqref{Eq:DA3}, we remark by \cite[Lemma 3.4]{DSKV} that we only need to check this identity on the elements of $\AA\subset \Jinf\AA$. 
Due to \eqref{Eq:PtoPA-a}, it is easy to see that 
\begin{align*}
&\dgal{a_\lambda \dgal{b_\mu c}}_L
-\dgal{b_\mu \dgal{a_\lambda c}}_R
-\dgal{\dgal{a_\lambda b}_{\lambda+\mu}c}_L \\
=&\dgal{a, \dgal{b, c}}_L
-\dgal{b, \dgal{a, c}}_R
-\dgal{\dgal{a, b},c}_L\,,
\end{align*}
and this is zero due to the Jacobi identity \eqref{Eq:DJ} of the double Poisson bracket. 
\end{proof}

\begin{proposition} \label{Pr:DPAtoDPA}
There exists a functor $\mathtt{J}\colon \DPA\to \DPVA$ which associates to a double Poisson algebra $\AA$ its jet algebra $\Jinf\AA$ equipped with the structure from Lemma \ref{Lem:DPAtoDPVA}.   
\end{proposition}
\begin{proof} 
Clearly there is a unique way to extend a double Poisson morphism $\theta\colon \AA_1\to\AA_2$ as a morphism of differential algebras $\Jinf\AA_1\to\Jinf\AA_2$, that we still abusively denote by $\theta$. 
We need to prove that it is actually a morphism of double Poisson vertex algebras with respect to the structures $\dgal{-_\lambda-}_i$ on $\AA_i$ defined in Lemma~\ref{Lem:DPAtoDPVA}. Consider $a,b\in\AA_1$ and $k,l\geq0$. Since $\theta$ commutes with $\del$ and $\lambda$, and extends a morphism of double Poisson algebras, we have
\begin{align*}
    \theta^{\otimes2}\dgal{\del^ka_\lambda\del^lb}_1 
    =\theta^{\otimes2}(-\lambda)^k(\del+\lambda)^l\dgal{a,b}_1  
    =\dgal{\theta\del^k a_\lambda\theta\del^l b}_2
\end{align*}
as wished.
\end{proof}

\subsection{Going to representation spaces}

Denote by $\PA$ and $\DPA$ the categories of Poisson and Poisson vertex algebras over $\kk$. 
For every $N\geq0$ we have representation functors
\begin{equation} \label{Eq:RepFunct}
    (-)_N\colon \DPA\to\PA\qquad\text{and}\qquad(-)_N\colon \DPVA\to\PVA\,.
\end{equation}
The functors $\mathtt{J}\colon \DPA\to \DPVA$ and $\mathtt{Q}\colon \DPVA\to \DPA$ have commutative analogues, see \S\ref{ss:Remind}. 
%, see \S\ref{ss:Com-Functor}.

\begin{theorem} \label{Thm:QJ-Rep}
The following diagrams commute: 
\begin{center}
\begin{tikzpicture}
 \node  (TopLeft) at (-1,1) {$\DPVA$};
 \node  (TopRight) at (1,1) {$\DPA$};
 \node  (BotLeft) at (-1,-1) {$\PVA$};
 \node  (BotRight) at (1,-1) {$\PA$};
\path[->,>=angle 90,font=\small]  
   (TopLeft) edge node[above] {$\mathtt{Q}$}  (TopRight) ;
   \path[->,>=angle 90,font=\small]  
   (BotLeft) edge node[above] {$\mathtt{Q}$}  (BotRight) ;
\path[->,>=angle 90,font=\small]  
   (TopLeft) edge node[left] {$(-)_N$}  (BotLeft) ;
\path[->,>=angle 90,font=\small]  
   (TopRight) edge node[right] {$(-)_N$}  (BotRight) ;
   \end{tikzpicture}
   \qquad
   \begin{tikzpicture}
   \node at (0,0) {and};
   \node at (0,-1.2) {~};
      \end{tikzpicture}
\qquad
   \begin{tikzpicture}
 \node  (TopLeft) at (-1,1) {$\DPVA$};
 \node  (TopRight) at (1,1) {$\DPA$};
 \node  (BotLeft) at (-1,-1) {$\PVA$};
 \node  (BotRight) at (1,-1) {$\PA$};
\path[->,>=angle 90,font=\small]  
   (TopRight) edge node[above] {$\mathtt{J}$}  (TopLeft) ;
   \path[->,>=angle 90,font=\small]  
   (BotRight) edge node[above] {$\mathtt{J}$}  (BotLeft) ;
\path[->,>=angle 90,font=\small]  
   (TopLeft) edge node[left] {$(-)_N$}  (BotLeft) ;
\path[->,>=angle 90,font=\small]  
   (TopRight) edge node[right] {$(-)_N$}  (BotRight) ;
   \end{tikzpicture}
\end{center}
\end{theorem}
The theorem will be proved in \S\ref{Proof_Thm:QJ-Rep}. 
As a preparation, we check that $(-)_N\colon \DPVA\to\PVA$ is a functor;  
one gets the functoriality of $(-)_N\colon \DPA\to\PA$ similarly. 

The functor is defined on objects through Theorem \ref{Thm:RepdPV}. (Below we use the language of $\lambda$-brackets, but the reader could start with Theorem \ref{Thm:RepdLVA} instead.) 
Given a morphism in $\DPVA$ 
$$\theta\colon (\VV_1,\del_1,\dgal{-_\lambda-}_1)\longrightarrow (\VV_2,\del_2,\dgal{-_\lambda-}_2)\,,$$
we directly get a unique morphism of algebras $\theta_N\colon (\VV_1)_N \to (\VV_2)_N$ satisfying $\theta_N(a_{ij})=\theta(a)_{ij}$ for any $a\in \VV_1$ and $1\leq i,j\leq N$. 
This is a morphism of differential algebras since 
$$\theta_N(\del_1(a_{ij})\!)=\theta_N(\!(\del_1(a)\!)_{ij})
=(\theta\circ \del_1(a)\!)_{ij}=(\del_2\circ \theta(a)\!)_{ij}=\del_2(\theta(a)_{ij})=\del_2(\theta_N(a_{ij})\!).$$ 
Finally, %using the notation \eqref{Eq:Nota-Rep}, 
%To check the different axioms on generators of $\VV_N$, we introduce the following notation: for any $a,b,c\in \VV$ and $1\leq i,j,k,l,u,v\leq N$, we set 
setting 
\begin{equation*}  %\label{Eq:Nota-Rep}
(a\otimes b)_{ij;kl}:=a_{ij} b_{kl}\in (\VV_1)_N\,, %\quad 
%(a\otimes b\otimes c)_{ij;kl;uv}:=a_{ij} b_{kl} c_{uv}\in (\VV_1)_N\,, 
\end{equation*}
we have for generators $a_{ij},b_{kl}\in (\VV_1)_N$ that \eqref{Eq:relDPVA} reads 
$\br{a_{ij}{}_\lambda b_{kl}}_1 =(\dgal{a_\lambda b}_1)_{kj;il}$. Thus  
\begin{align*}
\theta_N(\br{a_{ij}{}_\lambda b_{kl}}_1)    
&=\theta_N(\!(\dgal{a_\lambda b}_1)_{kj;il})  
=(\theta^{\otimes 2}(\dgal{a_\lambda b}_1)\!)_{kj;il}  \\
&=(\dgal{\theta(a)_\lambda \theta(b)}_2)_{kj;il}  
=\br{\theta(a)_{ij}{}_\lambda \theta(b)_{kl}}_2  
=\br{\theta_N(a_{ij}){}_\lambda \theta_N(b_{kl})}_2\,, 
\end{align*}
where in the second and fifth equalities we used the definition of $\theta_N$ while in the third we used that $\theta$ is a morphism in $\DPVA$. 
By the Leibniz rules, we get $\theta_N(\br{u_\lambda v}_1)=\br{\theta_N(u)_\lambda \theta_N(v)}_2$ for any two elements $u,v\in (\VV_1)_N$.

\subsection{Proof of Theorem \ref{Thm:QJ-Rep}}
\label{Proof_Thm:QJ-Rep}
Consider a commutative differential algebra $V$. 
Recall from~\cite[(2.15)]{BKR} and references therein that $(-)_N\colon \kk\mathtt{Alg}\to\kk\mathtt{CommAlg}$ is left adjoint to $-\otimes_\kk M_N(\kk)$, where $M_N$ denotes the algebra of square matrices of size $N$.
We omit the reference to $\kk$ in the sequel. We have by universality\begin{align*}
   \Hom_{\mathtt{CommDiffAlg}}(\Jinf(\AA_N),V)&=\Hom_{\mathtt{CommAlg}}(\AA_N,V)\\
   &=\Hom_{\mathtt{Alg}}(\AA,V\otimes M_N)\\
   &=\Hom_{\mathtt{DiffAlg}}(\Jinf\AA,V\otimes M_N)\,,
\end{align*}
thus we have an isomorphism $\Phi_\AA\colon \Jinf(\AA_N)\simeq (\Jinf(\AA)\!)_N$ of differential algebras. 
After picking a presentation of $\Jinf(\AA_N)$ with generators $\del^k(a_{ij})$ where $k\geq 0$, $a\in A$, $1\leq i,j\leq N$, and similarly a presentation of $(\Jinf(\AA)\!)_N$ with generators $(\del^k(a)\!)_{ij}$, the identification then becomes $\Phi_\AA(\del^k(a_{ij})\!)=(\del^k(a)\!)_{ij}$. 
Next, we have to check that $\Phi_\AA$ is compatible with the induced $\lambda$-brackets. This directly follows by applying $\Phi_\AA$ to 
\begin{align*}
\br{\del^k(a_{ij})_\lambda \del^\ell(b_{kl})}
=&(-\lambda)^k (\lambda+\del)^\ell \br{a_{ij}, b_{kl}} && (\text{by Lemma \ref{Lem:PAtoPVA}}), \\
=&(-\lambda)^k (\lambda+\del)^\ell \dgal{a,b}_{kj;il}  && (\text{by Theorem \ref{Thm:RepdP}}),
\end{align*}
and comparing with 
\begin{align*}
\br{(\del^k(a)\!)_{ij}{}_\lambda (\del^\ell(b)\!)_{kl}}
=& \Big(\dgal{\del^k(a){}_\lambda \del^\ell(b)}\Big)_{kj;il}&& (\text{by Theorem \ref{Thm:RepdPV}}), \\
=&\Big(\!(-\lambda)^k (\lambda+\del)^\ell \dgal{a,b}\Big)_{kj;il}  && (\text{by Lemma \ref{Lem:DPAtoDPVA}}).
\end{align*}
In particular, it means that $\Phi_\AA\colon \Jinf(\AA_N)\simeq (\Jinf(\AA)\!)_N$ is an isomorphism of Poisson vertex algebras.
To conclude that the right hand square commutes, it remains to check that a morphism $\theta\colon \AA_1\to \AA_2$ in $\DPA$ induces isomorphic morphisms, under $\mathtt{J}\circ (-)_N$ and $(-)_N\circ \mathtt{J}$. Denoting both morphisms $\theta_N$ by an abuse of notations, this follows from the fact that 
$$\Phi_{\AA_2}\circ \theta_N =  \theta_N \circ \Phi_{\AA_1}\colon  \Jinf(\!(\AA_1)_N)\longrightarrow (\Jinf(\AA_2)\!)_N$$ 
is a morphism in $\PVA$ which can be obtained from our previous computations. 

%\medskip 

Commutativity of the left-hand square is obtained in the exact same way. 
Using universality from Lemma \ref{Lem:Univ-Del}, we get for any differential algebra $\VV$ an isomorphism $\VV_N/\langle \del \rangle \simeq (\VV /\langle \del \rangle)_N$ in $\mathtt{CommAlg}$. 
To get the compatibility with the Poisson algebra structures, we need to apply Theorem \ref{Thm:RepdPV} then Lemma~\ref{Lem:PVAtoPA} to define the Poisson bracket on $\VV_N/\langle \del \rangle$, while we need Lemma~\ref{Lem:Univ-Del} and then Theorem \ref{Thm:RepdP} to define the Poisson bracket on $(\VV /\langle \del \rangle)_N$.

\begin{example}
Consider the double Poisson algebra $\AA=\kk[a]$ as in Example \ref{Ex:ku} 
with $\alpha=1$ and $\beta=\gamma=0$. 
After taking jets, $\Jinf \AA_N\simeq (\Jinf \AA)_N$ is 
$\kk[a_{ij}^{(\ell)} \mid \ell \geq 0,\, 1\leq i,j\leq N]$ equipped with the derivation $\del\colon  a_{ij}^{(\ell)}\mapsto a_{ij}^{(\ell+1)}$. 
Using \eqref{Eq:KKS}, the Poisson vertex algebra structure on $\Jinf \AA_N$ is obtained from the $\lambda$-bracket uniquely determined by 
\begin{equation} \label{Eq:KKS-J}
    \br{ a_{ij} {}_\lambda a_{kl} }
=a_{kj}  \delta_{il} - \delta_{kj} a_{il}\,.
\end{equation}
This $\lambda$-bracket is induced by Theorem \ref{Thm:RepdPV} from the double Poisson vertex algebra structure on $\Jinf \AA$ given in Example \ref{Ex:kuPVA} with $\epsilon=0$. 
(Note that applying the functor $\mathtt{Q}$ to Example \ref{Ex:kuPVA} for any $\epsilon\in \kk$ yields Example \ref{Ex:ku} with $\alpha=1$ and $\beta=\gamma=0$.) 
As described in \S\ref{sss:RepAlg}, both $\AA_N$ and $\Jinf \AA_N$ are equipped with a $\Gl_N$-action given on the latter by 
\begin{equation*}  
  g\cdot  a_{ij}^{(\ell)} =\sum_{k,l=1}^N (g^{-1})_{ik} a_{kl}^{(\ell)} g_{lj}\,,\qquad 
  g\in \Gl_N,\,\, 1\leq i,j\leq N,\,\, \ell \geq 0\,.
\end{equation*}
As it will be observed in %in \S\ref{ss:NaiveInv}, 
in \S\ref{ss:RedPVA}, 
it is more appropriate for our purpose 
to consider the action of $\Jinf\Gl_N$ instead of that of $\Gl_N$.   
%In the next section, we first study in a more general context the action of 
%the jet scheme $\Jinf \GG$ on the jet algebra 
%$\Jinf A$ induced by an 
%action of an algebraic group $\GG$ on a commutative algebra $A=\Spec Y$.  
\end{example}

%%%%%%%%%%%%%%%%%% NEW SECTION %%%%%%%%%%%%%%%%%%%%%
%%%%%%%%%%%%%%%%%% NEW SECTION %%%%%%%%%%%%%%%%%%%%%
%%%%%%%%%%%%%%%%%% NEW SECTION %%%%%%%%%%%%%%%%%%%%%
%%%%%%%%%%%%%%%%%% NEW SECTION %%%%%%%%%%%%%%%%%%%%%

\section{Invariants and Poisson vertex algebras} 
\label{Sec:Inv_PVA}

In this section, we assume that $\kk$ is algebraically closed and of characteristic $0$.

\subsection{Jets of an action and their invariants} 
\label{ss:JetInv}
We start by reviewing some constructions related to jet algebras/schemes that can be found e.g. in \cite[Chapter~1]{AM}. 
Let $A$ be a unital commutative algebra of finite type over $\kk$, and write $Y=\Spec(A)$. 
Let $\GG$ be an affine algebraic group with a left action $\GG\times Y\to Y$. 
Dually, we have a morphism of algebras 
\begin{equation} \label{Eq:rhoV}
  \rho\colon   A\longrightarrow \kk[\GG]\otimes A\,,
\end{equation}
which uniquely gives rise to a morphism of differential algebras  
\begin{equation} \label{Eq:rhoInf}
  \rho_\infty \colon   \Jinf(A)\longrightarrow \Jinf(\kk[\GG])\otimes \Jinf(A)\,.
\end{equation}
I.e., the (ind-)group scheme $\Jinf(\GG):=\Spec \Jinf(\kk[\GG])$ acts on $\Jinf (Y)$.  
We get a left action of $\GG$ on $A$ by $(g\cdot \tilde{F})(y):=\tilde{F}(g^{-1}\cdot y)$ for $g\in \GG$, $\tilde{F}\in A$, $y\in Y$, and similarly we have a left action of $\Jinf(\GG)$ on $\Jinf(A)$.
Hence, we can define the subalgebras $A^\GG\subset A$ 
and $\Jinf(A)^{\Jinf(\GG)}\subset \Jinf(A)$ of $\GG$- and $\Jinf(\GG)$-invariant elements, respectively. 
Denoting 
$\iota\colon  A \stackrel{\sim}{\longrightarrow} \kk \otimes A \longhookrightarrow  \kk[\GG]\otimes A$, 
we can see that  $A^\GG=\ker(\rho-\iota)$ while $\Jinf(A)^{\Jinf(\GG)}=\ker(\rho_\infty-\iota_\infty)$.  
Since jets of morphisms respect the differential algebra structure on the jet spaces, 
\begin{equation}
 (\rho_\infty-\iota_\infty)(\del^k \tilde{F})  
 =\del^k \big(\!(\rho-\iota)(\tilde{F})\big) =0\,, \qquad \forall\, \tilde{F}\in A^{\GG},\,\, k\geq 0\,.
\end{equation}
Hence the morphism $\Jinf(A^{\GG}) \to \Jinf(A)$ induced by the 
inclusion $A^{\GG} \hookrightarrow A$ factors through the following  natural morphism  of differential algebras 
\begin{equation} \label{Eq:InvMorph}
    j_{A} \colon  \Jinf(A^\GG)\longrightarrow (\Jinf(A)\!)^{\Jinf(\GG)}\,.
\end{equation}
Note that $j_{A}$ is functorial in $A$, seen as an element of the category of differential algebras 
with a $\GG$-action. 
In general, the morphism $j_A$ is neither injective, nor surjective.  
For example, 
if $\GG$ is finite nontrivial and $A$ is a $\GG$-module, then $j_A$ is not surjective 
(\cite[Theorem 3.13.]{LSS});  
if $\GG={\rm SL}_n$ 
and $Y=(\kk^n)^{\oplus p}$ is the direct sum of six copies
of the standard representation, then $j_{\kk[Y]}$ is not injective 
(\cite[Example 6.6]{LSS})\footnote{In this example, the kernel 
of $j_{\kk[Y]}$ %consists of nilpotent elements, 
is precisely the nilradical of $\Jinf (Y/\!\!/\GG)$.}. 
%and $\Jinf (Y/\!\!/\GG)$ is not reduced.}. 
See \cite{LS3} for a generalization, and 
\cite{Ishii} for other interesting counter-examples.

%if $\GG={\rm SL}_3$ 
%and $Y=(\kk^3)^{\oplus 6}$ is the direct sum of six copies
%of the standard representation, then $j_{\kk[Y]}$ is not injective 
%(\cite[Example 6.6]{LSS})\footnote{In this example, the kernel 
%of $j_{\kk[Y]}$ consists of nilpotent elements, and $\Jinf (Y/\!\!/\GG)$ 
%is not reduced.}. 
%See \cite{Ishii} for other interesting counter-examples, and \cite{LS3} for a generalization.

There are, however, interesting examples where $j_A$ 
is an isomorphism. 

\begin{example}  
\label{ex:LSS_quiver}
If $Y = (\kk^{N})^{\oplus p} \oplus ((\kk^{N})^*)^{\oplus q}$ 
 is a sum of $p$ copies of the standard module and $q$ 
 copies of its dual,
then $j_{\kk[Y]}$ is an isomorphism when  $\GG=\Gl_{N}$, due to the main result of \cite{LSS}. 
We will see in \S\ref{ss:SemiS-example}
%\S\ref{sub:quiver_LSS_example} 
an interpretation 
in term of quivers of this example.
%If $\GG={\rm SL}_{N}$, it is established in \cite{LS3} that $j_{\kk[Y]}$ is always surjective, 
%and that when it fails to
%be injective, its kernel is the nilradical of $\Jinf(\kk[Y]^\GG)$, which is described in all cases.
\end{example}

We refer to \cite{LSS,LS1,LS2,LS3} for other examples where the morphism 
\eqref{Eq:InvMorph} is an isomorphism  
in the case where 
$\GG={\rm SL}_N, \Gl_N, {\rm SO}_N$ or ${\rm Sp}_{2N}$, 
and $Y$ is a finite-dimensional $\GG$-module.
%such that $Y/\!\!/ \GG$ is smooth or a complete intersection. 
Note that $Y=(\kk^{2N})^{\oplus p}$, $\GG={\rm Sp}_{2N}$ 
provides examples where $Y/\!\!/ \GG$ is neither smooth nor a 
complete intersection.

\begin{example}
\label{Ex:inv_adjoint}
Let $\GG$ be a connected reductive group with Lie algebra $\g$. 
The group $\GG$ acts on $\g$ and its dual by 
the adjoint and the coadjoint action, respectively. 
It is well known since the works of Kostant 
\cite{Kos} that $\g^*/\!\!/\GG$ is an affine space 
of dimension the rank of $\g$, that is, 
$\kk[\g^*]^\GG$ is a polynomial algebra in the rank of $\g$ variables. 
By a deep result of Ra\"{i}s and Tauvel \cite{RT}, 
obtained independently by 
Beilinson and Drinfeld \cite{BD1} (see also \cite[Appendix]{Mu}), 
we have 
$$(\Jinf \kk[\g^*])^{\Jinf \GG} \cong \Jinf \kk[\g^*]^{\GG}.$$ 
\end{example}

\subsection{Structure of invariants for Poisson vertex algebras}
Assume from now on that $A$ is equipped with a Poisson bracket and 
that $\GG$ acts by Poisson automorphisms.  
Note that the jet algebra $\Jinf(A^\GG)$ is a Poisson vertex subalgebra of $\Jinf(A)$. 
Indeed, we compute that   
$\br{\del^r(\tilde{F}) _\lambda \del^s(\tilde{G}) }  \in \Jinf(A^\GG)[\lambda]$
for any $\tilde{F},\tilde{G}\in A^{\GG}$ and $r,s\geq 0$ due to the Leibniz rules and the fact that $\br{\tilde{F}_\lambda \tilde{G} }=\br{\tilde{F},  \tilde{G} }$ is $\GG$-invariant. 
The following theorem guarantees that $(\Jinf A)^{\Jinf \GG}$ is also a Poisson vertex subalgebra of $\Jinf(A)$, and therefore the morphism $j_A$ \eqref{Eq:InvMorph} intertwines the $\lambda$-brackets.

\begin{theorem}
\label{Th:invariants_are_PVA}
Let $A$ be a unital commutative Poisson algebra of finite type over $\kk$,  
and $\GG$ an affine algebraic group  
which acts by Poisson automorphisms on $A$. 
Then the invariant algebra $(\Jinf A)^{\Jinf \GG}$ 
is a Poisson vertex subalgebra of $\Jinf A$. 
Moreover, the morphism $j_A \colon \Jinf (A^{\GG}) \to (\Jinf A)^{\Jinf \GG}$ 
is a  Poisson vertex algebra morphism. 
\end{theorem}

 \begin{remark}
 	We want to point out that it is of current interest to determine when arc spaces are reduced, and if not, to describe the reduced scheme structure 
	-- see for example \cite{FM,Mak,DFMM}. Under the hypotheses of the above theorem, the nilradical $\mathcal N$ of $\Jinf (A^{\GG})$ is a Poisson vertex ideal, and thus the reduced ring $\Jinf (A^{\GG})/\mathcal N$ is a Poisson vertex algebra. If $\Jinf A$ is reduced, then so is $(\Jinf A)^{\Jinf \GG}$, and $\mathcal N$ lies in the kernel of $j_A$.
 \end{remark}

The rest of the section is devoted to the proof of Theorem \ref{Th:invariants_are_PVA}. 

\subsubsection{} 
Assume first that $\GG$ is connected, and let $\g$ be the Lie algebra 
of $\GG$. 
Then $\Jinf \g= \g\[[t\]]$ is the Lie algebra of $\Jinf \GG$ 
(see e.g.~\cite{AM}), and connectedness implies that 
$$(\Jinf A)^{\Jinf \GG} = (\Jinf A)^{\Jinf \g},$$
where 
$$(\Jinf A)^{\Jinf \g} = (\Jinf A)^{\g\[[t\]]} 
= \{a \in \Jinf A \mid x_{(k)} a = 0 \text{ for all } x \in \g, \, k \in \Z_{\geq 0} \},$$
with $x_{(k)} a := (xt^k).a$ given by the infinitesimal action of $xt^k\in \g\[[t\]]$ on $\Jinf A$. 
Since $\GG$ acts as Poisson automorphisms on $A$, 
the Lie algebra $\g$ acts on $A$ 
by derivations for both the commutative and the Poisson algebra 
structures on $A$. 
Let $Y:=\Spec A$. The left action 
$\GG\times Y\to Y$ induces 
a left action,  
$\Jinf \GG \times \Jinf Y \to \Jinf Y,$ 
and so an algebra morphism 
$$ \Jinf (A) \longrightarrow \Jinf (\kk[\GG]) \otimes \Jinf (A).$$
Therefore, the Lie algebra $\g\[[t\]]$ acts on $\Jinf A$ by derivations 
for the commutative algebra structure. 
The action of $\g\[[t\]]$  
on $\Jinf A$ 
is entirely determined by the action of $\g$ on $A$ as follows: 
\begin{align}
\label{eq:Id0}
x_{(k)}(\partial^l a) = \begin{cases}
\dfrac{l!}{(l-k)!} \partial^{l-k} (x.a) &\text{if } k\leq l,\\
0 & \text{otherwise,}
\end{cases}
\end{align}
for $k,l\geq 0$ and $a \in A$. In particular, 
$x_{(0)}a=x.a$.
We start by establishing a number of identities. 

\begin{lemma} 
\label{Lem:identities} 
Let $a,b \in \Jinf A$, $x \in \g$ and 
$n,k,i \in \Z_{\geq 0}$. 
\begin{enumerate}  
\item $(\partial^{i} a)_{(n)} b  = \delta_{i \leq n} 
(-1)^{i} \dfrac{n!}{(n-i)!} a_{(n-i)} b,$ 
\item  $x_{(k)} (\partial a) = \partial (x_{(k)}a) + k x_{(k-1)}a,$   
\item  $x_{(k)}( \partial^{i} a) = \partial^{i} (x_{(k)} a) + 
\sum\limits_{\ell=1}^{i}  
\begin{pmatrix} i \\
\ell 
\end{pmatrix}    
\dfrac{k!}{(k-\ell)!}
\partial^{i-\ell }x_{(k-\ell)} a,$
\item $x_{(k)}( (\partial^{i} a)_{(n)} b)  = 
(\partial^{i} a)_{(n)} (x_{(k)} b) 
+ \delta_{i \leq n} (-1)^{i} \dfrac{n!}{(n-i)!} (x_{(0)}a)_{(n-i+k)}b,$ $(a \in A)$,  
\end{enumerate}
where $\delta_{i \leq j} =1$ if $i \leq j$ and is zero otherwise. 
As a rule, our convention is that $x_{(n)}b=0$, or $a_{(n)}b=0$, whenever $n<0$. 
So for instance, one has omitted $\delta_{\ell \leq k}$ 
in the identity (3). 
\end{lemma}

\begin{proof}
In order to prove the above identities, one can assume that  
\begin{align}
\label{eq:ab_form}
a = \partial^{i_1}a_1\ldots \partial^{i_r}a_r
\quad \text{ and }\quad b = \partial^{j_1}b_1 \ldots \partial^{j_s}b_s
\end{align}
with $i_\ell,j_m \geq 0$ and $a_i, b_j \in A$. 
Part (1) is a direct induction from the identity \eqref{Eq:A1-b}.

(2)  
Since $x_{(k)}$ acts as a derivation, we have by \eqref{eq:Id0}: 
\begin{align*} 
x_{(k)} a & =
 \sum_{\ell=1}^r \partial^{i_1}a_1\ldots x_{(k)}(\partial^{i_\ell}a_\ell)\ldots \partial^{i_r}a_r 
  = \sum_{\ell=1}^r \partial^{i_1}a_1\ldots \delta_{k \leq i_{\ell}} 
\dfrac{i_\ell !}{(i_\ell-k)!} \partial^{i_\ell-k} (x.a_\ell)\ldots \partial^{i_r}a_r.&
\end{align*}
Differentiating the above identity, we get 
\begin{align*} 
\partial (x_{(k)} a) & = \sum_{\ell=1}^r 
\sum_{1\leq m \not=\ell \leq r}  \partial^{i_1}a_1\ldots \partial^{i_m+1}a_m 
\ldots \delta_{k \leq i_{\ell}} 
\dfrac{i_\ell !}{(i_\ell-k)!} \partial^{i_\ell-k} (x.a_\ell)\ldots \partial^{i_r}a_r & \\
& \qquad \quad + \sum_{\ell=1}^r \partial^{i_1}a_1\ldots \delta_{k \leq i_{\ell}} 
\dfrac{i_\ell !}{(i_\ell-k)!} \partial^{i_\ell+1-k} (x.a_\ell)\ldots \partial^{i_r}a_r .&
\end{align*} 
On the other hand, we have: 
\begin{align*} 
x_{(k)} (\partial a) &= \sum_{\ell=1}^r 
\sum_{1\leq m \not=\ell \leq r}  \partial^{i_1}a_1\ldots \partial^{i_m+1}a_m 
\ldots \delta_{k \leq i_{\ell}} 
\dfrac{i_\ell !}{(i_\ell-k)!} \partial^{i_\ell-k} (x.a_\ell)\ldots \partial^{i_r}a_r & \\
& \qquad \quad + \sum_{\ell=1}^r \partial^{i_1}a_1\ldots \delta_{k \leq i_{\ell}+1} 
\dfrac{(i_\ell+1) !}{(i_\ell+1-k)!} \partial^{i_\ell+1-k} (x.a_\ell)\ldots \partial^{i_r}a_r.&
\end{align*}
We can now prove the identity (2).
For $k=0$, 
\begin{align*}
x_{(0)} (\partial a)= & \sum_{\ell=1}^r 
\sum_{1\leq m \not=\ell \leq r}  \partial^{i_1}a_1\ldots \partial^{i_m+1}a_m 
\ldots  \partial^{i_\ell} (x.a_\ell)\ldots \partial^{i_r}a_r & \\
& \qquad + \sum_{\ell=1}^r \partial^{i_1}a_1\ldots 
 \partial^{i_\ell+1} (x.a_\ell)\ldots \partial^{i_r}a_r=\partial (x_{(0)} a).&
\end{align*}
whence (2) for $k=0$. For $k \geq 1$, 
\begin{align*}
x_{(k)} (\partial a)= & \quad 
\sum_{\ell=1}^r 
\sum_{1\leq m \not=\ell \leq r}  \partial^{i_1}a_1\ldots \partial^{i_m+1}a_m 
\ldots \delta_{k \leq i_{\ell}} 
\dfrac{i_\ell !}{(i_\ell-k)!} \partial^{i_\ell-k} (x.a_\ell)\ldots \partial^{i_r}a_r & \\
& \qquad \quad + \sum_{\ell=1}^r 
(i_\ell +1) \partial^{i_1}a_1\ldots \delta_{k \leq i_{\ell}+1} 
\dfrac{ i_\ell  !}{(i_\ell+1-k)!} \partial^{i_\ell+1-k} (x.a_\ell)\ldots \partial^{i_r}a_r &
\end{align*}
Hence, 
\begin{align*}
x_{(k)} (\partial a)= & \quad 
\sum_{\ell=1}^r 
\sum_{1\leq m \not=\ell \leq r}  \partial^{i_1}a_1\ldots \partial^{i_m+1}a_m 
\ldots \delta_{k \leq i_{\ell}} 
\dfrac{i_\ell !}{(i_\ell-k)!} \partial^{i_\ell-k} (x.a_\ell)\ldots \partial^{i_r}a_r & \\
& \qquad \quad + \sum_{\ell=1}^r 
(i_\ell +1-k) \partial^{i_1}a_1\ldots \delta_{k \leq i_{\ell}+1} 
\dfrac{ i_\ell  !}{(i_\ell+1-k)!} \partial^{i_\ell+1-k} (x.a_\ell)\ldots \partial^{i_r}a_r &\\
& \qquad \quad + 
k  \sum_{\ell=1}^r  \partial^{i_1}a_1\ldots \delta_{k \leq i_{\ell}+1} 
\dfrac{ i_\ell  !}{(i_\ell+1-k)!} \partial^{i_\ell+1-k} (x.a_\ell)\ldots \partial^{i_r}a_r  
.&
\end{align*}
Note that the last term of the sum is nothing but $x_{(k-1)} a$. 
If $k \not\in\{i_1+1,\ldots,i_r+1\}$, we get 
\begin{align*}
x_{(k)} (\partial a)= & \quad 
\sum_{\ell=1}^r 
\sum_{1\leq m \not=\ell \leq r}  \partial^{i_1}a_1\ldots \partial^{i_m+1}a_m 
\ldots \delta_{k \leq i_{\ell}} 
\dfrac{i_\ell !}{(i_\ell-k)!} \partial^{i_\ell-k} (x.a_\ell)\ldots \partial^{i_r}a_r & \\
& \qquad  + \sum_{\ell=1}^r 
\partial^{i_1}a_1\ldots \delta_{k \leq i_{\ell}} 
\dfrac{ i_\ell  !}{(i_\ell-k)!} \partial^{i_\ell+1-k} (x.a_\ell)\ldots \partial^{i_r}a_r 
+ k (x_{(k-1)} a) & \\
& = \partial (x_{(k)} a) + k (x_{(k-1)} a).&  
\end{align*}
If $k =i_j+1$ for some $j \in \{1,\ldots,r\}$, then 
\begin{align*}
x_{(k)} (\partial a) & 
=  \quad  
\sum_{\ell=1}^r 
\sum_{1\leq m \not=\ell \leq r}  \partial^{i_1}a_1\ldots \partial^{i_m+1}a_m 
\ldots \delta_{k \leq i_{\ell}} 
\dfrac{i_\ell !}{(i_\ell-k)!} \partial^{i_\ell-k} (x.a_\ell)\ldots \partial^{i_r}a_r & \\
& \;  + \sum_{1\leq \ell \leq r \atop 
i_\ell \not= i_j}  
\partial^{i_1}a_1\ldots \delta_{k \leq i_{\ell}} 
\dfrac{ i_\ell  !}{(i_\ell-k)!} \partial^{i_\ell+1-k} (x.a_\ell)\ldots \partial^{i_r}a_r 
+ k (x_{(k-1)} a) &\\
& \;  +\sum_{1\leq \ell \leq r \atop 
i_\ell=i_j} 
\underbrace{(i_\ell +1-k)}_{=0} \partial^{i_1}a_1\ldots \delta_{k \leq i_{\ell}+1} 
\dfrac{ i_\ell  !}{(i_\ell+1-k)!} \partial^{i_\ell+1-k} (x.a_\ell)\ldots \partial^{i_r}a_r &\\
& 
= \partial (x_{(k)} a) + k (x_{(k-1)} a)   - \sum_{1 \leq \ell \leq r\atop 
i_\ell=i_j} \partial^{i_1}a_1\ldots \delta_{k \leq i_{\ell}}
\dfrac{i_\ell !}{(i_\ell-k)!} \partial^{i_\ell+1-k} (x.a_\ell)\ldots \partial^{i_r}a_r &\\
& = \partial (x_{(k)} a) + k (x_{(k-1)} a).&
\end{align*}
Indeed, $\delta_{k \leq i_{\ell}}=0$ for $i_\ell=i_j$ since $k=i_j+1$, 
whence the identity (2) for $k\geq 1$.

(3) 
We argue by induction on $i$, the case $i=1$ being proved by (2).
We have: 
\begin{align*}
x_{(k)}( \partial^{i+1} a)   
& = 
\sum\limits_{\ell=0}^{i} 
\begin{pmatrix} i \\
\ell 
\end{pmatrix}   
\dfrac{k!}{(k-\ell)!}
\partial^{i-\ell} 
( \partial  x_{(k-\ell)} a + (k-\ell) x_{(k-\ell-1)} a )& \\
& = 
\sum\limits_{\ell=0}^{i} 
\begin{pmatrix} i \\
\ell 
\end{pmatrix}   
\dfrac{k!}{(k-\ell)!}
\partial^{i-\ell+1} 
  x_{(k-\ell)} a 
+ 
\sum\limits_{\ell=1}^{i+1} 
\begin{pmatrix} i \\
\ell -1 
\end{pmatrix}   
\dfrac{k!}{(k-\ell)!}
\partial^{i+1-\ell} 
x_{(k-\ell)} a &\\
& = 
\sum\limits_{\ell=0}^{i+1} 
\begin{pmatrix} i+1 \\
\ell 
\end{pmatrix}   
\dfrac{k!}{(k-\ell)!}
\partial^{i+1-\ell} 
  x_{(k-\ell)} a 
&
\end{align*}
by Pascal's triangle formula, whence the statement by induction.

(4) 
We start with the case $a=a_1 \in A$, $b=\partial^{j_1}b_1$, that is, $r=s=1$ and $i_1=0$ 
in \eqref{eq:ab_form}.
We have 
\begin{align*}
& x_{(k)} (a_{(n)} \partial^{j_1} b_{1} )
 = \delta_{k \leq j_1-n} \delta_{n \leq j_1} \dfrac{j_1!}{(j_1-n-k)!} \partial^{j_1-n-k} x.\{a,b_1\}& \\
& a_{(n)} (x_{(k)} \partial^{j_1} b_{1} ) 
 = \delta_{n \leq j_1-k} \delta_{k \leq j_1} \dfrac{j_1!}{(j-n-k)!} \partial^{j_1-n-k} \{a,x.b_1\}& \\
& (x_{(k)}a)_{(n)}( \partial^{j_1} b_{1} ) 
 =\delta_{k=0} \delta_{n \leq j_1} \dfrac{j_1!}{(j_1-n)!} \partial^{j_1-n} \{x.a,b_1\}.&
\end{align*}
Since $x.\{a,b_1\} = \{x.a,b_1\} + \{a,x.b_1\}$, 
we get 
$$x_{(k)} (a_{(n)} \partial^{j_1} b_{1} ) = a_{(n)} (x_{(k)} \partial^{j_1} b_{1} ) 
+ (x_{(0)} a)_{(n+k)} (\partial^{j_1}b_1).$$
Assume now that $b=\partial^{j_1}b_1 \ldots \partial^{j_s}b_s$. 
Since $x_{(k)}$, $a_{(n)}$ and $(x_{(0)} a)_{(n+k)}$ 
act as derivations on the product, 
we get by the first case: 
\begin{align*}
x_{(k)} (a_{(n)} b ) & =
\sum_{\ell=1}^{s} \sum_{1 \leq m\not= \ell \leq s} 
\partial^{j_1} b_1 \ldots a_{(n)} \partial^{j_\ell} b_{\ell} 
\ldots  x_{(k)} \partial^{j_m} b_{m}  \ldots \partial^{j_s} b_s&\\ 
& \qquad \quad + \sum_{\ell=1}^s \partial^{j_1} b_1 \ldots x_{(k)} (a_{(n)} \partial^{j_\ell} b_{\ell} ) 
\ldots \partial^{j_s} b_s,& \\
& =
\sum_{\ell=1}^{s} \sum_{1 \leq m\not= \ell \leq s} 
(\partial^{j_1} b_1) \ldots a_{(n)} \partial^{j_\ell} b_{\ell} 
\ldots  x_{(k)} \partial^{j_m} b_{m}  \ldots \partial^{j_s} b_s&\\ 
& \qquad \quad + \sum_{\ell=1}^s
 \partial^{j_1} b_1  \ldots 
\left(a_{(n)} x_{(k)} (\partial^{j_\ell} b_{\ell} ) + (x_{(0)} a)_{(n+k)} (\partial^{j_\ell}b_\ell) 
\right)
\ldots \partial^{j_s} b_s &\\
&  =  a_{(n)} (x_{(k)} b ) 
+(x_{(0)} a)_{(n+k)} b. &
\end{align*}
We can now prove the identity (4) for any $a\in A$ and $i\geq 0$. 
By the above case and Lemma~\ref{Lem:identities}, (1), we get: 
\begin{align*}
x_{(k)} ((\partial^{i}a)_{(n)} b ) 
& =\delta_{i\leq n} (-1)^{i} \dfrac{n!}{(n-i)!} x_{(k)} (a_{(n-i)} b) & \\
& =\delta_{i\leq n} (-1)^{i} \dfrac{n!}{(n-i)!} a_{(n-i)} (x_{(k)} b) 
+  \delta_{i\leq n} (-1)^{i} \dfrac{n!}{(n-i)!} (x_{(0)}a)_{(n-i+k)} b &\\
& =  (\partial^{i}a)_{(n)} (x_{(k)}  b ) 
+ \delta_{i\leq n} (-1)^{i} \dfrac{n!}{(n-i)!} (x_{(0)}a)_{(n-i+k)} b,&
\end{align*}
whence the identity (4). 
\end{proof}

We are now in a position to prove Theorem \ref{Th:invariants_are_PVA} 
in the case where $\GG$ is connected. 

\begin{lemma}
\label{Lem:connected}
If $\GG$ is connected, then $(\Jinf A)^{\Jinf \GG}$ 
is a Poisson vertex subalgebra of $\Jinf A$ under the hypothesis of  
Theorem \ref{Th:invariants_are_PVA}. 
\end{lemma}

\begin{proof}
Assume that $a \in (\Jinf A)^{\g\[[t\]]}$. 
One can assume that $a$ is of the form \eqref{eq:ab_form}. 
Then $x_{(k)}a=0$ for all $k\geq 0$, and so, 
by Lemma \ref{Lem:identities}, (2), we see that $x_{(k)}(\partial a)=0$ 
for any $k \geq 0$ as well, that is $\partial a$ is in $(\Jinf A)^{\g\[[t\]]}$. 
Next, we have to show that $a_{(n)}b$ is in $(\Jinf A)^{\g\[[t\]]}$ for any $n \geq 0$ 
if it is the case for both $a$ and $b$. 
Once again, one can assume that $a,b$ are of the form \eqref{eq:ab_form}. 
Set 
$$
\mathcal{P}_k^{n,a} := x_{(k)}( a_{(n)}b), 
\qquad 
\mathcal{S}_k^{n,a}  :=  a_{(n)} ( x_{(k)} b), 
\qquad 
\mathcal{R}_k^{n,a}  := (x_{(k)}a)_{(n)} b .
$$
Our aim is to prove the following formula for any $a,b \in \Jinf A$ 
as in \eqref{eq:ab_form} (not necessarily $\g\[[t\]]$-invariant): 
\begin{align}
\label{eq:main_ind}
\mathcal{P}_k^{n,a} = \mathcal{S}_k^{n,a} +  
\sum\limits_{\ell=0}^{k} 
\begin{pmatrix}
k \\
\ell
\end{pmatrix} \mathcal{R}_{k-\ell}^{n+\ell,a}.
\end{align} 
(The dependence in $b$ will not play any role in the induction so we 
omit it in the notation.)
If moreover $a,b$ are $\g\[[t\]]$-invariant, then so is  
$a_{(n)}b$ according to \eqref{eq:main_ind}.
We show  the formula \eqref{eq:main_ind} by induction on $r$. 
Remember here that $a$ is of the form \eqref{eq:ab_form} with $r\geq 1$. 

\smallskip

\noindent
$\bullet$ $r=1$. 
By Lemma \ref{Lem:identities}, (4), we have 
\begin{align} 
\label{eq:xab_1}
x_{(k)} ((\partial^{i_1}a_1)_{(n)}b)
=(\partial^{i_1}a_1)_{(n)} (x_{(k)}b) + \delta_{i_1\leq n} (-1)^{i_1} 
\dfrac{n!}{(n-i_1)!} (x_{(0)}a_1)_{(n-i_1+k)} b. 
\end{align}
We can also compute 
\begin{equation}
\label{eq:xab_2}
\begin{aligned}
& \sum\limits_{\ell=0}^{k} 
\begin{pmatrix}
k \\
\ell
\end{pmatrix} (x_{(k-\ell)} \partial^{i_1}a_1)_{(n+\ell)}b 
& \\
& =  
\sum\limits_{\ell=0}^{k} 
\begin{pmatrix}
k \\
\ell
\end{pmatrix} 
\dfrac{ \delta_{k-\ell \leq i_1} \, i_1 !}{(i_1-k+\ell)!}
(-1)^{i_1-k+\ell} 
\dfrac{
\delta_{i_1-k \leq n} \, (n+\ell)!}{(n-i_1+k)!} 
(x_{(0)}a_1)_{(n-i_1+k)}b& \\
& = 
(-1)^{i_1}
 \dfrac{\delta_{i_1\leq k+ n}\, i_1 !}{(n-i_1+k)!}
\sum\limits_{\ell=k-i_1}^{k} 
\begin{pmatrix}
k \\
\ell
\end{pmatrix} 
 \dfrac{(n+\ell)!  \, (-1)^{k-\ell}}{(i_1-k+\ell)!}
(x_{(0)}a_1)_{(n-i_1+k)}b.& 
\end{aligned}
\end{equation}

We claim that \eqref{eq:xab_1} and \eqref{eq:xab_2} also hold if we replace $x\in \g$ by some $c\in A$. Indeed, to establish these identities, the only way $\g$ was involved is through \eqref{eq:Id0} and the fact that $\g$ acts by derivations on the Poisson algebra $A$.   
We can derive \eqref{eq:Id0} for $x=c$ (we write $\br{c,a}$ instead of $x_{(0)}a=x.a$) by induction, and replace the action of $\g$ by derivations of the Poisson bracket of $A$ with the Jacobi identity in $A$; this proves our claim.

By the Jacobi identity \eqref{Eq:A3-b}, we have for $a_1,c\in A$, $b\in \Jinf(A)$ and $i_1\geq 0$ 
\begin{equation}
 c_{(k)}((\del^{i_1}a_1)_{(n)}b) = (\del^{i_1}a_1)_{(n)}(c_{(k)}b)
 +\sum\limits_{\ell=0}^{k} 
\begin{pmatrix} k \\ \ell \end{pmatrix} 
(c_{(k-\ell)} \partial^{i_1}a_1)_{(n+\ell)}b \,,
\end{equation}
which yields after plugging \eqref{eq:xab_2} and comparing with \eqref{eq:xab_1}:
\begin{align} 
\label{eq:xab_3}
&
\left( 
 \dfrac{\delta_{i_1\leq n} \,n !}{(n-i_1)!}
-
 \dfrac{\delta_{i_1\leq k+ n} \,i_1 !}{(n-i_1+k)!}
\sum\limits_{\ell=k-i_1}^{k} 
\begin{pmatrix}
k \\
\ell
\end{pmatrix} 
 \dfrac{(n+\ell)!  \, (-1)^{k-\ell}}{(i_1-k+\ell)!}
\right) 
(c_{(0)}a_1)_{(n-i_1+k)}b=0.& 
\end{align}
Note that this identity holds for any such $a_1,b,c$ and starting with \emph{any} Poisson algebra $A$. 
Hence we can assume that $(c_{(0)}a_1)_{(n-i_1+k)}b\neq 0$, and the leading factor in \eqref{eq:xab_3} vanishes.  

Going back to \eqref{eq:xab_2} in our case of interest, we deduce    
\begin{align*}
 & \sum\limits_{\ell=0}^{k} 
\begin{pmatrix}
k \\
\ell
\end{pmatrix} (x_{(k-\ell)} \partial^{i_1}a_1)_{(n+\ell)}b 
= \delta_{i_1\leq n} (-1)^{i_1} 
\dfrac{n!}{(n-i_1)!} (x_{(0)}a_1)_{(n-i_1+k)} b \,.&  
\end{align*} 
We conclude that 
\begin{align*}
x_{(k)} ((\partial^{i_1}a_1)_{(n)}b) 
= (\partial^{i_1}a_1)_{(n)} (x_{(k)}b)+ 
\sum\limits_{\ell=0}^{k} 
\begin{pmatrix}
k \\
\ell
\end{pmatrix} (x_{(k-\ell)} \partial^{i_1}a_1)_{(n+\ell)}b,
\end{align*}
that is, 
$
\mathcal{P}_k^{n,a_1} = \mathcal{S}_k^{n,a_1} + 
\sum\limits_{\ell=0}^{k} 
\begin{pmatrix}
k \\
\ell
\end{pmatrix} \mathcal{R}_{k-\ell}^{n+\ell,a_1},
$
as expected.

\smallskip

\noindent
$\bullet$ $r \geq 2$. 
Set 
$\partial^{\underline{i}_2} \underline{a}_2 := 
 \partial^{i_2} a_2 \ldots \partial^{i_r} a_r.$ 
Using \eqref{Eq:R-b}, we get 
\begin{align*}
\mathcal{P}_k^{n,a} 
& = 
\sum\limits_{i \geq 0} \dfrac{1}{i!}  
\left(
(x_{(k)}( \partial^{i}(\partial^{i_1} a_1)) ( (\partial^{\underline{i}_2} \underline{a}_{2})_{(n+i)}b) 
+ (\partial^{i}(\partial^{i_1} a_1)) (x_{(k)}( (\partial^{\underline{i}_2} \underline{a}_{2})_{(n+i)}b )) \right. & \\
&\left. \qquad \quad + \;  (x_{(k)}( (\partial^{i}(\partial^{\underline{i}_2} \underline{a}_{2})) ( (\partial^{i_1} a_1)_{(n+i)}b) 
+ (\partial^{i}(\partial^{\underline{i}_2} \underline{a}_{2})) (x_{(k)}( (\partial^{i_1} a_1)_{(n+i)}b )) \right).&
\end{align*}
Thus,  
$\mathcal{P}_k^{n,a}   = \mathcal{P}_{1,k}^{n,a} + \mathcal{P}_{2,k}^{n,a},$ 
where 
\begin{align*} 
\mathcal{P}_{1,k}^{n,a}
& := \sum\limits_{i \geq 0} \dfrac{1}{i!}   
\left( 
(x_{(k)}( \partial^{i}(\partial^{i_1} a_1)) ( (\partial^{\underline{i}_2} \underline{a}_{2})_{(n+i)}b) + 
(\partial^{i}(\partial^{i_1} a_1)) \mathcal{P}_k^{n+i,\underline{a}_{2}} \right) , & \\
 \mathcal{P}_{2,k}^{n,a}
& := \sum\limits_{i \geq 0} \dfrac{1}{i!}  
\left( (x_{(k)}( (\partial^{i}(\partial^{\underline{i}_2} \underline{a}_{2})) ( (\partial^{i_1} a_1)_{(n+i)}b) + 
(\partial^{i}(\partial^{\underline{i}_2} \underline{a}_{2})) \mathcal{P}_k^{n+i,a_1} 
\right).&
\end{align*}
Similarly,  
$\mathcal{S}_k^{n,a}
 = \mathcal{S}_{1,k}^{n,a} + \mathcal{S}_{2,k}^{n,a}$, 
where 
\begin{align*}
\mathcal{S}_{1,k}^{n,a} & := \sum\limits_{i \geq 0} \dfrac{1}{i!}   
\left(  \partial^{i}(\partial^{i_1} a_1) \mathcal{S}_k^{n+i,\underline{a}_{2}} \right), & 
\quad \mathcal{S}_{2,k}^{n,a} & :=\sum\limits_{i \geq 0} \dfrac{1}{i!}   
\left(   \partial^{i}(\partial^{\underline{i}_2} \underline{a}_{2}) \mathcal{S}_k^{n+i,a_1} 
\right) ,&
\end{align*}
and  
$\mathcal{R}_k^{n,a} 
= \mathcal{R}_{1,k}^{n,a} + \mathcal{R}_{2,k}^{n,a}$, 
where 
\begin{align*}
\mathcal{R}_{1,k}^{n,a} 
& : = \sum\limits_{i \geq 0} \dfrac{1}{i!}  
\left( 
(\partial^{i} x_{(k)} (\partial^{i_1}a_1)) ((\partial^{\underline{i}_2}\underline{a}_{2})_{(n+i)}b) 
+ (\partial^{i}(\partial^{i_1}a_1)) \mathcal{R}_k^{n+i,\underline{a}_{2}} \right) ,  & \\
\mathcal{R}_{2,k}^{n,a}
& : = \sum\limits_{i \geq 0} \dfrac{1}{i!}  
\left( (\partial^{i} x_{(k)} (\partial^{\underline{i}_2}\underline{a}_{2})) ((\partial^{i_1}a_1)_{(n+i)}b) 
+ (\partial^{i}(\partial^{\underline{i}_2}\underline{a}_{2}))  \mathcal{R}_k^{n+i,a_1} \right) .&
\end{align*}
Clearly, 
it is enough to show:
\begin{align}
\label{eq:induction1}
&\mathcal{P}_{1,k}^{n,a} = \mathcal{S}_{1,k}^{n,a}+ \sum\limits_{\ell=0}^{k} 
\begin{pmatrix}
k \\
\ell
\end{pmatrix} \mathcal{R}_{1,k-\ell}^{n+\ell,a},&  
&\mathcal{P}_{2,k}^{n,a} = \mathcal{S}_{2,k}^{n,a} + \sum\limits_{\ell=0}^{k} 
\begin{pmatrix}
k \\
\ell
\end{pmatrix} \mathcal{R}_{2,k-\ell}^{n+\ell,a}.&
\end{align} 

\smallskip

\noindent 
$\bullet$ $k=0$.
Using Lemma \ref{Lem:identities}, (2), and the induction hypothesis, we get:
\begin{align*}
\mathcal{P}_{1,0}^{n,a} 
 & =\sum\limits_{i \geq 0} \dfrac{1}{i!} 
\Big( \partial^{i} (x_{(0)} (\partial^{i_1} a_1)) ( (\partial^{\underline{i}_2} \underline{a}_2)_{(n+i)} b )  &\\
& \qquad   
+ \; \partial^{i}(\partial^{i_1} a_1) \left( 
(\partial^{\underline{i}_2} \underline{a}_{2})_{(n+i)} (x_{(0)}b) 
+ \mathcal{R}_0^{n+i,\underline{a}_{2}} \right)
\Big)  = \mathcal{S}_{1,0}^{n,a}  + \mathcal{R}_{1,0}^{n,a}. &
\end{align*}
Similarly, $\mathcal{P}_{2,0}^{n,a} = \mathcal{S}_{2,0}^{n,a}  + \mathcal{R}_{2,0}^{n,a}$.

\smallskip

\noindent 
$\bullet$ $k \geq 1$. 
Using Lemma \ref{Lem:identities}, (3), (4), and the induction hypothesis, we get: 
\begin{align*} 
& \mathcal{P}_{1,k}^{n,a} 
 = \sum\limits_{i \geq 0} \dfrac{1}{i!}  
\left[ 
\Big( \partial^{i} x_{(k)}( \partial^{i_1} a_1) 
+  \sum\limits_{\ell=1}^k 
\begin{pmatrix} 
k\\
\ell
\end{pmatrix}
\dfrac{i!}{ (i-\ell)!}  
 \partial^{i-\ell} 
(x_{(k-\ell)}  (\partial^{i_1} a_1) ) 
\Big)
 (\partial^{\underline{i}_2} \underline{a}_{2})_{(n+i)}b  \right. &\\
& \qquad \qquad \qquad \qquad 
+ \left. \partial^{i}(\partial^{i_1} a_1) 
\left(\mathcal{S}_{k}^{n+i,\underline{a}_{2}} +
\sum\limits_{\ell=0}^k \begin{pmatrix} 
k\\
\ell
\end{pmatrix} \mathcal{R}_{k-\ell}^{n+i+\ell,\underline{a}_{2}} \right) \right]&\\
= & 
\sum\limits_{i \geq 0} \dfrac{1}{i!}  
\left[
\sum\limits_{\ell=0}^k 
\begin{pmatrix} 
k\\
\ell
\end{pmatrix} 
\left(
 \partial^{i} 
x_{(k-\ell)}  (\partial^{i_1} a_1) \, 
 (\partial^{\underline{i}_2} \underline{a}_{2})_{(n+i+\ell)}b   
 + \partial^i(\partial^{i_1} a_1) \mathcal{R}_{k-\ell}^{n+i+\ell,\underline{a}_{2}} 
\right) 
 + \partial^{i}(\partial^{i_1} a_1) 
\mathcal{S}_{k}^{n+i,\underline{a}_{2}} 
\right]  . &
 \end{align*}
Similarly, 
\begin{align*} 
\mathcal{P}_{2,k}^{n,a}  
 = &
\sum\limits_{i \geq 0} \dfrac{1}{i!}  
\left[
\sum\limits_{\ell=0}^k 
\begin{pmatrix} 
k\\
\ell
\end{pmatrix} 
\left(
 \partial^{i} 
x_{(k-\ell)}  (\partial^{\underline{i}_2} \underline{a}_2) \, 
 (\partial^{i_1} a_1)_{(n+i+\ell)}b   
 + \partial^i(\partial^{\underline{i}_2} \underline{a}_2) \mathcal{R}_{k-\ell}^{n+i+\ell,a_1} 
\right) \right]  
\\
&\qquad \qquad \qquad  + 
\sum\limits_{i \geq 0} \dfrac{1}{i!}   \partial^{i}(\partial^{\underline{i}_2} \underline{a}_2) 
\mathcal{S}_{k}^{n+i,a_1} 
 . &
 \end{align*}
Hence, we obtain:
\begin{align*}
\mathcal{P}_{k}^{n,a}
& =  \sum\limits_{\ell=0}^k 
\begin{pmatrix} 
k\\
\ell
\end{pmatrix}
\sum\limits_{i \geq 0} \dfrac{1}{i!} 
\left(
 \partial^{i} 
x_{(k-\ell)}  (\partial^{i_1} a_1)  \, 
 (\partial^{\underline{i}_2} \underline{a}_{2})_{(n+i+\ell)}b   
 +\partial^{i}(\partial^{i_1} a_1)  \mathcal{R}_{k-\ell}^{n+i+\ell,\underline{a}_{2}} \right. & \\
& \qquad \qquad\qquad \qquad+ 
\left. 
 \partial^{i} 
x_{(k-\ell)}  (\partial^{\underline{i}_2} \underline{a}_2)  \,
 (\partial^{i_1} a_1)_{(n+i+\ell)}b  +  \partial^{i}(\partial^{\underline{i}_2} \underline{a}_2) 
 \mathcal{R}_{k-\ell}^{n+i+\ell,a_1}\right)  & \\
& \quad 
+ \sum\limits_{i \geq 0} \dfrac{1}{i!} 
\left( \partial^{i}(\partial^{i_1} a_1) 
\mathcal{S}_{k}^{n+i,\underline{a}_{2}}
+ \partial^{i}(\partial^{\underline{i}_2} \underline{a}_2) 
\mathcal{S}_{k}^{n+i,a_1} \right)  = \sum\limits_{\ell=0}^k 
\begin{pmatrix} 
k\\
\ell
\end{pmatrix} \mathcal{R}_{k-\ell}^{n+\ell,a} + \mathcal{S}_{k}^{n,a}. &
\end{align*} 
This shows the formula \eqref{eq:main_ind} by induction. 
\end{proof}

\begin{remark}
Recall from Example \ref{Ex:inv_adjoint}  
that for the algebra $A=\kk[\g^*]$ equipped with the induced 
coadjoint action of $\GG$, 
the morphism \eqref{Eq:InvMorph} is an isomorphism. 
The Poisson vertex algebra structure on $\Jinf \kk[\g^*]$ can be understood in two ways: 
from the Kirillov-Kostant-Souriau Poisson structure on  
$A=\kk[\g^*]$, 
or from the $\Jinf \GG$-action on its Lie algebra $\Jinf \g$, see \cite{AM}.
\end{remark}

\subsection{
%We can now achieve the proof of Theorem \ref{Th:invariants_are_PVA}. 
%\begin{proof}[
Proof of Theorem \ref{Th:invariants_are_PVA}}
First of all, it suffices to show that $(\Jinf A)^{\Jinf \GG}$ is a Poisson vertex subalgebra 
of $\Jinf A$ to ensure that  
$j_A$ is a Poisson vertex algebra morphism.  
The case where $\GG$ is connected has been solved in Lemma \ref{Lem:connected}. 
If $\GG=\Gamma$ is finite, then $\Jinf \Gamma=\Gamma$ and, hence, 
$\Jinf \Gamma$ acts by Poisson vertex algebra automorphisms  
on $\Jinf A$. 

For the general case, since $\GG$ is an affine algebraic group, 
the component groups $\Gamma:=\GG/\GG^0$ is finite. 
Let $\{g_1,\ldots,g_\ell\}$ 
be a finite set of representatives of $\Gamma$ in $\GG$ 
so that any element $g$ of $\GG$ is uniquely written 
as $g=g^0 g_i $, with $g^0 \in \GG^0$ and $i\in\{1,\ldots,\ell\}$. 
Because $\Jinf \GG \cong \GG \times \g\[[t\]]$ as topological space, 
any element $g$ of $\Jinf \GG$ is also uniquely written 
as $g=g^0 g_i $, with $g^0 \in \Jinf \GG^0$ and $i\in\{1,\ldots,\ell\}$.  
Let now $a,b \in (\Jinf A)^{\Jinf \GG}$, $n \in \Z_{\geq 0}$.  
Write $g=g^0 g_i$ as above. 
Then 
$$g(a_{(n)}b) = g^0 (( g_i a)_{(n)} (g_i b)) 
= g^0 ( a_{(n)} b) = a_{(n)} b.$$
The first equality holds because 
$g_i$ acts as a Poisson vertex algebra automorphism; 
the last  
because $(\Jinf A)^{\Jinf \GG^0}$ is a Poisson vertex subalgebra by 
Lemma \ref{Lem:connected}. 
Similarly, we get 
$$g(\partial a) = g^0  g_i (\partial a)  
= g^0 (\partial a) = \partial a.$$ 
This concludes the proof of the theorem. 
%\end{proof}

%%%%%%%%%%%%%%%%%% NEW SECTION %%%%%%%%%%%%%%%%%%%%%
%%%%%%%%%%%%%%%%%% NEW SECTION %%%%%%%%%%%%%%%%%%%%%
%%%%%%%%%%%%%%%%%% NEW SECTION %%%%%%%%%%%%%%%%%%%%%
%%%%%%%%%%%%%%%%%% NEW SECTION %%%%%%%%%%%%%%%%%%%%%

\section{Poisson (vertex) reduction}   \label{S:PoiRed}

We continue to assume that $\kk$ is algebraically closed and of characteristic $0$.
Our goal is to show the commutativity of the diagram presented in Figure \ref{Fig:PoiRed}. 
The different categories appearing in Figure \ref{Fig:PoiRed} (from top to bottom and left to right) are defined as follows: 
 
\begin{itemize}
\item $\Jinf\DPA$ is the image of the the category of double Poisson algebras $\DPA$ inside the category of double Poisson vertex algebras $\DPVA$ under the jet functor of Proposition~\ref{Pr:DPAtoDPA}. 
\item $\HoP$ is the category of $H_0$-Poisson structures (see \S\ref{ss:PoiRedNC}), while  $\mathtt{V}_\infty\HoP$ is a suitable subcategory of that of relative $H_0$-Poisson vertex structures $\widehat{\HoPV}$   (see \S\ref{ss:HOPVrel} and \S\ref{ss:Interpret}). 
\item $\PA^\GG$ is the category of Poisson algebras equipped with a compatible action by a group $\GG$, which is a subcategory of $\PA$ (see \S\ref{ss:PoiRedNC}); $\Jinf\PA^\GG$ is the image of $\PA^{\GG}$ inside $\PVA$ under the jet functor of Lemma \ref{Lem:PAtoPVA}. 
\item $\PA_{\Gl_N;0}$ and $\Jinf \PA_{\Gl_N;0}$ are the analogues of $\PA^{\GG}$ and $\Jinf\PA^\GG$ after taking $\GG$-invariants (see \S\ref{ss:RedPVA} for a precise definition). 
\end{itemize} 
We have constructed the functors appearing at the top of the diagram in Section \ref{Sec:JQ}, and the other functors are introduced below.

\subsection{Poisson reduction and its noncommutative version} 
\label{ss:PoiRedNC}

Let $\GG$ be a group.    
We form the subcategory $\PA^\GG$ of $\PA$ as follows. 
Its objects are Poisson algebras equipped with a group action of $\GG$ by Poisson automorphisms.  
Given two objects $A_1,A_2$, a morphism $\phi\colon A_1\to A_2$ is a $\GG$-equivariant Poisson homomorphism.  

Fix a Poisson algebra $A$ in $\PA^\GG$. 
Write $A^\GG$ 
for the subalgebra of $\GG$-invariant elements in $A$. 
Since $\GG$ acts by Poisson automorphisms, then it is plain that $A^{\GG}$ is itself a Poisson algebra, 
the \emph{Poisson reduction} of $A$ by $\GG$. 
The choice of morphisms in $\PA^\GG$ guarantees that we have a functor $\mathtt{R}\colon \PA^\GG\to \PA$ when performing Poisson reduction. 
To fix ideas, if $(\AA,\dgal{-,-})$ is a double Poisson algebra and $N\geq 1$, we are in the above situation for $\GG=\Gl_N$ and $A=\AA_N$  equipped with the Poisson bracket induced by 
Theorem~\ref{Thm:RepdP}. 
In particular, the functor $(-)_N\colon \DPA\to \PA$ restricts to $(-)_N\colon \DPA\to \PA^{\Gl_N}$. 
We review the noncommutative analogue of the Poisson reduction $\AA_N^{\Gl_N}$ hereafter.

%\subsubsection{$H_0$-Poisson structures} \label{ss:HOP} 
Let $\AA$ be a (unital associative) algebra of finite type over $\kk$. We consider the vector space $H_0(\AA):=\AA/[\AA,\AA]$ for $[\AA,\AA]$ the $\kk$-linear subspace of $\AA$ generated by commutators 
$ab-ba$ ($a,b \in \AA$), 
which is the $0$-th Hochschild homology group of $\AA$. We denote by $a_\sharp\in H_0(\AA)$ the image of $a\in \AA$ under the natural projection $\AA\to H_0(\AA)$. 

\begin{definition}[\cite{CB}]  \label{Def:HOP}
An \emph{$H_0$-Poisson structure} on $\AA$ (or on $H_0(\AA)$) is a Lie bracket $$[-,-]_\sharp\colon H_0(\AA)\times H_0(\AA)\to H_0(\AA), \qquad (a_\sharp,b_\sharp) \mapsto [a_\sharp,b_\sharp]_\sharp\,,$$ 
such that for any $a\in \AA$, the $\kk$-linear map $[a_\sharp,-]_\sharp\colon  H_0(\AA)\to H_0(\AA)$ is induced by a map $\del_a\in \Der(\AA)$.     
\end{definition}
\begin{definition}[\cite{F22}]
For $i=1,2$, let $\AA_i$ be endowed with an $H_0$-Poisson structure $[-,-]_{\sharp,i}$. 
A morphism of algebras $\varphi\colon \AA_1\to \AA_2$ is a \emph{$H_0$-Poisson morphism} if the induced map $\bar{\varphi}\colon H_0(\AA_1)\to H_0(\AA_2)$ is a morphism of Lie algebras with respect to the Lie brackets $[-,-]_{\sharp,i}$.  

\noindent We denote by $\HoP$ the category of $H_0$-Poisson structures, whose objects are given by algebras equipped with an $H_0$-Poisson structure, and whose morphisms are $H_0$-Poisson morphisms. 
\end{definition} 

For any $N\geq 1$ and $a\in \AA$, define the element $\tr(a):=\sum_{j=1}^N a_{jj}\in \AA_N$ which is invariant under the action of $\Gl_N$ by \eqref{Eq:ActRep}; hence we get a morphism of algebras $\tr_N\colon \AA\mapsto \AA_N^{\Gl_N}$, the \emph{trace map}. We write $\tr:=\tr_N$ if it is clear from the context that we work with the $N$-th representation algebra. 
We see that $[\AA,\AA]\subset \ker(\tr)$, hence we get a linear map $H_0(\AA)\to \AA_N^{\Gl_N}$ which we also denote by $\tr$. The subalgebra of $\AA_N^{\Gl_N}$ generated by the image of the trace map coincides with $\AA_N^{\Gl_N}$ as $\AA$ is of finite type, see \cite[Remark 2.3]{CB}. We can then summarize the role of Crawley-Boevey's $H_0$-Poisson structures as follows. 

\begin{proposition}[\cite{CB}] \label{Pr:HOPtoPA}
Assume that $\AA$ is equipped with a $H_0$-Poisson structure $[-,-]_\sharp$. 
Then there exists a unique Poisson bracket on $\AA_N^{\Gl_N}$ which satisfies for any $a,b\in \AA$, 
\begin{equation} \label{Eq:relHoP}
   \br{\tr(a),\tr(b)} = \tr( [a_\sharp,b_\sharp]_\sharp)\,.
\end{equation}
Furthermore, this construction defines a functor $\mathtt{tr}_N\colon \HoP\to \PA$.  
\end{proposition}
\begin{proof}
The first part is \cite[Theorem 1.6]{CB}. 
The second part follows from \cite[Proposition 5.9]{F22}. Explicitly, a $H_0$-Poisson morphism $\varphi\colon \AA_1\to \AA_2$ induces a unique morphism of algebras 
\begin{equation} \label{Eq:HoP-PA}
  \tr_N(\varphi) \colon  (\AA_1)_N^{\Gl_N}\to  (\AA_2)_N^{\Gl_N}  
\end{equation}
defined on generators by $\tr_N(\varphi)(\tr_N(a)\!) := \tr_N\left( \varphi(a) \right)$, $a\in \AA_1$, and it respects the induced Poisson brackets. 
\end{proof}

Following Van den Bergh, let $(\AA,\dgal{-,-})$ be a double Poisson algebra. Composition with the multiplication on $\AA$ yields a linear map $\mathrm{m} \circ \dgal{-,-}\colon  \AA\otimes \AA\to \AA$. 
Furthermore, by \cite[\S2.4]{VdB} this map descends to $H_0(\AA)$ through 
\begin{equation} \label{Eq:HO-ind}
[-,-]_\sharp\colon H_0(\AA)\times H_0(\AA)\longrightarrow H_0(\AA), \qquad 
[a_\sharp,b_\sharp]_\sharp := (\mathrm{m} \circ \dgal{a,b} )_\sharp \,,   
\end{equation}
where on the right-hand side $a,b\in \AA$ are arbitrary lifts of $a_\sharp,b_\sharp\in H_0(\AA)$. 
The following result follows from \cite[Lemma 2.6.2]{VdB} and \cite[Proposition 5.3]{F22}.
\begin{proposition}
 The map $[-,-]_\sharp$ on $H_0(\AA)$ defined through \eqref{Eq:HO-ind} is a Lie bracket, hence it endows $\AA$ with an $H_0$-Poisson structure. 
 Furthermore, if $\AA_1,\AA_2$ are equipped with double Poisson brackets and $\theta\colon \AA_1\to \AA_2$ is a morphism of double Poisson algebras, then $\theta$ is a $H_0$-Poisson morphism for the corresponding $H_0$-Poisson structures obtained through \eqref{Eq:HO-ind}. 
\end{proposition}
As a consequence of this proposition, we get a functor $\sharp\colon \DPA\to \HoP$. 

%\subsubsection{General commutativity of the front face} 
%\label{ss:ComFront}

\begin{proposition} \label{Pr:ComFront} 
Fix $N\geq 1$. 
 The following diagram is commutative: 
 \begin{center}
    \begin{tikzpicture}
 \node  (TopLeft) at (-2,1) {$\DPA$};
\node  (TopRight) at (2,1) {$\PA^{\Gl_N}$};
\node  (BotLeft) at (-2,-1) {$\HoP$}; 
\node  (BotRight) at (2,-1) {$\PA$}; 
\path[->,>=angle 90,font=\small]  
   (TopLeft) edge node[above] {$(-)_N$} (TopRight) ;
\path[->,>=angle 90,font=\small]  
    (BotLeft) edge node[above] {$\tr_N$} (BotRight) ; 
\path[->,>=angle 90,font=\small]  
    (TopLeft) edge node[left] {$\sharp$}  (BotLeft) ;
\path[->,>=angle 90,font=\small]  
    (TopRight) edge node[right] {$\mathtt{R}$} (BotRight) ;  
    %%% right diagram  
 \node  (TopLeft2) at (4,1) {$\AA$};
\node  (TopRight2) at (7,1) {$\AA_N$};
\node  (BotLeft2) at (4,-1) {$H_0(\AA)$}; 
\node  (BotRight2) at (7,-1) {$\AA_N^{\Gl_N}$};  
\path[->,>=angle 90,font=\small]  
   (TopLeft2) edge (TopRight2) ;
\path[->,>=angle 90,font=\small]  
    (BotLeft2) edge (BotRight2) ; 
\path[->,>=angle 90,font=\small]  
    (TopLeft2) edge (BotLeft2) ;
\path[->,>=angle 90,font=\small]  
    (TopRight2) edge (BotRight2) ;     
   \end{tikzpicture}
\end{center}
\end{proposition}
\begin{proof} 
Given $(\AA,\dgal{-,-})$, we end up with the commutative algebra $\AA_N^{\Gl_N}$ if we use both sides of the square. In fact, the two Poisson brackets obtained on $\AA_N^{\Gl_N}$  coincide by \cite[\S7]{VdB}. Indeed, computing the Poisson bracket of elements $\tr(a),\tr(b)\in \AA_N^{\Gl_N}$, $a,b\in \AA$, with respect to these two approaches gives, respectively,  
\begin{equation}
    \begin{aligned}
\br{\tr(a),\tr(b)} &= \sum_{j,k=1}^N \br{ a_{jj}, b_{kk}} 
= \sum_{j,k=1}^N \dgal{a,b}'_{kj} \dgal{a,b}''_{jk} = \tr(\dgal{a,b}'\dgal{a,b}'')  \,, \\
\br{\tr(a),\tr(b)} &= \tr([a_\sharp,b_\sharp]_\sharp) = \tr(\mathrm{m}\circ \dgal{a,b})\,,
    \end{aligned}
\end{equation}
after using \eqref{Eq:relPA} in the first line, or \eqref{Eq:relHoP} and \eqref{Eq:HO-ind} in the second line. 

It remains to check that the two compositions of functors act in the same way on morphisms. Fix a morphism of double Poisson algebras $\theta\colon \AA_1\to \AA_2$. By spelling out the functors, we obtain the morphisms  $\mathtt{R}(\theta_N), \tr_N(\sharp(\theta)\!) \colon (\AA_1)_N^{\Gl_N}\to (\AA_2)_N^{\Gl_N}$ given on $\tr(a)$, $a\in \AA_1$,  by 
\begin{equation}
 \begin{aligned} \label{Eq:ComFrom}
     \mathtt{R}(\theta_N)(\tr(a)\!) &= \sum_{1\leq j\leq N} \theta_N(a_{jj})=\sum_{1\leq j\leq N}(\theta(a)\!)_{jj} = \tr_N(\theta(a)\!)\,, \\
\tr_N(\sharp(\theta)\!)(\tr(a)\!) &= \tr_N(\sharp(\theta)(a)\!) = \tr_N(\theta(a)\!)\,.
 \end{aligned}
\end{equation}
Since a morphism $(\AA_1)_N^{\Gl_N}\to (\AA_2)_N^{\Gl_N}$ is uniquely determined by its value on generators, the two images of $\theta$ coincide as expected.  
\end{proof}

\subsection{Reduction in the Poisson vertex setting}
\label{ss:RedPVA}
We prove the commutativity of the right face of Figure \ref{Fig:PoiRed}. 
In full generalities, we work with an affine algebraic group $\GG$. 

\medskip 

\subsubsection{} %{Invariants of jet algebra} \label{ss:NaiveInv} 
In analogy with the construction of $\PA^\GG$, we could be considering a Poisson vertex algebra $(V,\del,\br{-_\lambda-})$ endowed with a left action of the group $\GG$ by Poisson vertex automorphisms. This means that for any $g\in \GG$, the automorphism $g\cdot - \colon  V\to V$ satisfies 
\begin{equation} \label{Eq:PVAaut}
g\cdot \del (\tilde{F})=\del(g\cdot \tilde{F})\,, \quad 
g\cdot \br{\tilde{F}_\lambda \tilde{G}}=  \br{(g\cdot \tilde{F})_\lambda (g\cdot \tilde{G})}\,, \qquad \tilde{F},\tilde{G}\in V\,.
\end{equation}
By restriction, the subalgebra of $\GG$-invariant elements $V^\GG$  
is equipped with the differential and the $\lambda$-bracket from $V$, hence it is a Poisson vertex algebra. 

Assume that $A$ is an object in $\PA^{\GG}$. 
Consider $V:=\Jinf(A)$ with the Poisson vertex algebra structure described in Lemma \ref{Lem:PAtoPVA}. 
We can extend the $\GG$-action of $A$ to $V$ by requiring that the first equality in \eqref{Eq:PVAaut} is satisfied; the second equality in \eqref{Eq:PVAaut} will then also hold as it is satisfied for $\tilde{F},\tilde{G}\in A\subset V$. 
As in the previous paragraph, we can consider the Poisson vertex subalgebra $V^\GG\subset V$. 
Nevertheless,  $V^\GG=\Jinf(A)^\GG$ will be far too big compared to $\Jinf(A^\GG)$ in most situations, as can be seen e.g. in Example  \ref{Ex:invMat} below.

\begin{example} \label{Ex:invMat}
Let $A=\kk[u_{ij} \mid 1\leq i,j\leq N]$ be the algebra of functions on $\gl_N$ endowed with the adjoint action of $\Gl_N$ by conjugation. (The choice of Poisson bracket on $A$ is irrelevant.)
It is well-known that $A^{\Gl_N}$ is generated by $\tr(u),\ldots,\tr(u^N)$, 
where $\tr(u^\ell):=\sum_{i_1,\ldots,i_\ell} u_{i_1i_2} \cdots u_{i_\ell i_1}$. 
Meanwhile, we can realise $V=\Jinf(A)$ as $\kk[u_{ij}^{(r)} \mid 1\leq i,j\leq N,\,r\geq 0]$ endowed with the differential $\del\colon u_{ij}^{(r)}\mapsto u_{ij}^{(r+1)}$. Then $V$ inherits a $\Gl_N$-action compatible with $\del$ by conjugation of each `matrix' $u^{(r)}:=(u_{ij}^{(r)})_{1\leq i,j\leq N}$. 
Any function $\tr(\!(u^{(r)})^2):=\sum_{i,j} u^{(r)}_{ij} u^{(r)}_{ji}$, $r\geq1$, belongs to $V^{\Gl_N}$ but not to $\Jinf(A^{\Gl_N})$. 
\end{example}

For this reason, we consider a different construction of the algebra of invariants, 
which is more natural from the point of view of jet algebras and arc spaces. 
Assume from now on that $A$ is equipped with a Poisson bracket and $\GG$ acts by Poisson automorphisms as in \S\ref{ss:PoiRedNC} (i.e.~$A$ is an element of $\PA^\GG$). Equip $\Jinf(A)$ with the Poisson vertex algebra structure induced from $A$ under the functor $\mathtt{J}\colon \PA\to \PVA$ of Lemma \ref{Lem:PAtoPVA}.  
We remark that $\Jinf\PA^\GG:=\mathtt{J}(\PA^\GG)$ is a subcategory of $\PVA$. Its objects inherit an action of $\Jinf(\GG)$ such that its morphisms are  $\Jinf(\GG)$-equivariant. Indeed, if $\phi\colon A_1\to A_2$ in $\PA^\GG$, then its $\GG$-equivariance amounts to $\rho_2\circ \phi=(\phi\otimes \Id_{\kk[\GG]})\circ \rho_1$ where $\rho_i$ is the coaction of $A_i$ from \eqref{Eq:rhoV}, hence we have an analogous equality for $\phi_\infty\colon  \Jinf(A_1)\to \Jinf(A_2)$.  

As a consequence of Theorem \ref{Th:invariants_are_PVA}, 
we get a functor $\mathtt{R}_\infty\colon \Jinf \PA^\GG\to \PVA$. 

%\subsubsection{Commutativity of the right face} \label{ss:RightF}

\subsubsection{} 
We form the category $\PA_{\GG;0}$ as follows. 
An object in $\PA_{\GG;0}$ is a pair $(A,A^\GG)$ where
$A$ is a commutative algebra equipped with a left action of $\GG$ and  
the subalgebra $A^\GG$ of $\GG$-invariant elements is endowed with a Poisson bracket.  
A morphism $\phi:(A_1,A_1^\GG)\to (A_2,A_2^\GG)$ is given by a $\GG$-equivariant morphism $\phi:A_1\to A_2$ of commutative algebras that restricts to a morphism $\phi:A_1^\GG\to A_2^\GG$ of Poisson algebras. We can obviously rewrite the Poisson reduction functor $\mathtt{R}:\PA^\GG\to \PA$ from \S\ref{ss:PoiRedNC} as a functor $\mathtt{R}:\PA^\GG\to \PA_{\GG;0}$. 
We can also rewrite the functor $\tr_N:\HoP\to \PA$ from Proposition \ref{Pr:HOPtoPA} as a functor $\tr_N:\HoP\to \PA_{\Gl_N;0}$ defined on objects by $\tr_N: \AA\mapsto (\AA_N,\AA_N^{\Gl_N})$ because we only need to see $\AA_N$ as a commutative algebra. 
In particular, we can rephrase the commutativity of the front face from 
Proposition \ref{Pr:ComFront}  as the equality of the  functors 
\begin{equation*}
\DPA \stackrel{(-)_N}{\longrightarrow} \PA^{\Gl_N} \stackrel{\mathtt{R}}{\longrightarrow} \PA_{\Gl_N;0}\,, \qquad 
\DPA \stackrel{\sharp}{\longrightarrow} \HoP \stackrel{\tr_N}{\longrightarrow} \PA_{\Gl_N;0}\,.
\end{equation*}

Next, we define the category $\Jinf \PA_{\GG;0}$ having objects $(\Jinf(A),\Jinf(A^\GG))$ for any $(A,A^\GG)\in \PA_{\GG;0}$ and whose morphisms are jets of morphisms in $\PA_{\GG;0}$. Note that we see $\Jinf(A)$ as a differential algebra while $\Jinf(A^\GG)$ is seen as a Poisson vertex algebra. 
We can then rewrite the jet functor $\mathtt{J}:\PA\to \PVA$ as a functor $\mathtt{J}:\PA_{\GG;0}\to \Jinf \PA_{\GG;0}$.

We shall abuse notation from now on and we simply write an object $(A,A^\GG)$ in $\PA_{\GG;0}$ as $A^\GG$, and do the same in $\Jinf \PA_{\GG;0}$. 

The introduction of the categories $\PA_{\GG;0}$ and $\Jinf \PA_{\GG;0}$ has the following advantage : the morphism $j_A$ \eqref{Eq:InvMorph} can be lifted to 
a functor $j_{(-)}:\Jinf \PA_{\GG;0}\to \PVA$. 
Indeed, if $\phi_\infty:\Jinf(A_1^\GG)\to \Jinf(A_2^\GG)$ is a morphism in $\Jinf \PA_{\GG;0}$, it must be the restriction of the jet $\phi_\infty:\Jinf(A_1)\to \Jinf(A_2)$ of a morphism $\phi:A_1\to A_2$ in $\PA^\GG$. 
Since the $\GG$-invariance of $\phi$ implies that $\phi_\infty$ is $\Jinf(\GG)$-equivariant, we obtain trivially a commutative diagram 
\begin{center}
    \begin{tikzpicture}
 \node  (TopLeft) at (-2,1) {$\Jinf(A_1^\GG)$};
 \node  (TopRight) at (2,1) {$\Jinf(A_1)^{\Jinf(\GG)}$};
 \node  (BotLeft) at (-2,-1) {$\Jinf(A_2^\GG)$};
 \node  (BotRight) at (2,-1) {$\Jinf(A_2)^{\Jinf(\GG)}$}; 
\path[->,>=angle 90,font=\small]  
   (TopLeft) edge node[above] {$j_{A_1}$} (TopRight) ;
   \path[->,>=angle 90,font=\small]  
   (BotLeft) edge node[above] {$j_{A_2}$} (BotRight) ;
\path[->,>=angle 90,font=\small]  
   (TopLeft) edge node[left] {$\phi_\infty$}  (BotLeft) ;
\path[->,>=angle 90,font=\small]  
   (TopRight) edge node[right] {$\phi_\infty$} (BotRight) ;  
   \end{tikzpicture}
\end{center}
from which we see that $j_{(-)}$ sends the restriction $\Jinf(A_1^\GG)\to \Jinf(A_2^\GG)$ of $\phi_\infty:\Jinf(A_1)\to \Jinf(A_2)$ onto the restriction $\Jinf(A_1)^{\Jinf(\GG)}\to \Jinf(A_2)^{\Jinf(\GG)}$.  

\subsubsection{} 
We are now in position to prove the commutativity of the right face of Figure \ref{Fig:PoiRed}. 
It follows by specializing the next lemma to the case $\GG=\Gl_N$. 
%In full generalities, the next result holds. 

\begin{lemma}[Commutativity of the right face] \label{Lem:RightF}
 The following diagram is commutative 
\begin{center}
    \begin{tikzpicture}
 \node  (TopLeft) at (-2,1) {$\PA^\GG$};
\node  (TopRight) at (2,1) {$ \Jinf\PA^\GG$};
\node  (BotLeft) at (-2,-1) {$\PA_{\GG;0}$};
\node (BotMid) at (0,-1) {$\Jinf \PA_{\GG;0}$};
\node  (BotRight) at (2,-1) {$\PVA$}; 
\path[->,>=angle 90,font=\small]  
   (TopLeft) edge node[above] {$\mathtt{J}$} (TopRight) ;
\path[->,>=angle 90,font=\small]  
    (BotLeft) edge node[above] {$\mathtt{J}$} (BotMid) ;
\path[->,>=angle 90,font=\small]  
    (BotMid) edge node[above] {$j_{(-)}$} (BotRight) ;
\path[->,>=angle 90,font=\small]  
    (TopLeft) edge node[left] {$\mathtt{R}$}  (BotLeft) ;
\path[->,>=angle 90,font=\small]  
    (TopRight) edge node[right] {$\mathtt{R}_\infty$} (BotRight) ;  
    %%% right diagram  
 \node  (TopLeft2) at (4,1) {$A$};
\node  (TopRight2) at (9,1) {$ \Jinf(A)$};
\node  (BotLeft2) at (4,-1) {$A^\GG$};
\node (BotMid2) at (6,-1) {$\Jinf(A^\GG)$};
\node  (BotRight2) at (9,-1) {$\Jinf(A)^{\Jinf(\GG)}$};  
\path[->,>=angle 90,font=\small]  
   (TopLeft2) edge (TopRight2) ;
\path[->,>=angle 90,font=\small]  
    (BotLeft2) edge (BotMid2) ;
\path[->,>=angle 90,font=\small]  
    (BotMid2) edge (BotRight2) ;
\path[->,>=angle 90,font=\small]  
    (TopLeft2) edge (BotLeft2) ;
\path[->,>=angle 90,font=\small]  
    (TopRight2) edge (BotRight2) ;     
   \end{tikzpicture}
\end{center}
\end{lemma}
\begin{proof}
Starting with $A\in \PA^{\GG}$, we end up with the differential algebra $\Jinf(A)^{\Jinf(\GG)}$ under both functors. 
At the level of the Poisson (vertex) bracket, note that under $\mathtt{R}$ (or $\mathtt{R}_\infty$ and $j_{(-)}$) we restrict the bracket to an invariant subalgebra of $A$ (or $\Jinf(A)$). 
Thus comparing the $\lambda$-brackets obtained under the composite of functors given above is direct since they are obtained from $A$ or the subalgebra $A^\GG$ by the jet construction of Lemma \ref{Lem:PAtoPVA}. 

Let us now check that both sides of the diagram send a morphism to the same morphism in $\PVA$. 
The functor  $\mathtt{R}$ (or $\mathtt{R}_\infty$) sends a morphism to its restriction to invariant elements. 
Thus, given a morphism $\phi\colon A_1\to A_2$ in $\PA^{\GG}$, we end up with the map 
$\phi_\infty \colon \Jinf(A_1)^{\Jinf(\GG)}\to \Jinf(A_2)^{\Jinf(\GG)}$ in both cases. 
In particular the two composites of functors $\mathtt{R}_\infty \circ \mathtt{J}$ and $j_{(-)}\circ \mathtt{J}\circ \mathtt{R}$ commute.  
\end{proof}
% Specializing Lemma \ref{Lem:RightF} to the case $\GG=\Gl_N$ provides the commutativity of the right face in Figure \ref{Fig:PoiRed}. 

\subsection{\texorpdfstring{$H_0$}{H0}-Poisson vertex structures} \label{ss:HOPconf}

We introduce a suitable `vertex' analogue of Crawley-Boevey's $H_0$-Poisson structures \cite{CB} that were considered in \S\ref{ss:PoiRedNC}. The rough idea consists in replacing algebras with differential algebras, and Lie brackets with Lie vertex brackets. 

%\subsubsection{First definition and examples} 
We denote by $\VV$ a (unital associative) differential algebra over $\kk$. Its differential  $\del\in \Der(\VV)$ is a derivation, hence it descends to a linear map $H_0(\VV)\to H_0(\VV)$ also denoted by $\del$.   

\begin{definition}  \label{Def:HOPV}
An \emph{$H_0$-Poisson vertex structure} on $\VV$ (or on $H_0(\VV)$) is a Lie vertex bracket $$[-_\lambda-]_\sharp\colon H_0(\VV)\times H_0(\VV)\longrightarrow H_0(\VV)[\lambda], \qquad (a_\sharp,b_\sharp) \longmapsto [a_\sharp{}_\lambda b_\sharp]_\sharp\,,$$ 
(in the sense of Remark \ref{Rem:LVA}) 
such that the $\kk$-linear map 
$[a_\sharp{}_\lambda -]_\sharp\colon  H_0(\VV)\to H_0(\VV)[\lambda]$ is induced for any $a\in \VV$ by a map $\del_a\in \Der(\VV,\VV[\lambda])$ satisfying 
\begin{equation} \label{Eq:HoP-cond}
   \del_a(\del(b)\!) = (\del+\lambda)\del_a(b)\,, \qquad  \forall\, b\in \VV\,.
\end{equation}   
\end{definition}
\begin{definition} 
For $i=1,2$, let $\VV_i$ be endowed with an $H_0$-Poisson vertex structure $[-{}_\lambda-]_{\sharp,i}$. 
A morphism of differential algebras $\varphi\colon \VV_1\to \VV_2$ is a \emph{$H_0$-Poisson vertex morphism} if the induced map $\bar{\varphi}\colon H_0(\VV_1)\to H_0(\VV_2)$ is a morphism of Lie vertex algebras with respect to the Lie vertex brackets $[-{}_\lambda-]_{\sharp,i}$.  

\noindent We denote by $\HoPV$ the category of $H_0$-Poisson vertex structures, whose  objects are given by differential algebras equipped with an $H_0$-Poisson vertex structure, and whose morphisms are $H_0$-Poisson vertex morphisms. 
\end{definition} 

Following \cite[\S3.2]{DSKV}, let $(\VV,\del,\dgal{-_\lambda-})$ be a double Poisson vertex algebra. Using the multiplication on $\VV$, we get a linear map $\mathrm{m}\circ \dgal{-_\lambda-} \colon  \VV\otimes \VV\to \VV[\lambda]$. 
It descends to $H_0(\VV)$ through 
\begin{equation} \label{Eq:HOvert-ind}
[-_\lambda-]_\sharp\colon H_0(\VV)\times H_0(\VV)\longrightarrow H_0(\VV)[\lambda], \qquad 
[a_\sharp{}_\lambda b_\sharp]_\sharp := (\mathrm{m} \circ \dgal{a_\lambda b} )_\sharp \,,   
\end{equation}
where on the right-hand side $a,b\in \VV$ are arbitrary lifts of $a_\sharp,b_\sharp\in H_0(\VV)$.  
\begin{proposition}  \label{Pr:DPVA-HOPV}
 The map $[-{}_\lambda-]_\sharp$ on $H_0(\VV)$ defined through \eqref{Eq:HOvert-ind} is a Lie vertex bracket, hence it endows $\VV$ with an $H_0$-Poisson vertex structure. 
 Furthermore, if $\VV_1,\VV_2$ are double Poisson vertex algebras and $\theta\colon \VV_1\to \VV_2$ is a morphism of double Poisson vertex algebras, then $\theta$ is a $H_0$-Poisson vertex morphism for the corresponding $H_0$-Poisson vertex structures obtained through \eqref{Eq:HOvert-ind}. 
\end{proposition}
\begin{proof}
The map \eqref{Eq:HO-ind} is a Lie vertex bracket by Theorem 3.6 (a)-(d) in  \cite{DSKV}.  Therefore it is a $H_0$-Poisson vertex structure since for any $a\in \VV$,  $[a_\sharp{}_\lambda -]_\sharp$  is induced by $\del_a:=\mathrm{m} \circ \dgal{a_\lambda -}$ which is a derivation by \eqref{Eq:DL} compatible with $\del$ by  \eqref{Eq:DA1}.  

For the second part, we only need to check that $\bar{\theta}\colon H_0(\VV_1)\to H_0(\VV_2)$ is a morphism of Lie vertex algebras. If $a_\sharp,b_\sharp\in H_0(\VV_1)$ admit lifts $a,b\in\VV_1$, then $\mathrm{m} \circ \dgal{a_\lambda b}_1$ is a lift of $[a_\sharp{}_\lambda b_\sharp]_{\sharp,1}$ while $\mathrm{m} \circ \dgal{\theta(a)_\lambda \theta(b)}_2$ is a lift of $[(\theta(a)\!)_\sharp{}_\lambda (\theta(b)\!)_\sharp]_{\sharp,2}$. This yields
\begin{equation}
\begin{aligned}
      \bar{\theta} ([a_\sharp{}_\lambda b_\sharp]_{\sharp,1})
  &= (\theta(\mathrm{m} \circ \dgal{a_\lambda b}_1)\!)_\sharp
 = (\mathrm{m} \circ \theta^{\otimes 2}\dgal{a_\lambda b}_1)_\sharp
 = (\mathrm{m} \circ \dgal{\theta(a)_\lambda \theta(b)}_2)_\sharp \,,\\
[\bar{\theta}(a_\sharp){}_\lambda \bar{\theta}(b_\sharp)]_{\sharp,2} 
&= [(\theta(a)\!)_\sharp{}_\lambda (\theta(b)\!)_\sharp]_{\sharp,2}
= (\mathrm{m} \circ \dgal{\theta(a)_\lambda \theta(b)}_2)_\sharp\,,
\end{aligned}
\end{equation}
where we used that $\theta$ is a morphism of double Poisson vertex algebras in the first line. Hence  $\bar{\theta}$ intertwines the two Lie vertex brackets, as desired. 
\end{proof}
As a consequence of this proposition, we get a functor $\sharp_\mathtt{V}\colon \DPVA\to \HoPV$.

\begin{example} \label{Ex:TrivHoPV}
Any (commutative) Poisson vertex algebra $(V,\del,\br{-_\lambda-})$ has a $H_0$-Poisson vertex structure defined by $[a {}_\lambda b]_{\sharp}:=\br{a_\lambda b}_\sharp$ for any $a,b\in H_0(V)=V$. 
\end{example}
\begin{example}
    Consider the commutative algebra $V=\kk[x,y]$ with differential $\del \equiv 0_V$. 
By reproducing the end of Example 2.4 in \cite{CB}, 
we can show that if $(V,\del,\dgal{-_\lambda-})$ is a double Poisson vertex algebra, the $H_0$-Poisson vertex structure induced by \eqref{Eq:HOvert-ind} is identically zero. 
In particular, applying Example \ref{Ex:TrivHoPV} to $V$ equipped with the Poisson vertex bracket satisfying 
$$\br{x_\lambda x}=0, \quad \br{y_\lambda y}=0, \quad  \br{x_\lambda y}=1\,,$$
yields an object in $\HoPV$ which is not in $\sharp_{\mathtt{V}}(\DPVA)$.  
\end{example}
\begin{remark}
Fix $\ell \geq 2$. 
It is proved by Powell \cite{Po} that if the polynomial algebra $\kk[x_1,\ldots,x_\ell]$ has a double Poisson bracket, then it is identically zero. 
It is an open problem to know whether a similar vanishing holds for a double $\lambda$-bracket on $\kk[x_1^{(r)} ,\ldots,x_\ell^{(r)} \mid r\geq 0]$ with differential  $\del(x_j^{(r)})=x_j^{(r+1)}$, $1\leq j \leq \ell$. 
\end{remark}

\subsection{Relative \texorpdfstring{$H_0$}{H0}-Poisson vertex structures and jets}
\label{ss:HOPVrel}

We let $\VV$ be a differential algebra. 
Consider a vector space $\hat{\VV}\subset \VV$ stable under the differential, i.e., $\del(\hat{\VV})\subset \hat{\VV}$. 
We denote by $H_0(\hat{\VV})$ the image of the composite map $\hat{\VV}\hookrightarrow \VV \twoheadrightarrow H_0(\VV)$ which is the subspace spanned by $\hat{\VV}$ in $H_0(\VV)$.
By assumption, $\del$ descends to a linear map on  $H_0(\hat{\VV})$. 

\begin{definition}  \label{Def:HOPVsub}
An \emph{$H_0$-Poisson vertex structure} on the pair $\hat{\VV}\subset \VV$ (or on $H_0(\hat{\VV})$) is a Lie vertex bracket $$[-_\lambda-]_\sharp\colon H_0(\hat{\VV})\times H_0(\hat{\VV})\longrightarrow H_0(\hat{\VV})[\lambda], \qquad (a_\sharp,b_\sharp) \longmapsto [a_\sharp{}_\lambda b_\sharp]_\sharp\,,$$ 
(in the sense of Remark \ref{Rem:LVA}) 
such that the $\kk$-linear map 
$[a_\sharp{}_\lambda -]_\sharp\colon  H_0(\hat{\VV})\to H_0(\hat{\VV})[\lambda]$ is induced for any $a\in \hat{\VV}$ by a map $\del_a\in \Der(\VV,\VV[\lambda])$ satisfying \eqref{Eq:HoP-cond}. 
\end{definition}
Let us emphasize that $\hat{\VV}$ may not inherit the multiplication from $\VV$, hence each linear map $\del_a$ is a derivation on $\VV$ that sends $\hat{\VV}$ into $\hat{\VV}[\lambda]$. We obviously recover Definition \ref{Def:HOPV}  when $\hat{\VV}= \VV$.

\begin{definition} 
For $i=1,2$, let $\hat{\VV}_i \subset \VV_i$ be endowed with an $H_0$-Poisson vertex structure $[-{}_\lambda -]_{\sharp,i}$. 
A morphism of differential algebras $\varphi\colon \VV_1\to \VV_2$ is a \emph{$H_0$-Poisson vertex morphism (relative to $\hat{\VV}_1$ and $\hat{\VV}_2$)} if $\varphi(\hat{\VV}_1)\subset \hat{\VV}_2$ and the induced map $\bar{\varphi}\colon H_0(\hat{\VV}_1)\to H_0(\hat{\VV}_2)$ is a morphism of Lie vertex algebras with respect to the Lie vertex brackets $[-{}_\lambda-]_{\sharp,i}$.  

\noindent We denote by $\widehat{\HoPV}$ the category of $H_0$-Poisson vertex structures, whose  objects are given by $H_0$-Poisson vertex structures on pairs $\hat{\VV}\subset \VV$ (as in Definition \ref{Def:HOPVsub}), and whose morphisms are (relative) $H_0$-Poisson vertex morphisms. 
\end{definition} 
We see $\HoPV$ as a subcategory of $\widehat{\HoPV}$ consisting of the pairs $\hat{\VV}\subset \VV$ where $\hat{\VV}= \VV$. We now explain the importance of the category $\widehat{\HoPV}$ when considering jets, i.e. when $\VV=\Jinf(\AA)$.

\begin{proposition} \label{Pr:HoP-HoPV}
 Let $\AA$ be endowed with a $H_0$-Poisson structure denoted $[-,-]_\sharp$. Consider the following vector space inside $\Jinf(\AA)$: 
 \begin{equation}  \label{Eq:VectInf}
     \Vect_\infty(\AA):=\operatorname{span}_\kk \{ \del^r(a) \mid a\in \AA,\,\, r\geq 0 \} \,.
 \end{equation}
 There exists a unique $H_0$-Poisson vertex structure on $\Vect_\infty(\AA)\subset \Jinf(\AA)$, denoted $[-_\lambda -]_\sharp$, that satisfies 
 \begin{equation}  \label{Eq:VectId}
[a_\sharp {}_\lambda b_\sharp]_\sharp :=  [a_\sharp , b_\sharp]_\sharp \, \lambda^0\,,   
 \end{equation}
 with $a:=\del^0(a)$ and $b:=\del^0(b)$ taken in $\Vect_\infty(\AA)$ 
 for any $a,b\in \AA$. 
 Furthermore, this construction extends to a functor $\Vect_\infty\colon \HoP \to \widehat{\HoPV}$. 
\end{proposition}
\begin{proof}
   By Definition \ref{Def:HOP}, for any $a\in \AA$  there exists $\del_a\in \Der(\AA)$ that induces $[a_\sharp , -]_\sharp \in \End(H_0(\AA)\!)$. 
We can extend $\del_a$ uniquely as a derivation $\Jinf(\AA)\to  \Jinf(\AA)[\lambda]$ by requiring that 
\begin{equation} \label{Eq:VectPf1}
    \del_a(\del^r(c)\!):=(\lambda+\del)^r\, \del_a(c)\,, \qquad \forall c\in \AA,\,\, r\geq 0\,.
\end{equation}
The element $\del_a\in \Der(\Jinf(\AA), \Jinf(\AA)[\lambda])$ hence obtained is well-defined, and it satisfies \eqref{Eq:HoP-cond} for any $b\in \Jinf(\AA)$. Indeed, \eqref{Eq:HoP-cond} is checked on $b=\del^{r_1}(c_1)\cdots \del^{r_\ell}(c_\ell)$, where $r_j\geq 0$ and $c_j\in \AA$ for $1\leq j\leq \ell$, by induction on $\ell\geq1$, noting that the case $\ell=1$ holds by \eqref{Eq:VectPf1}.  

For any $a_\sharp \in H_0(\AA)\subset  H_0(\Vect_\infty(\AA)\!)$,  
we set 
 \begin{equation}  \label{Eq:VectPf2}
[a_\sharp {}_\lambda b_\sharp]_\sharp :=  (\del_a(b)\!)_\sharp\,,   \quad \text{for any }b \in \Vect_\infty(\AA)\,.
 \end{equation}
Clearly,  \eqref{Eq:VectId} is satisfied. We also note that \eqref{Eq:VectPf2} is independent of the chosen $\del_a\in \Der(\AA)$: if $\tilde{\del}_a\in \Der(\AA)$ is another lift of  $[a_\sharp , -]_\sharp \in \End(H_0(\AA)\!)$, 
then $\del_a - \tilde{\del}_a\in \Der(\AA)$ takes value in $[\AA,\AA]$ and therefore its extension $\del_a - \tilde{\del}_a\in \Der(\Jinf(\AA), \Jinf(\AA)[\lambda])$ constructed as above takes value in  $[\Jinf(\AA),\Jinf(\AA)]\otimes \kk[\lambda]$, so it restricts to the zero map $H_0(\Vect_\infty(\AA)\!)\to H_0(\Vect_\infty(\AA)\!)[\lambda]$. 

Next, we fix $\del^r(a)\in \Vect_\infty(\AA)$ for $a\in \AA$ and $r\geq 0$. 
We let 
 \begin{equation}  \label{Eq:VectPf3}
[(\del^r(a)\!)_\sharp {}_\lambda b_\sharp]_\sharp := (-\lambda)^k (\del_a(b)\!)_\sharp\,,   \quad \text{for any }b \in \Vect_\infty(\AA)\,,
 \end{equation}
which is induced by $(-\lambda^k)\,\del_a\in \Der(\Jinf(\AA), \Jinf(\AA)[\lambda])$ that satisfies \eqref{Eq:HoP-cond}. We recover \eqref{Eq:VectPf2} for $r=0$. By extending \eqref{Eq:VectPf3} linearly in the first argument, we have defined a bilinear map 
\begin{equation*}
[-_\lambda-]_\sharp\colon H_0(\Vect_\infty(\AA)\!)\times H_0(\Vect_\infty(\AA)\!) \longrightarrow H_0(\Vect_\infty(\AA)\!)[\lambda]\,.
\end{equation*}
The sesquilinearity rules \eqref{Eq:A1} hold by construction. 
For skewsymmetry \eqref{Eq:A2}, we compute for any $r,s\geq 0$ and $a,c\in \Vect_\infty(\AA)$ 
\begin{equation*}
[(\del^r(a)\!)_\sharp {}_\lambda (\del^s c)_\sharp]_\sharp 
= (-\lambda)^r (\lambda+\del)^s [a_\sharp {}_\lambda c_\sharp]_\sharp    
= (-\lambda)^r (\lambda+\del)^s [a_\sharp , c_\sharp]_\sharp  \,. 
\end{equation*}
Using that we have a $H_0$-Poisson structure, $[a_\sharp , c_\sharp]_\sharp=-[c_\sharp , a_\sharp]_\sharp$. in terms of the operator $\mu:=-\lambda-\del$, we can write 
\begin{equation*}
[(\del^r(a)\!)_\sharp {}_\lambda (\del^s c)_\sharp]_\sharp 
=-\Big|_{\mu=-\lambda-\del} (-\mu)^s (\mu+\del)^r  [c_\sharp , a_\sharp]_\sharp 
=-\Big|_{\mu=-\lambda-\del} [(\del^s c)_\sharp {}_{\mu} (\del^r(a)\!)_\sharp ]_\sharp 
\,,
\end{equation*}
as desired. 

Finally, we have to check Jacobi identity \eqref{Eq:A3}. For any $r,s,t\geq 0$ and $a,b,c\in \AA$, a standard computation using sesquilinearity \eqref{Eq:A1} yields  
\begin{equation}
\begin{aligned}
  &[(\del^r(a)\!)_\sharp {}_\lambda [(\del^s(b)\!)_\sharp {}_\mu (\del^t c)_\sharp]_\sharp ]_\sharp  
  - [(\del^s(b)\!)_\sharp {}_\mu [(\del^r(a)\!)_\sharp {}_\lambda (\del^t c)_\sharp]_\sharp ]_\sharp   
  - [[ (\del^r(a)\!)_\sharp {}_\lambda (\del^s(b)\!)_\sharp]_{\lambda + \mu} (\del^t c)_\sharp]_\sharp ]_\sharp   \\
&=(-\lambda)^r (-\mu)^s (\lambda+\mu+\del)^t 
\Big([a_\sharp {}_\lambda [b_\sharp {}_\mu c_\sharp]_\sharp ]_\sharp  
- [b_\sharp {}_\mu [a_\sharp {}_\lambda c_\sharp]_\sharp ]_\sharp   
- [[ a_\sharp {}_\lambda b_\sharp]_{\lambda + \mu} c_\sharp]_\sharp ]_\sharp\Big)\,.   
\end{aligned}    
\end{equation}
Hence we are left to check Jacobi identity in $H_0(\AA)$ where, by \eqref{Eq:VectId}, it becomes the Jacobi identity of the Lie bracket $[-,-]_\sharp$. 

For the functoriality, consider a $H_0$-Poisson morphism $\varphi\colon \AA_1\to\AA_2$. Passing to the jets, we get a morphism of differential algebras $\varphi_\infty \colon \Jinf(\AA_1)\to\Jinf(\AA_2)$, so that $\varphi_\infty(\Vect_\infty(\AA_1)\!)\subset \Vect_\infty(\AA_2)$. This last morphism induces the map 
$\overline{\varphi}_\infty\colon H_0(\Vect_\infty(\AA_1)\!)\to H_0(\Vect_\infty(\AA_2)\!)$, whose restriction to $H_0(\AA_1)$ equals $\overline{\varphi}$ because the restriction of $\varphi_\infty$ to $\AA_1$ is $\varphi$. We can then compute  for any $r,s\geq 0$ and $a,b\in \AA_1$
\begin{equation}
\begin{aligned}
   \overline{\varphi}_\infty( [(\del^r(a)\!)_\sharp {}_\lambda (\del^s b)_\sharp]_{\sharp,1}  ) 
&=(-\lambda)^r (\lambda+\del)^s \,\overline{\varphi}( [a_\sharp , b_\sharp]_{\sharp,1}  ) \\
&=(-\lambda)^r (\lambda+\del)^s \,[\overline{\varphi}(a_\sharp) , \overline{\varphi}(b_\sharp)]_{\sharp,2} \\
&= \,[\overline{\varphi}_\infty(\!(\del^r(a)\!)_\sharp) \, {}_\lambda  \, \overline{\varphi}_\infty(\!(\del^s(b)\!)_\sharp)]_{\sharp,2} \,,
\end{aligned}
\end{equation}
after using sesquilinearity, \eqref{Eq:VectId} and that $\overline{\varphi}$ is a morphism of Lie algebras. 
In particular, $\overline{\varphi}_\infty$ is a morphism of Lie vertex algebras and in turn $\varphi_\infty \colon \Jinf(\AA_1)\to\Jinf(\AA_2)$ is a $H_0$-Poisson vertex morphism. 
\end{proof}

Recall from Proposition \ref{Pr:DPAtoDPA} the functor $\mathtt{J}\colon \DPA\to \DPVA$. The image of this functor, denoted $\Jinf\DPA$, defines a subcategory of $\DPVA$, hence we can consider $\mathtt{J}\colon \DPA\to \Jinf\DPA$ as a fully faithful functor. 
Let $(\Jinf(\AA),\del,\dgal{-,-})$ be an object in $\Jinf\DPA$ defined from a double Poisson algebra $(\AA,\dgal{-,-})$. Consider the vector space $\Vect_\infty(\AA) \subset \Jinf(\AA)$ defined by \eqref{Eq:VectInf}. 
Combining \eqref{Eq:PtoPA-a} and sesquilinearity \eqref{Eq:DA1}, the double $\lambda$-bracket on $\Jinf(\AA)$ restricts to a bilinear map
\begin{equation*}
\mathrm{m} \circ \dgal{-_\lambda -} \colon  \Vect_\infty(\AA) \times \Vect_\infty(\AA) \longrightarrow \Vect_\infty(\AA)[\lambda]\,, \quad 
(a,b) \longmapsto \mathrm{m} \circ \dgal{a_\lambda b}\,.
\end{equation*}
This operation descends to $H_0(\Vect_\infty(\AA)\!)$ where it is the restriction of \eqref{Eq:HOvert-ind}, which we also denote by $[-_\lambda-]_\sharp$. 
The following result is a direct consequence of Proposition \ref{Pr:DPVA-HOPV} by restriction from $H_0(\Jinf(\AA)\!)$  to $H_0(\Vect_\infty(\AA)\!)$. 
\begin{corollary}  \label{Cor:JDPA-HOPV}
 The map $[-{}_\lambda-]_\sharp$ on $H_0(\Vect_\infty(\AA)\!)$ endows $\Vect_\infty(\AA)\subset \Jinf(\AA)$ with an $H_0$-Poisson vertex structure. 
 Furthermore, if $\theta\colon \AA_1\to \AA_2$ is a morphism in $\Jinf\DPA$, then $\theta$ is a $H_0$-Poisson vertex morphism relative to $\Vect_\infty(\AA_1)$ and $\Vect_\infty(\AA_2)$ endowed with the $H_0$-Poisson vertex structures obtained through \eqref{Eq:HOvert-ind}. 
\end{corollary}
Thus, the functor $\sharp_{\mathtt{V}}\colon \DPVA \to \HoPV$ restricts to a functor 
 $\sharp_{\mathtt{V}}\colon \Jinf\DPA \to \widehat{\HoPV}$.

%\subsubsection{General commutativity of the left face}  \label{ss:ComLeft}

\begin{proposition}[Commutativity of the left face] \label{Pr:ComLeft}  
 The following diagram is commutative: 
 \begin{center}
    \begin{tikzpicture}
 \node  (TopLeft) at (-2,1) {$\DPA$};
\node  (TopRight) at (2,1) {$\Jinf\DPA$};
\node  (BotLeft) at (-2,-1) {$\HoP$}; 
\node  (BotRight) at (2,-1) {$\widehat{\HoPV}$}; 
\path[->,>=angle 90,font=\small]  
   (TopLeft) edge node[above] {$\mathtt{J}$} (TopRight) ;
\path[->,>=angle 90,font=\small]  
    (BotLeft) edge node[above] {$\Vect_\infty$} (BotRight) ; 
\path[->,>=angle 90,font=\small]  
    (TopLeft) edge node[left] {$\sharp$}  (BotLeft) ;
\path[->,>=angle 90,font=\small]  
    (TopRight) edge node[right] {$\sharp_{\mathtt{V}}$} (BotRight) ;  
    %%% right diagram  
 \node  (TopLeft2) at (4,1) {$\AA$};
\node  (TopRight2) at (7.5,1) {$\Jinf\AA$};
\node  (BotLeft2) at (4,-1) {$H_0(\AA)$}; 
\node  (BotRight2) at (7.5,-1) {$H_0(\Vect_\infty(\AA))$};  
\path[->,>=angle 90,font=\small]  
   (TopLeft2) edge (TopRight2) ;
\path[->,>=angle 90,font=\small]  
    (BotLeft2) edge (BotRight2) ; 
\path[->,>=angle 90,font=\small]  
    (TopLeft2) edge (BotLeft2) ;
\path[->,>=angle 90,font=\small]  
    (TopRight2) edge (BotRight2) ;     
   \end{tikzpicture}
\end{center}
\end{proposition}

\begin{proof}  
Given a double Poisson algebra $(\AA,\dgal{-,-})$, we end up with the pair $\Vect_\infty(\AA)\subset \Jinf(\AA)$ in both cases, which is equipped with two $H_0$-Poisson vertex structures. Let us prove that these structures are the same; this amounts to show that the Lie vertex bracket takes the same value on arbitrary elements $(\del^r(a)\!)_\sharp, (\del^s(b)\!)_\sharp \in \Vect_\infty(\AA)$, where $r,s\geq 0$ and $a,b\in \AA$. 
On the one hand, we compute using \eqref{Eq:HOvert-ind} and Lemma \ref{Lem:DPAtoDPVA} (for $\sharp_{\mathtt{V}}$ and $\mathtt{J}$, respectively)
\begin{equation} \,
[(\del^r(a)\!)_\sharp{}_\lambda (\del^s(b)\!)_\sharp]_\sharp 
= (\mathrm{m} \circ \dgal{\del^r(a)_\lambda \del^s(b)})_\sharp 
=(\!(-\lambda)^r (\lambda+\del)^s\, \mathrm{m} \circ \dgal{a, b})_\sharp \,.
\end{equation}
On the other hand, we can apply Proposition \ref{Pr:HoP-HoPV} and \eqref{Eq:HO-ind} (for $\Vect_\infty$ and $\sharp$, respectively); we should note that $[a_\sharp,-]_\sharp$ is lifted by the derivation $\del_a:=\mathrm{m}\circ \dgal{a,-}\in \Der(\AA)$ which is extended to an element of $\Der(\Jinf(\AA),\Jinf(\AA)[\lambda])$ satisfying \eqref{Eq:HoP-cond}. This leads to 
\begin{equation} 
  \begin{aligned}
      \, 
[(\del^r(a)\!)_\sharp{}_\lambda (\del^s(b)\!)_\sharp]_\sharp &= 
(-\lambda)^r \, (\del_a (\del^s(b) )\!)_\sharp \\
&= (-\lambda)^r \, (\!(\lambda+\del)^s\, \del_a (b)\!)_\sharp=
 (\!(-\lambda)^r (\lambda+\del)^s\, \mathrm{m}\circ \dgal{a,b})_\sharp \,.
  \end{aligned}
\end{equation} 
The functors $\sharp,\sharp_{\mathtt{V}}$ send a morphism to itself while $\mathtt{J},\Vect_\infty$ send a morphism to its jet, if we only consider the algebra structure. 
Thus, for a morphism of double Poisson algebras $\theta\colon \AA_1\to \AA_2$, both approaches yield the morphism  $\theta_\infty\colon \Jinf(\AA_1)\to \Jinf(\AA_2)$ and its restrictions to $\Vect_\infty(\AA_1)$ and $H_0(\Vect_\infty(\AA_1)\!)$. 
\end{proof}

%\subsubsection{Link to representation algebras}  \label{ss:HOPV-rep}

We explain the importance of relative $H_0$-Poisson vertex structures when going to representation algebras. 
Fix a differential algebra $\VV$. 
We denote by $\VV_N^{\tr}\subset \VV_N^{\Gl_N}$ the differential subalgebra of the representation algebra generated by the elements $\{\tr(a)  \mid a\in \VV \}$ as in \S\ref{ss:PoiRedNC}. (By assumption on the field, if $\VV$ is of finite type then $\VV_N^{\tr}= \VV_N^{\Gl_N}$). 
Consider a vector space  $\hat{\VV}\subset \VV$ stable under the differential. We can similarly define $\hat{\VV}_N^{\tr}\subset \VV_N^{\Gl_N}$ as the subalgebra generated by the elements $\{\tr(a)  \mid a\in \hat{\VV} \}$, which is itself a differential algebra with differential satisfying $\del(\tr(a)\!)=\tr(\del(a)\!)$.  
It is clear that the trace map $\tr\colon \hat{\VV}\to \hat{\VV}_N^{\tr}$ factors through $H_0(\hat{\VV})$. 
We have the following generalisation of Crawley-Boevey's theory \cite{CB}. 

\begin{theorem} \label{THM:HoPV-tr} 
Assume that the pair $\hat{\VV}\subset \VV$ has a $H_0$-Poisson vertex structure $[-_\lambda -]_\sharp$. 
Then $\hat{\VV}_N^{\tr}$ is a Poisson vertex algebra if it is endowed with the unique $\lambda$-bracket satisfying   
\begin{equation} \label{Eq:relHoPV}
   \br{\tr(a)_\lambda \tr(b)} = \tr( [a_\sharp{}_\lambda b_\sharp]_\sharp)\,, \qquad a,b\in \hat{\VV}\,.
\end{equation}
Furthermore, this construction defines a functor $\mathtt{tr}_N\colon \widehat{\HoPV}\to \PVA$.  
\end{theorem}
\begin{proof}
The first part of the statement is essentially a $\lambda$-bracket version of \cite[Theorem 4.5]{CB}, whose proof we are adapting. 

Since a $\lambda$-bracket satisfies Leibniz rules, it is completely determined by its value on generators. Hence we get uniqueness since $\hat{\VV}_N^{\tr}$ is generated by the elements, $\tr(a)$, $a\in \hat{\VV}$. 

Let us prove that a $\lambda$-bracket satisfying \eqref{Eq:relHoPV} exists. 
Fix $a\in \hat{\VV}$, and let $\del_a\in \Der(\VV,\VV[\lambda])$ be a derivation inducing the linear map $[a_\sharp{}_\lambda -]_\sharp$. 
We denote by\footnote{Instead of abusively using $\del_a$ for the induced derivation, the notation $D_a^{(\lambda)}$ will be important for 2 reasons: \texttt{1)} we shall change the indeterminate $\lambda$; \texttt{2)} it emphasizes the independence that we are going to prove of $D_a^{(\lambda)}$ with respect to the lift $\del_a$ after restriction to $\hat{\VV}_N^{\tr}$.} $D_a^{(\lambda)}\in \Der(\VV_N,\VV_N[\lambda])$ the induced derivation satisfying $D_a^{(\lambda)}(b_{ij})=(\del_a(b)\!)_{ij}$ for any $b\in \VV$. 
We note that for any $b\in \hat{\VV}$, 
\begin{equation*}
    D_a^{(\lambda)}(\tr(b)\!)=\tr(\del_a(b)\!)=\tr([a_\sharp{}_\lambda b_\sharp]_\sharp) \in \hat{\VV}_N^{\tr}[\lambda]\,.
\end{equation*}
Since $D_a^{(\lambda)}$ is a derivation on the whole of $\VV_N$, this gives for any $b_1,\ldots,b_\ell \in \hat{\VV}$
\begin{equation} \label{Eq:HoPV-pf1}
   D_a^{(\lambda)}(\prod_{\ell_0=1}^\ell \tr(b_{\ell_0})\!)
   =\sum_{1\leq \ell_0\leq \ell}\tr([a_\sharp{}_\lambda (b_{\ell_0})_\sharp]_\sharp) \, \,
   \prod_{\ell_1\neq \ell_0} \tr(b_{\ell_1}) \in \hat{\VV}_N^{\tr}[\lambda]  \,.  
\end{equation}
Hence by linearity $D_a^{(\lambda)}\in \Der(\hat{\VV}_N^{\tr}, \hat{\VV}_N^{\tr}[\lambda])$, and it is readily seen from \eqref{Eq:HoPV-pf1} that $D_a^{(\lambda)}$ is independent of the lift $\del_a$ that we picked (when seen as a map on $\hat{\VV}_N^{\tr}$). Since $\del_a$ satisfies \eqref{Eq:HoP-cond}, we also get 
\begin{equation} \label{Eq:HoPV-pf2}
    D_a^{(\lambda)} \circ \del = (\del+\lambda) \circ D_a^{(\lambda)}\,,
\end{equation}
when evaluated on generators of $\hat{\VV}_N^{\tr}$, hence this holds on $\hat{\VV}_N^{\tr}$.

Let us express two elements $\tilde{F}_1,\tilde{F}_2 \in \hat{\VV}_N^{\tr}$ in terms of generators as
\begin{equation*}
\tilde{F}_1=\sum_{\underline{a}=(a_1,\ldots,a_k)} \gamma_{\underline{a}}^{(1)}\,\prod_{k_0=1}^k \tr(a_{k_0})  \,, \qquad 
\tilde{F}_2=\sum_{\underline{b}=(b_1,\ldots,b_\ell)} \gamma_{\underline{b}}^{(2)}\,\prod_{\ell_0=1}^\ell \tr(b_{\ell_0})\,,   
\end{equation*}
where we sum over a finite number of tuples of elements in $\hat{\VV}$ with $\gamma_\bullet^{(i)}\in \kk$. 
With this notation, we define a bilinear map $\br{-_\lambda-}\colon \hat{\VV}_N^{\tr}\times \hat{\VV}_N^{\tr}\to \hat{\VV}_N^{\tr}[\lambda]$ by 
\begin{equation}
\begin{aligned} \label{Eq:HoPV-pf3}
     \br{\tilde{F}_1 {}_\lambda \tilde{F}_2}= 
\sum_{\underline{a},\underline{b}} \gamma_{\underline{a}}^{(1)} \gamma_{\underline{b}}^{(2)} 
\sum_{ \substack{1\leq k_0\leq k\\ 1\leq \ell_0\leq\ell} } 
& \left(\prod_{\ell_1\neq \ell_0} \tr(b_{\ell_1})\right) \\
&\quad \tr([(a_{k_0})_\sharp{}_{\lambda+x} (b_{\ell_0})_\sharp]_\sharp) \, 
\left(\Big|_{x=\del} \prod_{k_1\neq k_0} \tr(a_{k_1})\right) \, . 
\end{aligned}
\end{equation}
This map is well-defined in the second argument because we can write from \eqref{Eq:HoPV-pf1} 
\begin{equation} \label{Eq:HoPV-pf4}
     \br{\tilde{F}_1 {}_\lambda \tilde{F}_2}= 
\sum_{\underline{a}} \gamma_{\underline{a}}^{(1)} 
\sum_{ 1\leq k_0\leq k }  
D_{a_{k_0}}^{(\lambda+x)}\big(\tilde{F}_2 \big)  
\left(\Big|_{x=\del} \prod_{k_1\neq k_0} \tr(a_{k_1})\right) \, , 
\end{equation}
which is independent of the expression for $\tilde{F}_2$. If we can show that this map satisfies the skewsymmetry \eqref{Eq:A2}, then it will be independent of the expression for $\tilde{F}_1$ as well, and therefore it will be well-defined. 

In order to show skewsymmetry, we compute for any  $\tilde{G},\tilde{H} \in \hat{\VV}_N^{\tr}$ and $a,b\in \hat{\VV}$ 
\begin{equation*}
\begin{aligned}
     \tilde{G} \tr([a_\sharp{}_{\lambda+x} b_\sharp]_\sharp) \, 
\Big|_{x=\del} \tilde{H} 
&=-\tilde{G} \left(\Big|_{y=\del} \tr([b_\sharp{}_{-\lambda-x-y} a_\sharp]_\sharp) \right)\, 
\Big|_{x=\del} \tilde{H}  \\
&=-\Big|_{y=\del}  \left(\tilde{H} \, \tr([b_\sharp{}_{-\lambda-y+z} a_\sharp]_\sharp) \, \Big|_{z=\del} \tilde{G} \right)\, ,
\end{aligned}
\end{equation*}
where we used skewsymmetry of the Lie vertex bracket $[-_\lambda-]_\sharp$ in the first equality, 
then the identity $\alpha \,\del^k(\beta)=(\del-z)^k(\beta\, \big|_{z=\del}\alpha)$ which holds for any $k\geq 0$ on elements $\alpha,\beta$ in a differential algebra. 
Applying this computation to \eqref{Eq:HoPV-pf3} yields the expression for $-\big|_{x=\del} \br{\tilde{F}_2 \, {}_{-\lambda-x} \tilde{F}_1}$, hence skewsymmetry \eqref{Eq:A2} holds. 

It is plain from \eqref{Eq:HoPV-pf4} that $\dgal{-_\lambda-}$ satisfies the left Leibniz rule \eqref{Eq:lL}; by skewsymmetry it must satisfy the right Leibniz rule \eqref{Eq:rL}. 
Combining \eqref{Eq:HoPV-pf2} and \eqref{Eq:HoPV-pf4}, we also get sesquilinearity \eqref{Eq:A1} for  $\dgal{-_\lambda-}$. Thus, it remains to check Jacobi identity \eqref{Eq:A3}. 

In a differential algebra $(V,\del)$ with a $\lambda$-bracket $\br{-_\lambda-}$, a  standard computation using the sesquilinearity \eqref{Eq:A1} and the right Leibniz rule \eqref{Eq:rL} yields:
\begin{equation*}
\begin{aligned}
    \mathrm{Jac}_{\lambda,\mu}(a,b,c)&=
    \Big|_{x=\del}  \mathrm{Jac}_{\mu,-\lambda-\mu-x}(b,c,a) \,,\\
    \mathrm{Jac}_{\lambda,\mu}(a,b,cd)&=c\, \mathrm{Jac}_{\lambda,\mu}(a,b,d) + d\, \mathrm{Jac}_{\lambda,\mu}(a,b,c)\,,
\end{aligned}
\end{equation*}
with $a,b,c,d\in V$ and where $\mathrm{Jac}_{\lambda,\mu}(a,b,c)$ denotes the left-hand side of \eqref{Eq:A3}. 
These two equalities guarantee that Jacobi identity is equivalent to the vanishing of $\mathrm{Jac}_{\lambda,\mu}\colon (\hat{\VV}_N^{\tr})^{\times 3}\to \hat{\VV}_N^{\tr}[\lambda,\mu]$ on generators of $\hat{\VV}_N^{\tr}$. We can conclude as for any  $a,b,c\in \hat{\VV}$, we have in $\hat{\VV}_N^{\tr}$
\begin{equation*} 
  \mathrm{Jac}_{\lambda,\mu}(\tr(a),\tr(b),\tr(c)\!)  
 =\tr\Big(
[a_\sharp{}_{\lambda} [b_\sharp{}_{\mu} c_\sharp]_\sharp]_\sharp 
- [b_\sharp{}_{\mu} [a_\sharp{}_{\lambda} c_\sharp]_\sharp]_\sharp 
-[[a_\sharp{}_{\lambda} b_\sharp]_\sharp {}_{\lambda+\mu} c_\sharp]_\sharp
 \Big)\,, 
\end{equation*}
and this is zero since $[-_{\lambda} -]_\sharp$ is a Lie vertex bracket on $H_0(\hat{\VV})$. 

Consider two pairs $\hat{\VV}_i\subset \VV_i$, $i=1,2$, with $H_0$-Poisson vertex structures, and assume that $\varphi\colon \VV_1\to \VV_2$ is a $H_0$-Poisson vertex morphism relative to those.
The unique induced morphism $\varphi\colon (\VV_1)_N\to (\VV_2)_N$ that satisfies $\varphi(a_{ij})=(\varphi(a)\!)_{ij}$ for $a\in \VV_1$ and $1\leq i,j\leq N$ restricts to the morphism 
$\tr_N(\varphi)\colon (\hat{\VV}_1)_N^{\tr}\to (\hat{\VV}_2)_N^{\tr}$, $\tr_N(\varphi)(\tr_N(a)\!):=\tr_N(\varphi(a)\!)$, because $\varphi(\hat{\VV}_1)\subset  \hat{\VV}_2$ by definition. 
Moreover, we can compute that for any $a,b\in \hat{\VV}_1$, (omitting the dimension index~$N$)
\begin{equation*} 
    \tr(\varphi) (\br{\tr(a)_\lambda \tr(b)}_1) 
    = \tr( \varphi([a_\sharp{}_\lambda b_\sharp]_{\sharp,1})\!)     
=\br{\tr(\varphi)(\tr(a)\!)_\lambda \tr(\varphi)(\tr(b)\!)}_2\,, 
\end{equation*}
where we used \eqref{Eq:relHoPV} and the fact that $\varphi$ intertwines the Lie vertex brackets. Hence $\tr_N(\varphi)$ is a morphism of Poisson vertex algebras. 
\end{proof}

\begin{corollary}    \label{Cor:HoPV-tr} 
Assume that $\VV$ has a $H_0$-Poisson vertex structure $[-_\lambda -]_\sharp$. 
Then $\VV_N^{\tr}$ is a Poisson vertex algebra if it is endowed with the unique $\lambda$-bracket satisfying   
\begin{equation} 
   \br{\tr(a)_\lambda \tr(b)} = \tr( [a_\sharp{}_\lambda b_\sharp]_\sharp)\,, \qquad a,b\in\VV\,.
\end{equation}
Furthermore, this construction defines a functor $\mathtt{tr}_N\colon \HoPV\to \PVA$ which is the restriction of $\mathtt{tr}_N\colon \widehat{\HoPV}\to \PVA$.   
\end{corollary}
\begin{proof}
We see $\HoPV$ as a subcategory of $\widehat{\HoPV}$ made of pairs $\hat{\VV}\subset \VV$ with $\hat{\VV}=\VV$ and then apply Theorem \ref{THM:HoPV-tr}. 
\end{proof}

\begin{remark}
Consider the Definition \ref{Def:HOPV} of a $H_0$-Poisson vertex structure, and assume that $\del=0$ on $\VV$. 
Letting $\lambda=0$ in the Lie vertex bracket and in \eqref{Eq:HoP-cond}, we recover Definition \ref{Def:HOP} of a $H_0$-Poisson structure. 
Repeating the above proofs in that case, Corollary \ref{Cor:HoPV-tr}  yields that $\VV_N^{\tr}$ (with the zero differential and $\lambda$-bracket evaluated at $\lambda=0$) is in fact a Poisson bracket. In other words, this turns $\VV_N^{\tr}$ into a Poisson algebra. We recover from this case Crawley-Boevey's result  \cite[Theorem 4.5]{CB} (over a base field $\kk$).

Similarly, consider the case of a $H_0$-Poisson vertex structure on the pair $\hat{\VV}\subset \VV$ with $\del=0$ in which we set $\lambda=0$. 
Then the same reasoning based on Theorem \ref{THM:HoPV-tr} yields that $\hat{\VV}_N^{\tr}$ is equipped with a Poisson bracket. This turns $\hat{\VV}_N^{\tr}$ into a Poisson algebra, which is a relative version of  \cite[Theorem 4.5]{CB} (over a base field $\kk$). 
This relative version of $H_0$-Poisson structures was not explicitly spelled out as part of \S\ref{ss:PoiRedNC} because we will not be using this observation in the remainder of the paper. 
\end{remark}

\subsection{Interpretation of the functors in the presence of jets} \label{ss:Interpret}
Given a $H_0$-Poisson structure on some $\AA$, recall from Proposition \ref{Pr:HoP-HoPV} that we get a pair $\Vect_\infty(\AA)\subset \Jinf(\AA)$ in $\widehat{\HoPV}$. 
Since $\AA$ is assumed to be of finite type, we have $(\Vect_\infty(\AA)\!)_N^{\tr}\simeq \Jinf(\AA_N^{\Gl_N})$ as differential algebras under the identification $\tr(\del^r(a)\!)\stackrel{\sim}{\longrightarrow} \del^r (\tr(a)\!)$ for any $r\geq 0$ and $a\in \AA$. 
We then get the map $j_{\AA_N}:(\Vect_\infty(\AA)\!)_N^{\tr}\to (\Jinf(\AA)\!)^{\Jinf(\Gl_N)}$, which allows to complete the back face of the cube from Figure \ref{Fig:PoiRed}. 
Note that, to define this map, we need to ensure that we end up in $\Jinf \PA_{\Gl_N;0}$ after applying the functor $\tr_N:\widehat{\HoPV}\to \PVA$ from Theorem \ref{THM:HoPV-tr}. 
Hence we need to define a suitable category of jets inside $\widehat{\HoPV}$. 

Consider the image of $\HoP$ under the functor $\Vect_\infty:\HoP\to \widehat{\HoPV}$ from Proposition \ref{Pr:HoP-HoPV}. It is not hard to see that we can compose morphisms in the image since they are just jets of morphisms from $\HoP$. Therefore the image of $\Vect_\infty$ defines a subcategory of $\widehat{\HoPV}$, which we label $\mathtt{V}_\infty \HoP$. 
By construction, we can restrict the functor from Proposition \ref{Pr:HoP-HoPV} to $\Vect_\infty:\HoP\to \mathtt{V}_\infty \HoP$. We shall now see that the other two functors involving $\widehat{\HoPV}$ that we have constructed can be reformulated using $\mathtt{V}_\infty \HoP$.

\begin{lemma} \label{Lem:SharpV-res}
The functor $\sharp_{\mathtt{V}}:\Jinf\DPA\to \widehat{\HoPV}$ from Corollary \ref{Cor:JDPA-HOPV} restricts to a functor 
$\sharp_{\mathtt{V}}:\Jinf\DPA\to \mathtt{V}_\infty \HoP$.  
\end{lemma}
\begin{proof}
 Observe from Corollary \ref{Cor:JDPA-HOPV} that the functor $\sharp_{\mathtt{V}}:\Jinf\DPA\to \widehat{\HoPV}$ takes value in $\mathtt{V}_\infty \HoP$ on objects. 
Consider a morphism $\theta_\infty:\Jinf(\AA_1)\to \Jinf(\AA_2)$ in  $\Jinf\DPA$. 
By definition, it is the jet of a morphism $\theta:\AA_1\to \AA_2$ in $\DPA$. 
Meanwhile, the morphism that $\theta_\infty$ induces in $\widehat{\HoPV}$ is given by $\theta_\infty$ (seen as a morphism of differential algebras) which restricts to a morphism of Lie vertex algebras $\bar{\theta}_\infty:H_0(\Vect_\infty(\AA_1))\to H_0(\Vect_\infty(\AA_2))$. 
Our aim is to show that the later morphism is of the form $\Vect_\infty(\theta')$ for some morphism $\theta':\AA_1\to \AA_2$ in $\HoP$. 
This is true by taking $\theta'=\sharp(\theta)$ thanks to Proposition \ref{Pr:ComLeft}. 
\end{proof}

\begin{lemma} \label{Lem:HoPV-PVA-res}
The functor $\mathtt{tr}_N\colon \widehat{\HoPV}\to \PVA$ from Theorem \ref{THM:HoPV-tr} restricts to a functor 
$\mathtt{tr}_N\colon \mathtt{V}_\infty \HoP\to \Jinf\PA_{\Gl_N;0}$ by sending  
$\Vect_\infty(\AA)\subset \Jinf(\AA)$ to $(\Jinf(\AA)_N,(\Vect_\infty(\AA))_N^{\tr})$.  
\end{lemma} 
\begin{proof}
 The functor is simply defined on objects by adding the datum of the differential algebra $\Jinf(\AA)$ to 
the Poisson vertex algebra $\hat{\AA}_N^{\tr}$. 
Fix a $H_0$-Poisson vertex morphism $\varphi\colon \AA_1\to \AA_2$ relative to pairs $\Vect_\infty(\AA_i)\subset \AA_i$, $i=1,2$. On representation algebras, we get the $\Gl_N$-equivariant morphism $\varphi_N  \colon (\AA_1)_N\to (\AA_2)_N$ which defines the morphism of Poisson algebras $\tr_N(\varphi)  \colon (\AA_1)_N^{\Gl_N}\to (\AA_2)_N^{\Gl_N}$. 
Taking jets, we have the morphism $(\varphi_N)_\infty \colon \Jinf(\AA_1)_N\to \Jinf(\AA_2)_N$  which restricts to 
$\Jinf((\AA_1)_N^{\Gl_N}) \to  \Jinf((\AA_2)_N^{\Gl_N})$. 
To conclude, it suffices to remark that this is the morphism of Poisson vertex algebras 
$\tr_N(\varphi_\infty):(\Vect_\infty(\AA_1))_N^{\tr} \to  (\Vect_\infty(\AA_2))_N^{\tr}$ from the proof of Theorem \ref{THM:HoPV-tr} after using the identifications 
$\Jinf((\AA_i)_N^{\Gl_N})\simeq (\Vect_\infty(\AA_i))_N^{\tr}$, $i=1,2$. 
\end{proof}

\subsection{Main result}
Fix $N\geq 1$. We are in position to prove the next result. 
\begin{theorem}[Commutativity of the cube]  
\label{Thm:ComPoissonRed}
 The  diagram depicted in Figure \ref{Fig:PoiRed} is commutative. 
\end{theorem}
The proof consists in checking that all 6 faces of the cube are commuting. 
Commutativity of the top and right faces  
was respectively obtained in Theorem \ref{Thm:QJ-Rep} and in Lemma \ref{Lem:RightF} with $\GG=\Gl_N$.

\subsubsection*{Front face} 

In Proposition \ref{Pr:ComFront}, we proved that the front face commute if we end up in $\PA$. 
We noticed in \S\ref{ss:RedPVA} that $\PA$ can be replaced with $\PA_{\Gl_N;0}$.

\subsubsection*{Left face} 

In Proposition~\ref{Pr:ComLeft}, we proved that the front face commutes if we end up in $\widehat{\HoPV}$. 
It follows from Lemmas \ref{Lem:SharpV-res} and \ref{Lem:HoPV-PVA-res} that we can take $\mathtt{V}_\infty \HoP$ as the target category. 

\subsubsection*{Bottom face}\label{subsubbottom}
We need to compare the functors 
\begin{equation*}
\HoP \stackrel{\Vect_\infty}{\longrightarrow} \mathtt{V}_\infty \HoP  
\stackrel{\mathtt{tr}_N}{\longrightarrow} \Jinf \PA_{\Gl_N;0}\,, \qquad 
\HoP \stackrel{\mathtt{tr}_N}{\longrightarrow} \PA_{\Gl_N;0} 
\stackrel{\mathtt{J}}{\longrightarrow} \Jinf \PA_{\Gl_N;0}\,.
\end{equation*} 
We noticed that they are the restrictions of the following functors 
\begin{equation*}
\HoP \stackrel{\Vect_\infty}{\longrightarrow} \widehat{\HoPV}  
\stackrel{\mathtt{tr}_N}{\longrightarrow} \PVA\,, \qquad 
\HoP \stackrel{\mathtt{tr}_N}{\longrightarrow} \PA 
\stackrel{\mathtt{J}}{\longrightarrow} \PVA\,,
\end{equation*} 
so we shall check that these two compositions are equal. 

Fix $\AA$ with a $H_0$-Poisson structure, so that $H_0(\AA)$ is a Lie vertex algebra. 
Under these two functors, we end up with $(\Vect_\infty(\AA)\!)_N^{\tr}$ and $\Jinf(\AA_N^{\Gl_N})$, which are both isomorphic. We can identify these two algebras explicitly through 
\begin{equation} \label{Eq:IsoVinfRep}
    (\Vect_\infty(\AA)\!)_N^{\tr}\longrightarrow \Jinf(\AA_N^{\Gl_N})\,, \qquad 
    \tr(\del^r(a)\!)\mapsto \del^r(\tr(a)\!)\,, \quad a\in \AA,\,\, r \geq 0\,.
\end{equation}
Let us check that this is an isomorphism of Poisson vertex algebras. 
For $a,b\in \AA$ and $r,s\geq 0$, we have on the one hand  
\begin{equation} \,
\br{\tr(\del^r(a)\!) {}_\lambda \tr(\del^s(b)\!) }
=\tr([\del^r(a)_{\sharp} {}_\lambda \del^s(b)_\sharp]_\sharp) 
=\tr(\!(-\lambda)^r (\lambda+\del)^s\,[a_\sharp , b_\sharp]_\sharp) \,, 
\end{equation}
after using \eqref{Eq:relHoPV} and Proposition \ref{Pr:HoP-HoPV} (for $\mathtt{tr}_N$ and $\Vect_\infty$, respectively). 
On the other hand, we compute 
\begin{equation} \,
\br{\del^r \tr(a) {}_\lambda \del^s \tr(b) }
=(-\lambda)^r (\lambda+\del)^s  \, \br{\tr(a),\tr(b)}
=(-\lambda)^r (\lambda+\del)^s\,\tr([a_\sharp , b_\sharp]_\sharp) \,, 
\end{equation}
by applying Lemma \ref{Lem:PAtoPVA} and \eqref{Eq:relHoP} (for $\mathtt{J}$ and $\mathtt{tr}_N$, respectively). It is clear that the two $\lambda$-brackets coincide on generators under the identification \eqref{Eq:IsoVinfRep}, hence they are the same.  

Fix a morphism $\varphi\colon \AA_1\to \AA_2$ in $\HoP$. 
We see from the proof of Theorem \ref{THM:HoPV-tr} that $\varphi$ is sent under the first functor on the restriction of $(\varphi_\infty)_N\colon (\Jinf(\AA_1)\!)_N\to (\Jinf(\AA_2)\!)_N$ to $(\Vect_\infty(\AA_1)\!)_N^{\tr}$. in particular, the morphism hence obtained satisfies $\tr(\del^r(a)\!)\mapsto \tr(\del^r\,\varphi(a)\!)$ for any $a\in \AA_1$ and $r\geq 0$. 
Under the second functor, $\varphi$ is mapped on the jet of the restriction of the induced map $\varphi_N\colon (\AA_1)_N\to (\AA_2)_N$, which satisfies $\del^r(\tr(a)\!)\mapsto \del^r(\tr(\varphi(a)\!)$. The 2 morphisms obviously coincide under the identification \eqref{Eq:IsoVinfRep}.

\subsubsection*{Back face} 
  
We need to compare the functors 
\begin{equation*}
\Jinf\DPA \stackrel{(-)_N}{\longrightarrow} \Jinf\PA^{\Gl_N}
\stackrel{\mathtt{R}_\infty}{\longrightarrow} \PVA\,, \qquad 
\Jinf\DPA \stackrel{\sharp_\mathtt{V}}{\longrightarrow} \mathtt{V}_\infty \HoP 
\stackrel{\mathtt{tr}_N}{\longrightarrow} \Jinf \PA_{\Gl_N;0}
\stackrel{j_{(-)}}{\longrightarrow} \PVA\,.
\end{equation*} 
Fix $\Jinf(\AA)$ in $\Jinf\DPA$. 
Under these two functors, we end up with $(\Jinf(\AA)\!)^{\Jinf(\Gl_N)}$.  
We need to consider the natural morphism $j_{\AA_N}\colon \Jinf(\AA_N^{\Gl_N})\to (\Jinf(\AA)\!)^{\Jinf(\Gl_N)}$ given in \eqref{Eq:InvMorph} which is defined on generators by $\del^r(\tr(a)\!)\mapsto \tr(\del^r(a)\!)$ for any $a\in \AA$ and $r\geq 0$. 

To compare the induced $\lambda$-brackets, it suffices to do so on generators of $\Jinf(\AA_N^{\Gl_N})$ sent into $(\Jinf(\AA)\!)^{\Jinf(\Gl_N)}$ by $j_{\AA_N}$. 
For $a,b\in \AA$ and $r,s\geq 0$, we have from the first functor   
\begin{equation*} \,
\begin{aligned}
  \br{\tr(\del^r(a)\!) {}_\lambda \tr(\del^s(b)\!) }
=\sum_{i,j=1}^N 
\br{(\del^r(a)\!)_{ii} {}_\lambda (\del^s(b)\!)_{jj} }  
=\tr(\mathrm{m}\circ \dgal{\del^r(a) {}_\lambda \del^s(b) }) \,,
\end{aligned}
\end{equation*}
where we used \eqref{Eq:relDPVA}. 
The second functor yields 
\begin{equation*} \,
\begin{aligned}
  \br{\tr(\del^r(a)\!) {}_\lambda \tr(\del^s(b)\!) }
= \tr([(\del^r(a)\!)_\sharp{}_\lambda (\del^s(b)\!)_\sharp]_\sharp)
=\tr(\mathrm{m}\circ \dgal{(\del^r(a)\!) {}_\lambda (\del^s(b)\!)})\,, 
\end{aligned}
\end{equation*}
thanks to \eqref{Eq:relHoPV} and Corollary \ref{Cor:JDPA-HOPV} (for $\mathtt{tr}_N$ and $\sharp_\mathtt{V}$, respectively, after restriction to $\mathtt{V}_\infty\HoP$ according to \S\ref{ss:Interpret}).

Finally, we fix a morphism  $\theta\colon \Jinf(\AA_1)\to \Jinf(\AA_2)$ in $\Jinf \DPA$. By spelling out the functors, we obtain the morphisms  
$\mathtt{R}_\infty(\theta_N), \tr_N(\sharp(\theta)\!) \colon (\Jinf(\AA_1)\!)^{\Jinf(\Gl_N)}\to (\Jinf(\AA_2)\!)^{\Jinf(\Gl_N)}$ which coincide on any element $\tr(\del^r(a)\!)$, $a\in \AA_1$, $r\geq 0$. Indeed, this follows from a simple computation analogous to \eqref{Eq:ComFrom}. 
Since both morphisms are equal when restricted to generators, they are identically equal.

%%%%%%%%%%%%%%%%%% NEW SECTION %%%%%%%%%%%%%%%%%%%%%
%%%%%%%%%%%%%%%%%% NEW SECTION %%%%%%%%%%%%%%%%%%%%%
%%%%%%%%%%%%%%%%%% NEW SECTION %%%%%%%%%%%%%%%%%%%%%
%%%%%%%%%%%%%%%%%% NEW SECTION %%%%%%%%%%%%%%%%%%%%%

\section{Semisimple version of Poisson reduction} 
\label{Sec:SemiS}

We can derive an analogue of Theorem \ref{Thm:ComPoissonRed} in the presence of idempotents. 
The constructions needed for its proof are extremely close to the considerations of Section \ref{S:PoiRed}. Since we will be working with an idempotent decomposition when dealing with Hamiltonian reduction in Section \ref{Sec:Momap}, we shall leave all details to the reader and instead we will focus on describing the different nodes of this new commutative cube in an example as part of \S\ref{ss:SemiS-example}.

\subsection{General statement}\label{subsecgenstat}

Fix a $\kk$-algebra $B$. 
We denote by ${}_B\DPA$ the subcategory of $\DPA$ made of double Poisson algebras \emph{relative to} $B$. 
An object $\AA$ in ${}_B\DPA$ contains $B$ as a subalgebra and has a $B$-linear double bracket, 
which means that $\dgal{a,B}=0$ identically for any $a\in \AA$. 
A morphism $\theta:\AA_1\to \AA_2$ in ${}_B\DPA$ must restrict to the identity on $B$, i.e. $\theta\big|_B=\id_B$. 

By adapting the definition of ${}_B\DPA$ from $\DPA$ that we have just given, 
we can introduce the categories ${}_B\HoP$ and ${}_B\DPVA,{}_B\HoPV,{}_B\widehat{\HoPV}$ relative to $B$; in the last 3 cases, we require the differentials to be $B$-linear as well, i.e. $\del(B)=0$. 
We define $\Jinf {}_B\DPVA$ and $\mathtt{V}_\infty{}_B\HoP$ as images of the functors $\mathtt{J}:{}_B\DPA\to {}_B\DPVA$ and $\mathtt{V}_\infty:{}_B\HoP\to {}_B\widehat{\HoPV}$ introduced as $B$-linear versions of Propositions \ref{Pr:DPAtoDPA} and \ref{Pr:HoP-HoPV}. 

Let us now assume that $B=\oplus_{s\in S} \kk e_s$ is a finite-dimensional semisimple algebra where the $(e_s)_{s\in S}$ are orthogonal idempotents. 
Fix $\underline{n}=(n_s)\in \Z_{\geq0}^S$ and set $\Gl_{\underline{n}}:=\prod_{s\in S}\Gl_{n_s}$. 
To motivate the following relative version of Theorem \ref{Thm:ComPoissonRed}, note that we have the representation functor 
$(-)_{\underline{n}}\colon {}_B\mathtt{DP(V)A}\to ( \Jinf)\PA^{\Gl_{\underline{n}}}$, see \S\ref{ss:HamRedDPA}.

\begin{theorem}  \label{Thm:rel-ComPoissonRed}
Consider the  diagram depicted in Figure \ref{Fig:PoiRed} where the categories appearing on the left are taken to be relative to  $B=\oplus_{s\in S} \kk e_s$, 
and the categories appearing on the right are obtained by replacing $\Gl_n$ with $\Gl_{\underline{n}}$. 
Then this $B$-relative diagram is commutative. 
\end{theorem}

We leave the proof of this result to the reader; e.g. the commutativity of the right face is simply Lemma \ref{Lem:RightF} with $\GG=\Gl_{\underline{n}}$. Let us note that we will make some of the $B$-relative functors from the $B$-relative diagram explicit as part of the Hamiltonian reduction treated in Section \ref{Sec:Momap}.  
When it is clear from the context that we work over $B$, we shall simply write ${}_B\DPA$ as $\DPA$ and do the same for all the other relative categories.

\subsection{Example} 
\label{ss:SemiS-example}

\begin{remark}[Convention for path algebras]
\label{Rem:ConvQuiver}
In this paper, we write  paths from left to right as in \cite{VdB}.
Hence the path algebra $\kk Q$ of a quiver $Q$ is the $\kk$-algebra generated by the elements in $Q\cup \{e_s \mid s\in S\}$ subject to the relations
\begin{equation} \label{Eq:ConvPath}
  e_r e_s=\delta_{rs}e_r\,, \quad 1=\sum_{s\in S} e_s\,,   \quad
  a=e_{t(a)}a e_{h(a)}\,, \quad \text{ for }r,s\in S,\,\, a\in Q.
\end{equation}
As in \S\ref{ss:double_quiver}, $\kk Q$ is regarded as an algebra over $B=\oplus_{s\in S} \kk e_s$.
\end{remark}

Fix integers\footnote{Although it can be considered, we omit the case $c=0$ because it yields trivial brackets in the construction presented below.} $c\geq 1$ and $p,q\geq c$. Assume that $\kk=\bar{\kk}$ is of characteristic $0$.
When we refer to an earlier result from the text, it will be assumed that we consider its $B$-relative version (for $B=\kk e_1\oplus \kk e_2$ as defined below).

%\subsubsection{} 
Consider the quiver $Q_{p,q}$ made of the vertex set $\{1,2\}$ and arrows $v_1,\ldots,v_p:1\to 2$, $w_1,\ldots,w_q:2\to 1$. 

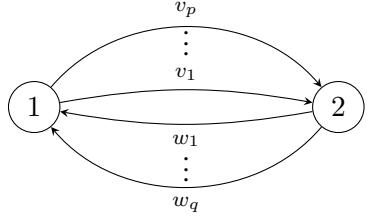
\begin{center}
   \begin{tikzpicture} 
\node[circle,draw] (vA)  at (-2,0) {$1$};
\node[circle,draw] (vB)  at (2,0) {$2$};
  \node  (dot1) at (0,0.95) {$\vdots$};
  \node  (dot2) at (0,-0.72) {$\vdots$};
% % v,w 
   \path[->,>=stealth,font=\scriptsize]  % v at 0
   (vA) edge [bend left=10,looseness=1] node[above] {$v_1$}  (vB) ;
   \path[->,>=stealth,font=\scriptsize]
   (vA) edge [bend left=50,looseness=1] node[above] {$v_p$} (vB) ;
      \path[->,>=stealth,font=\scriptsize]  % v at 0
   (vB) edge [bend left=10,looseness=1] node[below] {$w_1$}  (vA) ;
   \path[->,>=stealth,font=\scriptsize]
   (vB) edge [bend left=50,looseness=1] node[below] {$w_q$} (vA) ;
   \end{tikzpicture}  
\end{center}

The path algebra $\kk Q_{p,q}$ is generated over $\kk$ by the orthogonal idempotents $e_1,e_2$ and the arrows $\{v_{p'},w_{q'} \mid 1\leq p'\leq p,\,\, 1\leq q'\leq q\}$ subject to the relations 
\begin{equation*}
 e_1+e_2=1\,, \quad v_{p'}=e_1v_{p'}e_2\,, \quad w_{q'}=e_2w_{q'}e_1\,.
\end{equation*}
It is clear that $\kk Q_{p,q}$ contains  $B=\kk e_1\oplus \kk e_2$ as a subalgebra. 
We note that $\kk Q_{p,q}$ has a $B$-linear double Poisson bracket given by 
\begin{equation} \label{Eq:SemiS-1}
 \begin{aligned}
\dgal{v_{p_1},w_{q_1}}&=\delta_{p_1,q_1} \delta_{(p_1\leq c)}\, e_2\otimes e_1 \,, \quad 
\dgal{v_{p_1},v_{p_2}}&=0=\dgal{w_{q_1},w_{q_2}}\,, \quad
 \end{aligned}
\end{equation}
for any $1\leq p_1,p_2\leq p$ and $1\leq q_1,q_2\leq q$. 
We see from \S\ref{ss:double_quiver}  that the case $p=q=c$ is a double Poisson algebra obtained from the double of a quiver by Van den Bergh \cite[\S6.3]{VdB}; taking $p>c$ and/or $q>c$ just amounts to add arrows having zero double bracket with all the other elements.

%\subsubsection{}  %%DPVA

The double Poisson vertex algebra $\Jinf(\kk Q_{p,q})$ is isomorphic as a differential algebra to $\kk Q_{p,q}^\infty$, where $Q_{p,q}^\infty$ has arrows 
\begin{equation*}
 v_{p'}^{(\ell)}:1\to 2\,, \quad w_{q'}^{(\ell)}\,, \quad 
 1\leq p'\leq p,\,\, 1\leq q'\leq q,\,\, \ell\geq 0\,.
\end{equation*}
The $B$-linear differential $\del$ acts as $v_{p'}^{(\ell)}\mapsto v_{p'}^{(\ell+1)}$, $w_{q'}^{(\ell)}\mapsto w_{q'}^{(\ell+1)}$. 
The natural $B$-linear embedding $\kk Q_{p,q}\hookrightarrow \Jinf(\kk Q_{p,q})$ satisfies $v_{p'}\mapsto v_{p'}^{(0)}$, $w_{q'}\mapsto w_{q'}^{(0)}$. 
The double $\lambda$-bracket induced by Lemma \ref{Lem:DPAtoDPVA} is then determined by 
\begin{equation} \label{Eq:SemiS-2}
\dgal{v_{p'}^{(\ell)}  {}_\lambda w_{q'}^{(m)}}=(-1)^{\ell}\,\lambda^{\ell+m}\,
\delta_{p',q'} \delta_{(p'\leq c)}\, e_2\otimes e_1  \,,
\end{equation}
and it is zero on any other pair of generators. Indeed, this follows by applying \eqref{Eq:PtoPA-a} to the double Poisson bracket in \eqref{Eq:SemiS-1} before using sesquilinearity \eqref{Eq:DA1}. 
Note that \eqref{Eq:SemiS-2} is particularly simple because the double brackets in \eqref{Eq:SemiS-1} are valued in $B\otimes B$ so applying $\del$ gives zero. 

%\subsubsection{} %% HoP

The $H_0$-Poisson structure on $H_0(\kk Q_{p,q})$ is obtained using \eqref{Eq:HO-ind}. 
Remark that $(e_1)_\sharp,(e_2)_\sharp$ and the elements 
\begin{equation} \label{Eq:SemiS-3}
 (v_{p_1}w_{q_1}v_{p_2}w_{q_2}\cdots v_{p_t}w_{q_t})_\sharp \,, \quad 
 t\geq1,\,\, 1\leq p_j\leq p,\,\, 1\leq q_j\leq q\,,
\end{equation}
span $H_0(\kk Q_{p,q})$. 
The Lie bracket on $H_0(\kk Q_{p,q})$ can be explicitly computed in a way similar to the necklace Lie bracket of \cite{BLB,Gi}, see \cite[\S6.4]{VdB} (we fall within that case if $p=q=c$). 
As an example of computation, note that \eqref{Eq:HO-ind} and \eqref{Eq:SemiS-1}  yield 
\begin{equation} \label{Eq:SemiS-4}
 \begin{aligned}
[(v_{p_1}w_{q_1})_\sharp , (v_{p_2}w_{q_2})_\sharp]_\sharp 
&=\big(\mathrm{m}\circ (v_{p_1}\ast_1 \dgal{w_{q_1},v_{p_2}}w_{q_2} + v_{p_2}\dgal{v_{p_1},w_{q_2}} \ast_1 w_{q_1} ) \big)_\sharp \\
&= \delta_{p_1,q_2} \, \delta_{(p_1\leq c)} (v_{p_2}w_{q_1})_\sharp 
- \delta_{p_2,q_1} \, \delta_{(p_2\leq c)} (v_{p_1}w_{q_2})_\sharp\,, 
 \end{aligned}
\end{equation}
for any $1\leq p_1,p_2\leq p$ and $1\leq q_1,q_2\leq q$. 

%\subsubsection{} %% HOPV

The relative $H_0$-Poisson vertex structure on $\Vect_\infty(\kk Q_{p,q})\subset \Jinf (\kk Q_{p,q})$ can be easily derived from the previous paragraph.
The vector space $H_0(\Vect_\infty(\kk Q_{p,q}))$ is spanned by $(e_1)_\sharp,(e_2)_\sharp$ and the elements $\del^\ell \nu$, $\ell \geq 0$, where $\nu$ is as in  \eqref{Eq:SemiS-3}.
The Lie vertex bracket can be computed from the Lie bracket on $H_0(\kk Q_{p,q})$ using Corollary \ref{Cor:JDPA-HOPV} and the formula \eqref{Eq:HOvert-ind}. We get for example ($\ell_i\geq 0$, $1\leq p_i\leq p$, $1\leq q_i \leq q$ for $i=1,2$)
\begin{equation} \label{Eq:SemiS-5}
[\del^{\ell_1}(v_{p_1}w_{q_1})_\sharp  {}_\lambda  \del^{\ell_2}(v_{p_2}w_{q_2})_\sharp]_\sharp
=(-\lambda)^{\ell_1} (\lambda+\del)^{\ell_2}\,[(v_{p_1}w_{q_1})_\sharp , (v_{p_2}w_{q_2})_\sharp]_\sharp \,,
\end{equation}
where the Lie bracket on the right-hand side is given in \eqref{Eq:SemiS-4}.
Note from Proposition \ref{Pr:ComLeft} that we could also have read the Lie vertex bracket by applying the functor $\sharp_{\mathtt{V}}$ to $\kk Q_{p,q}^\infty$ with the double $\lambda$-bracket given in \eqref{Eq:SemiS-2}.

%\subsubsection{} %% rep
%\label{sub:quiver_LSS_example}

Hereafter we work with the dimension vector $\underline{n}=(n_1,n_2)$ fixed to $(n,1)$ for some $n\geq 1$.
A representation $\rho\in \Rep_B(\kk Q_{p,q},\underline{n})$ is such that
\begin{equation}  \label{Eq:SemiS-6}
 \rho(v_{p'})=\left(
\begin{array}{c|c}
 0_{n\times n}&V_{p'} \\
 \hline
0_{1 \times n} &0
\end{array}
 \right)\,, \quad
 \rho(w_{q'})=\left(
\begin{array}{c|c}
 0_{n\times n}&0_{n\times 1} \\
 \hline
W_{q'} &0
\end{array}
 \right)\,,
\end{equation}
with $V_{p'}\in \kk^n$ and $W_{q'}\in (\kk^n)^\ast$ for any $1\leq p'\leq p$, $1\leq q'\leq q$.
Denoting by $(V_{p'})_j$ and $(W_{q'})_j$, $1\leq j \leq n$, the elements of $(\kk Q_{p,q})_{\underline{n}}$ returning the entries of the nonzero blocks of $\rho(v_{p'})$ and $\rho(w_{q'})$, we obtain an obvious isomorphism
\begin{equation} \label{Eq:SemiS-iso}
 (\kk Q_{p,q})_{\underline{n}} \simeq \kk[(\kk^n)^{\oplus p} \oplus ((\kk^n)^\ast)^{\oplus q}]\,.
\end{equation}
The Poisson bracket is then uniquely determined by
\begin{equation} \label{Eq:SemiS-7} 
\br{(V_{p_1})_j,(W_{q_1})_k}=\delta_{p_1,q_1} \delta_{(p_1\leq c)}\,\delta_{kj}  \,, \quad 
\br{(V_{p_1})_j,(V_{p_2})_k}=0=\br{(W_{q_1})_j,(W_{q_2})_k}\,,  
\end{equation}
(for all possible indices)
after combining \eqref{Eq:SemiS-1} and Theorem \ref{Thm:RepdP}.

As we are in $\PA^{\Gl_{\underline{n}}}$, the Poisson bracket \eqref{Eq:SemiS-7} is $\Gl_{\underline{n}}$-invariant for the action by simultaneous conjugation of all the matrices in \eqref{Eq:SemiS-6}.
Since $(c\Id_n,c)\in \Gl_{\underline{n}}$ acts trivially for any $c \in \kk^\times$, we are only interested in the action of $\Gl_n$ given by
\begin{equation} \label{Eq:CpCq-act}
 g\cdot (V_{1},\ldots,V_{p},W_{1},\ldots,W_q) = (g V_{1},\ldots,g V_{p},W_{1}g^{-1},\ldots,W_qg^{-1})\,.
\end{equation}
Therefore \eqref{Eq:SemiS-7} defines a $\Gl_n$-invariant Poisson bracket on $Y:=(\kk^n)^{\oplus p} \oplus ((\kk^n)^\ast)^{\oplus q}$, i.e. $\kk[Y]$ is an object in $\PA^{\Gl_n}$.
Let us emphasize that for the coordinate ring $A=\kk[Y]$, 
\eqref{Eq:InvMorph} 
is an isomorphism (cf.~Example \ref{ex:LSS_quiver}).  
In particular, the last node of the cube that we are describing will be $\Jinf(\kk[Y]^{\Gl_n}) \simeq \Jinf(\kk[Y])^{\Jinf(\Gl_n)}$.

%\subsubsection{} %% rep PVA

%Continuing the previous paragraph, 
Next, we take jets and obtain
\begin{equation}
 \Jinf((\kk Q_{p,q})_{\underline{n}}) \simeq \Jinf(\kk[Y])\,, \quad 
 Y:= (\kk^n)^{\oplus p} \oplus ((\kk^n)^\ast)^{\oplus q}\,, 
\end{equation}
in $\Jinf \PA^{\Gl_n}$. In terms of the functions
\begin{equation*}
 (V_{p'}^{(\ell)})_j,\,\, (W_{q'}^{(\ell)})_j, \qquad 1\leq p'\leq p,\,\, 1\leq q'\leq q,\,\, \ell \geq 0,\,\, 1\leq j \leq n\,,
\end{equation*}
defined in the obvious way with differential acting as the shift $\ell \mapsto \ell+1$, the $\lambda$-bracket satisfies
\begin{equation} \label{Eq:SemiS-8}
\br{(V_{p'}^{(\ell)})_j {}_\lambda (W_{q'}^{(m)})_k}=(-1)^\ell \,\lambda^{\ell+m}\, \delta_{p',q'} \delta_{(p'\leq c)}\,\delta_{kj}  \,,
\end{equation}
and it is zero on the other pairs of generators.
This follows directly from \eqref{Eq:SemiS-7} and Lemma \ref{Lem:PAtoPVA}. Alternatively, one can consider the double $\lambda$-bracket from \eqref{Eq:SemiS-2} and go to the representation algebra $(\Jinf(\kk Q_{p,q}))_{\underline{n}}$ using Theorem \ref{Thm:RepdPV}.

%\subsubsection{} %% reduction

A non-constant invariant function on $Y$  under the action \eqref{Eq:CpCq-act} of $\Gl_n$
is a linear combination of functions of the form
\begin{equation*}
 \tr(V_{p_1}W_{q_1}\cdots V_{p_t}W_{q_t})=\tr(V_{p_1}W_{q_t})\, \prod_{2\leq t'\leq t} \tr(V_{p_{t'}}W_{q_{t'-1}})\,,
\end{equation*}
where $t\geq 1$ and $1\leq p_{t'}\leq p$, $1\leq q_{t'}\leq q$.
Hence the Poisson bracket on $\kk[Y]^{\Gl_n}$ obtained by Poisson reduction is determined by its value on the generators $\tr(V_{p'}W_{q'})$.
It can be easily written from the Lie bracket on $H_0(\kk Q_{p,q})$: taking together \eqref{Eq:relHoP} and \eqref{Eq:SemiS-4}, one finds
\begin{equation} \label{Eq:SemiS-9}
 \begin{aligned}
\br{ \tr(V_{p_1}W_{q_1}) , \tr(V_{p_2}W_{q_2}) }
=\,& \delta_{p_1,q_2} \, \delta_{(p_1\leq c)} \tr(V_{p_2}W_{q_1}) \\
&- \delta_{p_2,q_1} \, \delta_{(p_2\leq c)} \tr(V_{p_1}W_{q_2})\,,
 \end{aligned}
\end{equation}
where $1\leq p_1,p_2\leq p$, $1\leq q_1,q_2\leq q$.

%\subsubsection{} %% reduction PVA

Thanks to the isomorphism \eqref{Eq:SemiS-iso},
recall that the last node of the cube in $\PVA$ is $\Jinf(\kk[Y]^{\Gl_n}) \simeq \Jinf(\kk[Y])^{\Jinf(\Gl_n)}$, cf. Example \ref{ex:LSS_quiver}.
By taking jets in the previous paragraph, we see that we only need to know the $\lambda$-bracket
on the generators $\del^\ell\tr(V_{p'}W_{q'})$, where we identify $\tr(V_{p'}W_{q'})$ and $\del^0\tr(V_{p'}W_{q'})$.
Its value follows directly from Lemma \ref{Lem:PAtoPVA} and \eqref{Eq:SemiS-9}, which give
\begin{equation} \label{Eq:SemiS-10}
 \begin{aligned}
\br{ \del^{\ell_1} \tr(V_{p_1}W_{q_1})  {}_\lambda  \del^{\ell_2} \tr(V_{p_2}W_{q_2}) }
=\,&(-\lambda)^{\ell_1} (\lambda+\del)^{\ell_2}\, \delta_{p_1,q_2} \, \delta_{(p_1\leq c)} \tr(V_{p_2}W_{q_1}) \\
&-(-\lambda)^{\ell_1} (\lambda+\del)^{\ell_2} \delta_{p_2,q_1} \, \delta_{(p_2\leq c)} \tr(V_{p_1}W_{q_2})\,,
 \end{aligned}
\end{equation}
where $1\leq p_1,p_2\leq p$, $1\leq q_1,q_2\leq q$ and $\ell_1,\ell_2\geq 0$.
We can also obtain \eqref{Eq:SemiS-10} easily from \eqref{Eq:SemiS-5} by applying the functor $\tr_{\underline{n}}:\mathtt{V}_\infty\HoP \to \Jinf \PA_{\Gl_{\underline{n}};0}$ and considering the image in $\PVA$.

%%%%%%%%%%%%%%%%%% NEW SECTION %%%%%%%%%%%%%%%%%%%%%
%%%%%%%%%%%%%%%%%% NEW SECTION %%%%%%%%%%%%%%%%%%%%%
%%%%%%%%%%%%%%%%%% NEW SECTION %%%%%%%%%%%%%%%%%%%%%
%%%%%%%%%%%%%%%%%% NEW SECTION %%%%%%%%%%%%%%%%%%%%%

\section{Hamiltonian Reduction in the commutative setting}  
\label{Sec:HamRed}

\subsection{Hamiltonian reduction} \label{ss:HamRed}

%\subsubsection{Classical case of Poisson algebras} \label{ss:Red-PA}

We follow \cite[\S5.2,5.4]{LGPV} for the algebraic treatment of Hamiltonian reduction. 
Fix $(A,\br{-,-})$ a Poisson algebra. 
Let $I\subset A$ be an ideal of~$A$. 
Consider the projection $\pi\colon A\to A/I$ which is an algebra homomorphism. 
\begin{definition}
A subalgebra $A_{\red}\subset A/I$ is a \emph{Poisson reduction} of $A$ if 
it is endowed with a Poisson bracket $\br{-,-}_{A_{\red}}$ satisfying 
\begin{equation} \label{Eq:PRed-FG}
    \br{F,G}_{A_{\red}} = \pi (\br{\tilde{F},\tilde{G}})\,,
\end{equation}
for any $F,G\in A_{\red}$ and $\tilde{F},\tilde{G}\in A$ such that $\pi(\tilde{F})=F$, $\pi(\tilde{G})=G$. 
\end{definition}

Set $Y=\Spec(A)$ and let $\GG$ be an affine algebraic group with a left action $\GG\times Y \to Y$. Assume that the corresponding left action of $\GG$ on $A$  
is by Poisson automorphisms. We obtain an action by derivations on $A$ of the Lie algebra $\g$ of $\GG$, which is denoted\footnote{We simply wrote this action as $x.\tilde{F}$ in the proof of Theorem \ref{Th:invariants_are_PVA}. Hereafter, we keep track of the algebra on which $\g$ is acting to avoid confusion when takings jets in \S\ref{ss:Red-PVA-jet}.} $(x, \tilde{F})\mapsto x_A(\tilde{F})$ for $x\in \g$ and $\tilde{F}\in A$. 
Note that $x_V(A^\GG)=0$.  

A \emph{comoment map} is a morphism of Lie algebras $\tilde{\mu}\colon \g\to A$ which is equivariant (for the adjoint action on $\g$) and such that $x_A(\tilde{F})=\br{\tilde{\mu}(x),\tilde{F}}$ for any $x\in \g$ and $\tilde{F}\in A$.   
It will be convenient to see the comoment map as a morphism of Poisson algebras $\tilde{\mu}\colon \kk[\g^\ast]\to A$ under the identification $\kk[\g^\ast]\simeq Sym(\g)$, where $\g^\ast$ is the dual of $\g$. 

Denote by  $\langle-,-\rangle:\g^\ast\times \g\to \kk$ the natural pairing $\langle \xi,x\rangle=\xi(x)$.  
If we view $A$ as the finitely generated algebra of functions on the affine scheme $Y=\Spec(A)$, 
then $\tilde{\mu}$ corresponds to an equivariant map $\mu\colon Y\to \g^\ast$ through $\tilde{\mu}(x)(y)=\langle \mu(y),x \rangle$  referred to as the \emph{moment map}. 
If $\xi\in \g^*$ is invariant under the coadjoint action of $\GG$, 
we get that $\mu^{-1}(\xi)$ is $\GG$-invariant. 
Imposing the relation $\mu=\xi$ is equivalent to requiring $\langle \mu,x\rangle=\langle \xi,x\rangle$ for any $x\in \g$; hence the closed invariant subspace $\mu^{-1}(\xi)$ corresponds to the $\GG$-invariant ideal 
\begin{equation}     \label{Eq:I-Xinv} 
I_{\xi}:=(\tilde{\mu}(x)-\langle \xi,x\rangle \mid x \in \g\,)\,.
\end{equation}

The definition of a comoment map entails $\br{I_{\xi},A^\GG}\subset I_{\xi}$, hence the next result.

\begin{proposition} \label{Pr:HamRed-0}
 The subalgebra $\pi(A^\GG) \subset A/I_{\xi}$ is a Poisson reduction of $A$.     
\end{proposition}
A refinement of this proposition is obtained by considering the $\GG$ action induced on $A/I_{\xi}$ and noting that $\br{I_{\xi},a}\in I_{\xi}$ for any $a\in \pi^{-1}((A/I_{\xi})^\GG)$.  
\begin{proposition} \label{Pr:HamRed}
 The subalgebra of invariant functions $(A/I_{\xi})^\GG \subset A/I_{\xi}$ is a Poisson reduction of $A$, which is called the \emph{Hamiltonian reduction} of $A$ at $\xi$.  
Equivalently, $\Spec ((A/I_{\xi})^\GG)$ is an affine Poisson scheme called the \emph{Hamiltonian reduction} of $Y$ at $\xi$.
\end{proposition}
We can remark that the two propositions are equivalent if $\GG$ is linearly reductive as the map $A^\GG \to (A/I_{\xi})^\GG$ is surjective. 

%\subsubsection{Case of Poisson vertex algebras} \label{ss:Red-PVA}
\subsection{Hamiltonian reduction in $\PVA$} \label{ss:HamRed-PVA}

Fix a Poisson vertex algebra $V$ with differential $\del$ and $\lambda$-bracket $\br{-_\lambda-}$.  
Let $I$ be a differential ideal of $V$, i.e. an ideal such that $\partial(I)\subset I$. 
Consider the projection $\pi\colon V\to V/I$ which is a morphism of differential algebras. 
\begin{definition}  
A differential subalgebra $V_{\red}\subset V/I$ is a \emph{Poisson vertex reduction} of $V$ if it is a Poisson vertex algebra with $\lambda$-bracket $\br{-_\lambda-}_{V_\red}\colon V_\red\times V_\red\to V_\red[\lambda]$ satisfying  
\begin{equation} \label{Eq:PVRed-FG}
    \br{F_\lambda G}_{V_\red} = \pi (\br{\tilde{F}_\lambda\tilde{G}})\,,
\end{equation}
for any $F,G\in V_\red$ and $\tilde{F},\tilde{G}\in V$ such that $\pi(\tilde{F})=F$, $\pi(\tilde{G})=G$. 
\end{definition}
First, we seek a Poisson vertex analogue of Proposition \ref{Pr:HamRed-0}; 
this requires some preparation. 

\begin{lemma} \label{Lem:PVred1}
    Let $V,I,\pi$ be as above. 
A differential subalgebra $V_\red \subset V/I$ is a Poisson vertex reduction of $V$ if and only if the following two conditions hold:  
\begin{enumerate}
    \item $\pi^{-1}(V_\red)_{(n)} I\subset I$ for all $n\geq 0$; 
    \item $\pi^{-1}(V_\red)$ is a Lie conformal subalgebra of $V$. 
\end{enumerate}
In particular, the $\lambda$-bracket on $V_\red$ can be computed using \eqref{Eq:PVRed-FG}. 
\end{lemma}
\begin{proof}
It suffices to adapt Proposition 5.5 in \cite{LGPV} to the vertex setting using that $\pi$ is a morphism of differential algebras. 
We leave the details to the reader. 
\end{proof}

\begin{lemma} \label{Lem:PVred2}
Let $I$ be a differential ideal in $V$ with projection $\pi\colon V\to V/I$. 
Let $V^\star\subset V$ be a Poisson vertex subalgebra. 
Then the differential algebra $V_\red=\pi(V^\star)$ is a Poisson vertex reduction of $V$ if 
for any $n\geq0$, we have $V^\star{}_{(n)} I\subset I$ and $I_{(n)} I\subset I$. 
\end{lemma}
\begin{proof}
We can build on the equivalent characterisation of Poisson vertex reduction from Lemma \ref{Lem:PVred1}. 
Since $\pi^{-1}(V_\red)=V^\star + I$, the two conditions from the statement are easily seen to imply (1) from  Lemma \ref{Lem:PVred1}. Moreover, (2) from that lemma is equivalent to requiring that the differential and the $\lambda$-bracket restrict to $\pi^{-1}(V_\red)$; this is also a direct consequence of the stated conditions since $V^\star$ is a Poisson vertex subalgebra of $V$. 
\end{proof}

Let $\HH$ be an affine (pro)algebraic group with a left action on $Y=\Spec(V)$, 
such that $V^\HH$ is a Poisson vertex subalgebra of $V$. 
For example, this is the case when $\HH$ acts by Poisson vertex automorphisms 
%as in \S\ref{ss:NaiveInv}, 
or when $\HH=\Jinf(\GG)$ and $V=\Jinf(A)$ with $\GG$ acting by Poisson automorphisms on $A$ due to Theorem \ref{Th:invariants_are_PVA}. 
We obtain an action by derivations of the Lie algebra $\h$ of $\HH$, which we denote $(x, \tilde{F})\mapsto x_V(\tilde{F})$ for any $x\in \h$ and $\tilde{F}\in V$. In particular, $x_V(V^\HH)=0$.   
Note that by Lemma \ref{Lem:PVred2} with $V^\star=V^\HH$, we can construct a Poisson vertex reduction $V_\red=\pi(V^\HH)$ for any projection $\pi$ onto a quotient of $V$ by a differential ideal satisfying the assumptions of the lemma.  

\medskip 

We assume from now on that $\HH=\Jinf(\GG)$ where $\GG$ is an affine algebraic group. 
We view $\Jinf(\kk[\g^\ast])$ as a Poisson vertex algebra with $\lambda$-bracket obtained through 
Lemma~\ref{Lem:PAtoPVA}. 
The coadjoint action of $\GG$ on $\g^\ast$ naturally extends to an action of $\Jinf(\GG)$ on $\Jinf(\g^\ast)$. 

\begin{definition}\label{defcommappva}
A comoment map is a $\Jinf(\GG)$-equivariant morphism of Poisson vertex algebras 
$\tilde{\mu}\colon \Jinf(\kk[\g^\ast])\to V$ such that 
\begin{subequations}
    \begin{align}
&\tilde{\mu}(x)_{(0)}\tilde{F}=x_V(\tilde{F})\,, \qquad \forall x\in \g,\,\, \tilde{F}\in V , \label{Eq:MomapV-a}\\ 
&\tilde{\mu}(x)_{(n)}\tilde{F}=0\,, \qquad \qquad \forall x\in \g,\,\, \tilde{F}\in V^{\Jinf(\GG)},\,\, n\geq 1\,. \label{Eq:MomapV-b}
    \end{align}
\end{subequations} 
\end{definition}
Note that by \eqref{Eq:MomapV-a}, we can allow $n=0$ in \eqref{Eq:MomapV-b}. 
Furthermore, the definition entails 
\begin{equation} \label{Eq:MomapV-c}
\tilde{\mu}([x,x'])
=  \br{\tilde{\mu}(x)_\lambda \tilde{\mu}(x')}  
=  x_V(\tilde{\mu}(x')\!) 
\,,\quad \forall x,x'\in \g.
\end{equation}

Fix some $\xi\in \g^*$ invariant under the coadjoint action (hence under the induced action of $\Jinf(\GG)$). In analogy with the Poisson case, we consider the differential ideal 
\begin{equation} \label{Eq:IV-Xinv}
    I_{\xi}^{\partial}:=(\tilde{\mu}(x)-\langle \xi,x\rangle \mid x \in \g\,)\,.
\end{equation} 
Compatibility with the differential yields for any $k\geq 1$ and $x\in \g$ that  $\tilde{\mu}(\del^kx)=\del^k(\tilde{\mu}(x)\!)$ is an element in $I_{\xi}^{\partial}$. 
Furthermore, since $\xi$ is $\GG$-invariant we see that $I_{\xi}^{\partial}$ is stable under the action of $\Jinf(\GG)$,  hence it is stable under the action of $\Jinf(\g)$ by derivations. 
We let $\pi\colon V\to V/I_{\xi}^{\partial}$ be the corresponding quotient map of differential algebras. 
Although $I_{\xi}^{\partial}$ is \emph{not} a Poisson vertex ideal (we do not have $V_{(n)}I_{\xi}^{\partial}\subset I_{\xi}^{\partial}$ for all $n\geq0$) in general, the following statement holds.

\begin{proposition} \label{Pr:HamRedV-b}
Consider a Poisson vertex subalgebra $V^{\star}\subseteq V^{\Jinf(\GG)}$. 
Then its projection  $\pi(V^{\star}) \subset V/I_{\xi}^{\partial}$ is a Poisson vertex reduction of $V$. 
\end{proposition}
\begin{proof} 
We get by \eqref{Eq:MomapV-a}--\eqref{Eq:MomapV-b} that 
$\br{ \tilde{F}{}_{\lambda}\,  \tilde{\mu}(x)-\langle \xi,x\rangle }=0$ for any $x\in \g$, $\tilde{F}\in V^{\Jinf(\GG)}$. This yields the following inclusions for any $n\geq 0$ 
$$V^{\star}{}_{(n)} I_{\xi}^{\partial}\subset V^{\Jinf(\GG)}{}_{(n)} I_{\xi}^{\partial} \subset I_{\xi}^{\partial}\,.$$  
Since $I_{\xi}^{\partial}$ is stable under the action of $\Jinf(\g)$, we get from \eqref{Eq:MomapV-c} 
$$\br{\tilde{\mu}(x)_\lambda \tilde{\mu}(x')}  
=  x_V(\tilde{\mu}(x')-\langle \xi,x'\rangle ) \in I_{\xi}^{\partial}\,, \qquad 
\forall\, x,x'\in \g\,.$$ 
This yields the inclusion  $ I_{\xi}^{\partial}{}_{(n)} I_{\xi}^{\partial}\subset I_{\xi}^{\partial}$ for all $n\geq 0$. 
We conclude by Lemma \ref{Lem:PVred2}. 
\end{proof}

Taking $V^\star=V^{\Jinf(\GG)}$ yields the following result. 

\begin{corollary} \label{Cor:HamRedV-wk}
 The subalgebra of invariant functions $\pi(V^{\Jinf(\GG)}) \subset V/I_{\xi}^{\partial}$ is a Poisson vertex reduction of $V$, called the \emph{weak Hamiltonian reduction} of $V$ at $\xi$.     
\end{corollary}

\subsubsection{Case of jets} \label{ss:Red-PVA-jet}

Fix a Poisson algebra $A$ acted upon by an affine algebraic group $\GG$. 
To derive a Poisson vertex analogue of Proposition \ref{Pr:HamRed}, 
consider $V=\Jinf(A)$ equipped with the $\lambda$-bracket induced from $A$ by Lemma \ref{Lem:PAtoPVA} 
and the induced action of $\Jinf(\GG)$. 

\begin{lemma} \label{Lem:MomapJets}
 Assume that $\tilde{\mu}:\kk[\g^\ast]\to A$ is a comoment map.   
 Then its jet morphism $\tilde{\mu}_\infty\colon \Jinf\kk[\g^\ast]\to \Jinf(A)$ is a comoment map according to Definition \ref{defcommappva}. 
\end{lemma}

\begin{proof}
Under the functor $\mathtt{J}:\PA^\GG\to \Jinf\PA^\GG$, $\tilde{\mu}$ induces a $\Jinf(\GG)$-equivariant morphism of Poisson vertex algebras $\tilde{\mu}_\infty:\Jinf\kk[\g^\ast]\to \Jinf(A)$. 
Hence, we need to check \eqref{Eq:MomapV-a}--\eqref{Eq:MomapV-b}.  
For any $x\in \g$, $a\in A$ and $\ell\geq 0$, we compute
\begin{equation} \label{Eq:Lem-PVAred}
 \br{\tilde{\mu}_\infty(x)_\lambda \del^\ell a}
 =
(\del+\lambda)^\ell \br{\tilde{\mu}(x) , a}
=
(\del+\lambda)^\ell \,x_{A}(a)
=
x_{\Jinf(A)}((\del+\lambda)^\ell a)\,,
\end{equation}
where we used \eqref{eq:Id0} with $k=0$ in the last equality.
We easily deduce \eqref{Eq:MomapV-a}. 
Let $\tilde{F}=\del^{\ell_1}(a_1)\ldots \del^{\ell_r}(a_r)$  with $a_j\in A$, $\ell_j\geq 0$.  
We note by \eqref{eq:Id0} that the action of $x t^k\in \g\[[t\]]=\Jinf(\g)$ on $\tilde{F}$ is given by
\begin{equation} \label{Eq:Lem-PVAred-2}
 \begin{aligned} 
 x_{(k)}\tilde{F}  
 &=
 \sum_{s=1}^r \del^{\ell_1}(a_1)\ldots \delta_{(k\leq \ell_s)} \frac{\ell_s!}{(\ell_s-k)!} x_{\Jinf(A)}(\del^{\ell_s-k}(a_s))\ldots \del^{\ell_r}(a_r)\,.
\end{aligned}
\end{equation}
Together with \eqref{Eq:Lem-PVAred}, this implies
\begin{equation} \label{Eq:Lem-PVAred-3}
 \begin{aligned} 
 \br{\tilde{\mu}_\infty(x)_\lambda \tilde{F}} =
 \sum_{s=1}^r \del^{\ell_1}(a_1)\ldots x_{\Jinf(A)}((\del+\lambda)^{\ell_s} a_s) \ldots \del^{\ell_r}(a_r)  
= \sum_{k\geq 0} \frac{\lambda^k}{k!} x_{(k)}\tilde{F}\,.
\end{aligned}
\end{equation}
By linearity, \eqref{Eq:Lem-PVAred-3} holds for any $\tilde{F}\in \Jinf(A)$. 
Now, if $\tilde{F}$ is $\Jinf(\GG)$-invariant, $x_{(k)}\tilde{F}=0$ for any $k\geq0$ and we deduce that 
\eqref{Eq:MomapV-b} is satisfied. Thus $\tilde{\mu}_\infty$ is a comoment map.
\end{proof}

Let us now assume that $\Jinf(A)$ is equipped with the comoment map $\tilde{\mu}_\infty$ obtained from\footnote{While this condition seems restrictive, in Lemma \ref{Lem:JPAGmu}, we prove that any comoment map on $\Jinf(A)$ (seen as an element of $\Jinf\PA^\GG$) is of that form.} a comoment map on $A$ by Lemma \ref{Lem:MomapJets}.  
Define  the ideal $I_{\xi}^{\partial}$ as in \eqref{Eq:IV-Xinv}.

We are in position to state the following Poisson vertex analogue of Proposition \ref{Pr:HamRed}.
\begin{proposition} \label{Pr:HamRedV} 
The subalgebra of invariant functions 
$$V_{\red;\xi}:=(\Jinf(A)/I_{\xi}^{\partial})^{\Jinf(\GG)} \subset \Jinf(A)/I_{\xi}^{\partial}$$ 
is a Poisson vertex reduction of $\Jinf(A)$, called the \emph{Hamiltonian reduction} of $\Jinf(A)$ at $\xi$.   
\end{proposition}
\begin{proof}
The aim is to verify the two criteria from Lemma \ref{Lem:PVred1} for the projection $\pi:\Jinf(A)\to \Jinf(A)/I_{\xi}^{\partial}$. 
As in Theorem \ref{Th:invariants_are_PVA}, we start by assuming that $\GG$ is connected. 
An element of $\Jinf(A)/I_{\xi}^{\partial}$ is $\Jinf(\GG)$-invariant if and only if 
for any of its lifts $\tilde{F}\in \Jinf(A)$ one has $\Jinf(\GG)\cdot \tilde{F}\subset \tilde{F}+I_{\xi}^{\partial}$, or equivalently 
\begin{equation} \label{Eq:HamRed-pf1}
 x_{(k)}\tilde{F} \in I_{\xi}^{\partial}\,, \quad \text{ for any }k\geq 0,\,\, x \in \g\,.
\end{equation}
For the first criterion of Lemma \ref{Lem:PVred1}, we need to check $\tilde{F}_{(n)} I_{\xi}^{\partial} \subset I_{\xi}^{\partial}$ for any $n\geq 0$, which holds provided that 
$\del^\ell(\tilde{\mu}_\infty(x)-\langle \xi,x\rangle)_{(n)} \tilde{F} \in I_{\xi}^{\partial}$ for any $\ell \geq 0$, $x\in \g$.  
We compute  
\begin{align*}
&\del^\ell(\tilde{\mu}_\infty(x)-\langle \xi,x\rangle)_{(n)} \tilde{F} 
= \delta_{(\ell\leq n)} \frac{(-1)^\ell\, n!}{(n-\ell)!}  \, \tilde{\mu}_\infty(x)_{(n-\ell)} \tilde{F} 
= \delta_{(\ell\leq n)} \frac{(-1)^\ell\, n!}{(n-\ell)!}  \,  x_{(n-\ell)}\tilde{F} \,,
\end{align*} 
where the second equality follows from \eqref{Eq:Lem-PVAred-3} at order $\lambda^{n-\ell}$. 
This belongs to $I_{\xi}^{\partial}$ by \eqref{Eq:HamRed-pf1}. 

To check the second criterion of Lemma \ref{Lem:PVred1}, we take lifts $\tilde{F},\tilde{G}\in \Jinf(A)$ of two elements in $(\Jinf(A)/I_{\xi}^{\partial})^{\Jinf(\GG)}$.  
By Lemma \ref{Lem:identities}, (2) and \eqref{Eq:HamRed-pf1}, $x_{(k)}\del\tilde{F}\subset \del I_{\xi}^{\partial}\subset I_{\xi}^{\partial}$, hence $\pi^{-1}(V_{\red;\xi})$ is a differential subalgebra of $\Jinf(A)$. 
Meanwhile, we have by \eqref{eq:main_ind}  
\begin{align*}
x_{(k)} (\tilde{F}_{(n)} \tilde{G})
=\tilde{F}_{(n)} (\underbrace{x_{(k)}\tilde{G}}_{\in I_{\xi}^{\partial}})
+ \sum_{\ell=0}^k \binom{k}{\ell} 
(\underbrace{x_{(k-\ell)} \tilde{F}}_{\in I_{\xi}^{\partial}})_{(n+\ell)} \tilde{G} \,.
\end{align*}
Since $\tilde{F}_{(m)} I_{\xi}^{\partial}$ and $I_{\xi}^{\partial}{}_{(m)}\tilde{G}$ belong to $I_{\xi}^{\partial}$ for any $m\geq 0$ by the first part of the proof (and skewsymmetry \eqref{Eq:A2-b}), we find that $x_{(k)} (\tilde{F}_{(n)} \tilde{G})\in I_{\xi}^{\partial}$.  
Thus $\pi(\tilde{F}_{(n)} \tilde{G})\in V_{\red;\xi}$ and $\pi^{-1}(V_{\red;\xi})$ is a Lie conformal subalgebra of $\Jinf(A)$. 

If $\GG$ is not connected, we can derive the first criterion without any change. 
For the second criterion, we have to consider an extra finite group $\Gamma:=\GG/\GG^0$. 
As $\Gamma$ is linearly reductive and $I_{\xi}^{\partial}$ is a $\Gamma$-module, we can take the two lifts $\tilde{F},\tilde{G}\in \Jinf(A)$ considered above to be $\Gamma$-invariant. 
Then, as noticed at the end of the proof of Theorem \ref{Th:invariants_are_PVA}, 
$\del\tilde{F}$ and $\tilde{F}_{(n)} \tilde{G}$, for any $n\geq 0$, are $\Gamma$-invariant. We can conclude. 
\end{proof}

\begin{corollary} \label{Cor:HamRedV-emb}
 The embedding $\pi(\Jinf(A)^{\Jinf(\GG)})\hookrightarrow V_{\red;\xi}$ 
(where $\pi \colon \Jinf(A)\to \Jinf(A)/I_{\xi}^{\partial}$) is a morphism of Poisson vertex algebras 
for the structures obtained in Corollary \ref{Cor:HamRedV-wk} and Proposition \ref{Pr:HamRedV}, respectively.
\end{corollary}
\begin{proof}
This is clearly an embedding of differential algebras. 
Both $\lambda$-brackets can be computed through \eqref{Eq:PVRed-FG} by projecting the $\lambda$-bracket on $\Jinf(A)$, so the embedding clearly intertwines the $\lambda$-brackets. 
\end{proof}

\begin{remark}
The morphism from Corollary \ref{Cor:HamRedV-emb} is an isomorphism when $\GG$ is linearly reductive and the following two morphisms are isomorphisms: 
\begin{equation*}
 j_A\colon \Jinf(A^{\GG})\to \Jinf(A)^{\Jinf(\GG)}\,, \quad 
  j_{A/I_{\xi}} \colon \Jinf((A/I_{\xi})^{\GG})\to \Jinf(A/I_{\xi})^{\Jinf(\GG)}\,.
\end{equation*} 
\end{remark}

\begin{remark}
Our approach is an analogue for group actions of the reduction techniques for Lie vertex algebra actions on Poisson vertex algebras that appeared in the study of classical $\mathcal{W}$-algebras, see \cite{DS} for a review. 
\end{remark}

\subsection{Functorial interpretation} \label{ss:HamRed-Funct}

Fix an affine algebraic group $\GG$ with Lie algebra $\g$. 
We give a categorical flavour to the constructions presented above; this is standard in the Poisson case.  
%\subsubsection{Poisson algebras} 
Recall the category $\PA^\GG$ of Poisson $\GG$-algebras from \S\ref{ss:PoiRedNC}. 
We form the category $\PA^\GG_\mu$ of \emph{Hamiltonian Poisson $\GG$-algebras} as follows. 
Its objects are triplets $(A,\br{-,-},\tilde{\mu})$ where $(A,\br{-,-})$ is an object in $\PA^\GG$ 
and  $\tilde{\mu}\colon \kk[\g^\ast]\to V$ is a comoment map as in 
\S\ref{ss:HamRed}. 
Its morphisms are $\GG$-equivariant Poisson homomorphisms (i.e. morphisms in $\PA^\GG$) that respect the comoment maps. 
The additional condition on such a map 
\begin{equation} \label{Eq:Mor-PAG} 
    \phi\colon (A_1,\br{-,-}_1,\tilde{\mu}_1)\to (A_2,\br{-,-}_2,\tilde{\mu}_2)
\end{equation}
means that $\phi\circ \tilde{\mu}_1=\tilde{\mu}_2$.  

Given a $\GG$-invariant element $\xi\in \g^\ast$, define $I_{j;\xi}\subset A_j$ as in \eqref{Eq:I-Xinv}. 
As $\phi$ intertwines the comoment maps, we get $\phi(I_{1;\xi})\subset I_{2;\xi}$. 
Therefore there is a morphism of algebras 
\begin{equation} \label{Eq:Mor-PAG-res} 
 \phi_{\xi}\colon A_1/I_{1;\xi}\to A_2/I_{2;\xi}\,,
\end{equation}
which satisfies $\pi_2\circ \phi=\phi_{\xi} \circ \pi_1$ for the projections $\pi_j\colon A_j\to A_j/I_{j;\xi}$. 
This morphism is $\GG$-equivariant since this property holds for $\phi$ and the ideals $I_{j;\xi}$ are $\GG$-invariant. 
Furthermore, by \eqref{Eq:PRed-FG}, $\phi_{\xi}$ restricts to a Poisson homomorphism $\phi_{\xi}\colon A_{1;\red}\to A_{2;\red}$, where we set $A_{j;\red}:= (A_j/I_{j;\xi})^\GG$ for the Hamiltonian reduction of $A_j$ at $\xi$ 
(see Proposition~\ref{Pr:HamRed}). 
Hence, we have a functor of Hamiltonian reduction $\mathtt{R}_{\xi}\colon  \PA_\mu^\GG\to \PA$. 

We can refine the functor $\mathtt{R}_{\xi}$ to take value in the category $\PA_{\GG;0}$ introduced in \S\ref{ss:RedPVA}. 
Using notation as above, the functor is defined on objects by sending $A$ to the pair $(A/I_{\xi},A_{\red})$, 
where $A_{\red}:=(A/I_{\xi})^\GG$  is seen as a Poisson algebra. 
The morphism $\phi$ from \eqref{Eq:Mor-PAG} is simply sent to $\phi_{\xi}$ \eqref{Eq:Mor-PAG-res}, which restricts to a Poisson homomorphism $A_{1;\red}\to A_{2;\red}$ as already noticed. 
This yields the desired functor $\mathtt{R}_{\xi}\colon  \PA_\mu^\GG\to \PA_{\GG;0}$.

\medskip 

%\subsubsection{Poisson vertex algebras} 
Based on Corollary \ref{Cor:HamRedV-wk}, a Hamiltonian reduction functor can be constructed similarly to the Poisson case  
if we consider Poisson vertex algebras equipped with a comoment map according to Definition~\ref{defcommappva} for some $\Jinf(\GG)$-action, provided that the subalgebras of $\Jinf(\GG)$-invariant elements are Poisson vertex subalgebras. 
For our study, we shall be interested in the more specific case of jets of Hamiltonian Poisson algebras where Proposition \ref{Pr:HamRedV} holds. 

%\medskip 

Recall from \S\ref{ss:JetInv} that $\Jinf\PA^\GG:=\mathtt{J}(\PA^\GG)$ is a subcategory of $\PVA$ under the jet functor $\mathtt{J}:\PA\to \PVA$ of Lemma \ref{Lem:PAtoPVA}. 
On the one hand, we can define $(\Jinf\PA^\GG)_\mu$ from the category $\Jinf\PA^\GG$ by endowing objects with a comoment map while a morphism 
\begin{equation} \label{Eq:Mor-jetPAG} 
    \phi\colon (\Jinf(A_1),\br{-_\lambda-}_1,\tilde{\mu}_1)\to (\Jinf(A_2),\br{-_\lambda-}_2,\tilde{\mu}_2)
\end{equation}
must also satisfy  $\phi\circ \tilde{\mu}_1=\tilde{\mu}_2$.  
On the other hand, we define $\mathtt{J}(\PA^\GG_\mu)$ from $\Jinf\PA^\GG$ by endowing each object $\Jinf(A)$ with the jet $\tilde{\mu}_\infty$ of a comoment map $\tilde{\mu}$ for $A$; morphisms in $\mathtt{J}(\PA^\GG_\mu)$ are jets of morphisms in $\PA^\GG_\mu$. 

\begin{lemma} \label{Lem:JPAGmu} 
The following statements hold: 
\begin{enumerate}
\item $\mathtt{J}(\PA^\GG_\mu)$ is a subcategory of $(\Jinf\PA^\GG)_\mu$.
 \item The functor $\mathtt{Q}:\PVA\to \PA$ restricts to $\mathtt{Q}:(\Jinf\PA^\GG)_\mu \to \PA^\GG_\mu$.
\item There is an equivalence of categories 
$\mathtt{J}\circ \mathtt{Q}:(\Jinf\PA^\GG)_\mu \stackrel{\sim}{\longrightarrow} \mathtt{J}(\PA^\GG_\mu)$.
\end{enumerate}
\end{lemma}
\begin{proof}
(1) Let $(A,\br{-,-},\tilde{\mu})$ be an object in $\PA^\GG_\mu$.
We get the pair $(\Jinf(A),\br{-_\lambda-})$ and  the $\Jinf(\GG)$-equivariant morphism of Poisson vertex algebras $\tilde{\mu}_\infty:\Jinf\kk[\g^\ast]\to \Jinf(A)$  under the functor $\mathtt{J}\colon\PA^\GG\to \Jinf\PA^\GG$; this is a comoment map by Lemma \ref{Lem:MomapJets}.   
Thus a morphism in $\mathtt{J}(\PA^\GG_\mu)$ respects the induced comoment maps, hence it defines a morphism in $(\Jinf\PA^\GG)_\mu$.

(2) Take $\Jinf(A)\in (\Jinf\PA^\GG)_\mu$ with comoment map $\tilde{\nu}$.
Applying the functor $\mathtt{Q}$, we obtain $\Jinf(A)/\langle \del \Jinf(A)\rangle\simeq A$ and by Definition \ref{defcommappva} the comoment map is sent to a morphism $\tilde{\mu}:\kk[\g^\ast]\to A$ which is $\GG$-equivariant and such that for any $x\in \g$ and $\tilde{F}\in A\hookrightarrow \Jinf(A)$,
\begin{align*}
 \br{\tilde{\mu}(x),\tilde{F}}+ \langle \del \Jinf(A)\rangle
 = \tilde{\nu}(x)_{(0)}\tilde{F}\, + \langle \del \Jinf(A)\rangle   
 =x_{A}\tilde{F}\, + \langle \del \Jinf(A)\rangle\,.
\end{align*}
Under the above identification, this means that $\tilde{\mu}$ defines a moment map on $A$; by universality, $\tilde{\mu}_\infty=\tilde{\nu}$ on jets.
A morphism $\phi_\infty:\Jinf(A_1)\to\Jinf(A_2)$ in $(\Jinf\PA^\GG)_\mu$ sends a ($\Jinf(\GG)$-invariant) comoment map to another, so after applying $\mathtt{Q}$ we get that $\phi:A_1\to A_2$ sends the corresponding ($\GG$-invariant) comoment map to the other one.

(3) 
This functor is essentially surjective since we have shown that comoment maps in $\Jinf\PA^\GG$ are jets of comoment maps in $\PA^\GG$. It is also fully faithful as morphisms in both categories are jets of morphisms in $\PA^\GG_\mu$. 
\end{proof}

We shall denote $(\Jinf\PA^\GG)_\mu \simeq \mathtt{J}(\PA^\GG_\mu)$ simply as $\Jinf \PA^\GG_\mu$, which we call the category of \emph{Hamiltonian jets of Poisson $\GG$-algebras}. 

\medskip

Fix a $\GG$-invariant element $\xi\in \g^\ast$.
Given $\Jinf(A_j)\in \Jinf \PA^\GG_\mu$, $j=1,2$, 
consider the differential ideals $I_{j;\xi}^{\partial}\subset \Jinf(A_j)$ defined through \eqref{Eq:I-Xinv} 
and the corresponding projections $\pi_j:\Jinf(A_j)\to \Jinf(A_j)/I_{j;\xi}^{\partial}$.
A morphism $\phi_\infty:\Jinf(A_1)\to \Jinf(A_2)$ in $\Jinf \PA^\GG_\mu$ intertwines the comoment maps and is $\Jinf(\GG)$-equivariant, therefore we get a morphism of differential algebras 
\begin{equation} \label{Eq:Mor-PVAG-res}
 (\phi_\infty)_{\xi}\colon V_{1;\red} \to V_{2;\red}\,, \quad \text{for }  V_{j;\red}:=\big(\pi_j(\Jinf(A_j))\big)^{\Jinf(\GG)}\,, 
\end{equation}
which is a morphism of Poisson vertex algebras for the $\lambda$-brackets defined through 
Proposition~\ref{Pr:HamRedV}.
Assigning $\Jinf(A_j)\to V_{j;\red}$ and $\phi_\infty \to (\phi_\infty)_{\xi}$ defines the functor of Hamiltonian reduction
$\mathtt{R}_{\xi;\infty}\colon \Jinf \PA^\GG_\mu \to \PVA$.

%\subsection{Hamiltonian reduction commutes with taking jets}

Let $\GG$ be an affine algebraic group with Lie algebra $\g$. 
Fix a $\GG$-invariant element $\xi\in \g^\ast$. 

\begin{proposition} \label{Pr:HamRed_commute}
 The following diagram is commutative
\begin{center}
    \begin{tikzpicture}
 \node  (TopLeft) at (-2,1) {$\PA^\GG_\mu$};
\node  (TopRight) at (2,1) {$ \Jinf\PA^\GG_\mu$};
\node  (BotLeft) at (-2,-1) {$\PA_{\GG;0}$};
\node (BotMid) at (0,-1) {$\Jinf \PA_{\GG;0}$};
\node  (BotRight) at (2,-1) {$\PVA$}; 
\path[->,>=angle 90,font=\small]  
   (TopLeft) edge node[above] {$\mathtt{J}$} (TopRight) ;
\path[->,>=angle 90,font=\small]  
    (BotLeft) edge node[above] {$\mathtt{J}$} (BotMid) ;
\path[->,>=angle 90,font=\small]  
    (BotMid) edge node[above] {$j_{(-)}$} (BotRight) ;
\path[->,>=angle 90,font=\small]  
    (TopLeft) edge node[left] {$\mathtt{R}_{\xi}$}  (BotLeft) ;
\path[->,>=angle 90,font=\small]  
    (TopRight) edge node[right] {$\mathtt{R}_{\xi;\infty}$} (BotRight) ;  
    %%% right diagram  
 \node  (TopLeft2) at (4,1) {$A$};
\node  (TopRight2) at (9,1) {$\Jinf(A)$};
\node  (BotLeft2) at (4,-1) {$A_{\red}$};
\node (BotMid2) at (6.5,-1) {$\Jinf(A_{\red})$}; 
\node  (BotRight2) at (9,-1) {$V_\red$};  
\path[->,>=angle 90,font=\small]  
   (TopLeft2) edge (TopRight2) ;
\path[->,>=angle 90,font=\small]  
    (BotLeft2) edge (BotMid2) ;
\path[->,>=angle 90,font=\small]  
    (BotMid2) edge (BotRight2) ;
\path[->,>=angle 90,font=\small]  
    (TopLeft2) edge (BotLeft2) ;
\path[->,>=angle 90,font=\small]  
    (TopRight2) edge (BotRight2) ;     
   \end{tikzpicture}
\end{center}
where $A_\red$ and $V_\red$ denote the Hamiltonian reductions of $A$ and $\Jinf(A)$ at $\xi$ with respect to $\GG$ and $\Jinf(\GG)$, respectively. 
\end{proposition}
\begin{proof}
We start with some observations. 
Fix $A\in \PA^\GG_\mu$. The projection $\Jinf(A)\to \Jinf(A)/I_{\xi}^\partial$ used to define $V_\red$ is the jet of $\pi:A\to A/I_{\xi}$. Indeed, $I_{\xi}^\partial$ is the differential ideal generated by the image of $I_{\xi}$ under the inclusion $A\hookrightarrow \Jinf(A)$. 
In particular, we can write $\Jinf(\pi(A))=\pi_\infty(\Jinf(A))$ hence $V_\red=\Jinf(\pi(A))^{\Jinf(\GG)}$. 
Recalling that $A_\red= \pi(A)^\GG$,  
the second map appearing at the bottom part of the diagram is 
$$j_{\pi(A)} : \Jinf(\pi(A)^\GG)\longrightarrow \Jinf(\pi(A))^{\Jinf(\GG)}\,,$$
see \eqref{Eq:InvMorph}.  
Thus we end up with the differential algebra $V_\red$ from both parts of the diagram. 
Fix $\tilde{F},\tilde{G}\in A^\GG$ and $k,\ell\geq 0$. 
The $\lambda$-bracket obtained by applying $\mathtt{R}_{\xi;\infty}\circ \mathtt{J}$ to $A$ is such that 
\begin{align*}
 \br{\pi_\infty(\del^k\tilde{F}){}_\lambda \pi_\infty(\del^\ell \tilde{G})}_{V_\red}
 =\pi_\infty (\br{\del^k\tilde{F}{}_\lambda\del^\ell \tilde{G}})  
 =(-\lambda)^k (\lambda+\del)^\ell\,\pi( \br{\tilde{F},\tilde{G}})
\end{align*}
using \eqref{Eq:PVRed-FG} then Lemma \ref{Lem:PAtoPVA}. 
Similarly by applying $\mathtt{J}\circ \mathtt{R}_{\xi}$, 
we find in $\Jinf(\pi(A)^\GG)$ 
\begin{align*}
 \br{\pi_\infty(\del^k\tilde{F}){}_\lambda \pi_\infty(\del^\ell \tilde{G})}
=(-\lambda)^k (\lambda+\del)^\ell \br{\pi(\tilde{F}),\pi(\tilde{G})}_{A_\red} 
=(-\lambda)^k (\lambda+\del)^\ell\,  \pi(\br{\tilde{F}, \tilde{G}}) 
\end{align*}
after using Lemma \ref{Lem:PAtoPVA} then \eqref{Eq:PRed-FG}. 
This is sent to the previous expression under the morphism $j_{\pi(A)}$ of Poisson vertex algebras. 

Given a morphism $\phi:A_1\to A_2$, the composite $\mathtt{R}_{\xi;\infty}\circ \mathtt{J}$ yields 
$(\phi_\infty)_{\xi}$ from \eqref{Eq:Mor-PVAG-res}, which is defined from the jet $\phi_\infty:\Jinf(A_1)\to \Jinf(A_2)$ modulo $I_{2;\xi}^\partial$. 
The composite $\mathtt{J}\circ \mathtt{R}_{\xi}$ yields the jet 
$$(\phi_{\xi})_\infty:\Jinf(A_1/I_{1;\xi})\to \Jinf(A_2/I_{2;\xi})$$
of $\phi_{\xi}$  from \eqref{Eq:Mor-PAG-res}. 
Since $I_{j;\xi}^\partial$ is generated as a differential ideal by $I_{j;\xi}$, $j=1,2$,  
this map is also computed by taking a lift to $\Jinf(A_1)$ before applying $\phi_\infty$ modulo $I_{2;\xi}^\partial$.
\end{proof}

%%%%%%%%%%%%%%%%%% NEW SECTION %%%%%%%%%%%%%%%%%%%%%
%%%%%%%%%%%%%%%%%% NEW SECTION %%%%%%%%%%%%%%%%%%%%%
%%%%%%%%%%%%%%%%%% NEW SECTION %%%%%%%%%%%%%%%%%%%%%
%%%%%%%%%%%%%%%%%% NEW SECTION %%%%%%%%%%%%%%%%%%%%%

\section{Moment maps and noncommutative Hamiltonian reduction} \label{Sec:Momap}

In this section we work over a semisimple algebra $B=\oplus_{s\in S} \kk e_s$, with $S$ finite, and drop the index $B$ in the notations introduced in \S\ref{subsecgenstat}.

\subsection{Hamiltonian reduction for DPAs}
\label{ss:HamRedDPA}

We review the notion of moment map in the noncommutative relative setting.

\begin{definition} \label{def:NCmomentmap}  (\cite{VdB})
A moment map for a double Poisson algebra $\AA$ over $B$ is an element $\bbmu=(\bbmu_s)_s\in\oplus_{s\in S}e_s\AA e_s$ such that for all $a\in\AA$,
\begin{equation} \label{Eq:NCmomentmap}
\dgal{\bbmu_s,a}=ae_s\otimes e_s-e_s\otimes e_sa\,.
\end{equation}
\end{definition}

\begin{example}
Over $B=\kk$, the double Poisson algebra $A=\kk[a]$ from Example \ref{Ex:ku} with $\alpha=1$, $\beta=\gamma=0$, admits $\bbmu=a$ as a moment map. 
Similarly, we have the moment maps $\bbmu=ba-ab$ for Example \ref{Exmp:Symp} and $\bbmu=a-bab^{-1}$ for Example \ref{Exmp:SympGL}. 
\end{example}

As mentioned in~\cite[Proposition 2.6.5]{VdB}, a moment map induces an $H_0$-Poisson structure on $\AA/(\bbmu-\zeta)$ for any $\zeta\in B$. Working functorially, we get the following, where $\DPA_\mu$ denotes the category of double Poisson algebras with a (noncommutative) moment map, with morphisms required to preserve moment maps. The result is proved in~\cite[Proposition 5.3]{F22}.

\begin{proposition}\label{propdpahamred}
The above induces a functor $\sharp^\zeta:\DPA_\mu\to\HoP$.
\end{proposition}

%\subsection{Representation theory} \label{ss:RepTh-rel}

Consider a dimension vector $\underline{n}=(n_s)\in \Z_{\geq0}^S$, and set $N=\sum_{s\in S}n_s$. 
Fix an ordering $S\simeq\{1,\ldots,|S|\}$. 
It yields a decomposition $R^N=\oplus_{s\in S} R^{n_s}$ for any commutative $k$-algebra $R$, and then an embedding $B\to\Mat_N(R)$, mapping $e_s$ on the projection matrix $P_s$ with respect to the direct summand $R^{n_s}$.
The functor\begin{equation*}
R\mapsto \Hom_B(\AA,\Mat_N(R))\end{equation*}
is represented by an affine scheme that we denote by $\Rep(\AA,\underline n)$. Its coordinate ring $\AA_{\underline{n}}$ is the quotient algebra of $\AA_N$ by the relations $(e_s)_{ij}=(P_s)_{ij}$.

In full generalities, Theorem~\ref{Thm:RepdP} holds on $\Rep(\AA,\underline n)$ when $\AA\in {}_B\DPA$ \cite{VdB}.  
Since the group $\Gl_{\underline n}:=\prod_s\Gl_{n_s}$ acts by Poisson automorphisms on the corresponding Poisson algebra $\AA_{\underline{n}}$, we get the following.

\begin{proposition} 
We have a functor $(-)_{\underline n}:\DPA\to\PA^{\Gl_{\underline n}}$.
\end{proposition}

As in \S\ref{sss:RepAlg}, denote by $X(a)$ the $\Mat_N(\kk)$-valued function on $\Rep(\AA,\underline n)$ given by $X(a)_{ij}=a_{ij}$, for $a\in\AA$. The derivative action of $\oplus_s\Mat_{n_s}(\kk)=\mathrm{Lie}(\Gl_{\underline n})=:\mathfrak g_{\underline n}$ on $\Rep(\AA,\underline n)$ is given by $\xi.X(a)=[X(a),\xi]$ 
for all $\xi\in\mathfrak g_{\underline n}$ and $a\in\AA$, cf. \eqref{Eq:ActInfRep}. 
Assume that $\AA$ comes with a moment map $\bbmu$. 
Then~\cite[Proposition 7.11.1]{VdB} states that \begin{equation*}
X(\bbmu)=\sum_sX(\bbmu_s):\Rep(\AA,\underline n)\to\mathfrak g_{\underline n}\end{equation*}
is a moment map for the Poisson structure induced by Theorem~\ref{Thm:RepdP}, meaning that for all $a\in\AA_{\underline n}$ and $\xi\in\mathfrak g_{\underline n}$, we have $\{\tr(\xi X(\bbmu)),a\}=\xi.a$. 
Denoting by $\PA_\mu^{\Gl_{\underline n}}$ the category of Poisson algebras with a $\Gl_{\underline n}$-Poisson action and a moment map, we get the following.

\begin{proposition}\label{propdpamurep}
We have a functor $(-)_{\underline n}:\DPA_\mu\to\PA_{\mu}^{\Gl_{\underline n}}$ given on objects by 
$(\AA,\bbmu)\mapsto (\AA_{\underline n},X(\bbmu))$. 
\end{proposition}

Fix $\zeta=\sum_{s\in S} \zeta_{s} e_s \in B$ and consider the functor of commutative Hamiltonian reduction $\mathtt{R}^\zeta:=\mathtt{R}_{X(\zeta)}$ from \S\ref{ss:HamRed-Funct} given by
\begin{equation*}
\begin{aligned}
&\PA_\mu^{\Gl_{\underline n}}\to \PA_{\Gl_{\underline n};0}\,,\quad
(\AA_{\underline n},X(\bbmu))\mapsto (\mathbb{P}_{\underline n}, \mathbb{P}_{\underline n}^{\Gl_{\underline n}}) \,, 
\end{aligned}
\end{equation*}
where 
$\mathbb{P}_{\underline n}:=(\AA/(\bbmu-\zeta))_{\underline n}
=\AA_{\underline n}/(\bbmu_{s,ij}-\zeta_s\, e_{s,ij}\mid s\in S,\,1\leq i,j\leq N).$
Along with Propositions \ref{propdpamurep} and~\ref{propdpahamred} it yields the following.

\begin{corollary}\label{corohamfront}
The diagram\begin{center}
    \begin{tikzpicture}
 \node  (TopLeft) at (-1,1) {$\DPA_\mu$};
 \node  (TopRight) at (1,1) {$\PA^{\Gl_{\underline n}}_\mu$};
 \node  (BotLeft) at (-1,-1) {$\HoP$};
 \node  (BotRight) at (1,-1) {$\PA_{\Gl_{\underline{n}};0}$};
\path[->,>=angle 90,font=\small]  
   (TopLeft) edge node[above]{$(-)_{\underline n}$} (TopRight) ;
   \path[->,>=angle 90,font=\small]  
   (BotLeft) edge node[below]{$\mathtt{tr}_{\underline n}$} (BotRight) ;
\path[->,>=angle 90,font=\small]  
   (TopLeft) edge node[left]{$\sharp^\zeta$}  (BotLeft) ;
\path[->,>=angle 90,font=\small]  
   (TopRight) edge node[right]{$\mathtt{R}^\zeta$} (BotRight) ;
   \end{tikzpicture}
\end{center} 
commutes, where $\mathtt{tr}_{\underline n}$ is the $B$-linear analog of $\mathtt{tr}_{N}$.
\end{corollary}

\subsection{Quiver representations}  \label{ss:QuivRep}

Consider a finite quiver $Q$ with vertex set $S$ as in \S\ref{ss:double_quiver}.
We see  $\kk \bar{Q}$ as a $B$-algebra for $B=\oplus_{s\in S} \kk e_s$.
The $B$-linear double Poisson bracket on $\kk \bar{Q}$ satisfying \eqref{Eq:dbr-quiver} admits $\bbmu=\sum_{a\in \bar{Q}}\epsilon(a) aa^\ast$ as a moment map \cite{VdB}.

The Poisson bracket on $(\kk \bar{Q})_{\underline{n}}$ induced by Theorem \ref{Thm:RepdP} is uniquely determined by
\begin{equation} \label{Eq:Q-Rep}
   \br{a_{ij},b_{kl}}= \left\{
\begin{array}{ll}
\epsilon(a) \, (e_{h(a)})_{kj} \, (e_{t(a)})_{il}  &\text{ if } b=a^\ast  \\
0     & \text{ else}
\end{array}
   \right. \qquad \text{ for }a,b\in \bar{Q}\,;
\end{equation}
it is obviously vanishing on the constant generators $(e_s)_{ij}$. The corresponding moment map induced by $\bbmu$ can be decomposed in terms of the $|S|$ functions
\begin{equation}
 X(\bbmu_s)=\sum_{\substack{a\in \bar{Q}\\ t(a)=s}}\epsilon(a) aa^\ast :
 \Rep(\kk \bar{Q},\underline n)\to\mathfrak g_{n_s}\,, \quad s\in S\,.
\end{equation}

Performing Hamiltonian reduction at $\zeta\in B$ yields the Poisson algebra $\mathbb{P}_{\underline n}^{\Gl_{\underline n}}$, where
\begin{equation} \label{Eq:Pn-red}
\mathbb{P}_{\underline n}= (\kk \overline{Q})_{\underline n}/\Big(\sum_{a \mid  t(a)=s}\epsilon(a)\,  (aa^\ast)_{ij} -\zeta_s \, e_{s,ij}\mid s\in S,\,1\leq i,j\leq N \Big)
\end{equation}
is the representation algebra of the deformed preprojective algebra $\mathbb{P}:=\Pi^\zeta=\kk \overline{Q}/(\bbmu-\zeta)$.
Thus $\mathbb{P}_{\underline n}^{\Gl_{\underline n}}$ is the coordinate ring of an affine quiver variety equipped with its standard Poisson structure \cite{VdB}.

\begin{example}
Consider the double quiver $\bar{Q}:=Q_{p,p}$ of \S\ref{ss:SemiS-example} with $p=q=c$.
The moment map reads $\bbmu=\sum_{p'=1}^p [v_{p'},w_{p'}]$.
For $n\geq 1$, we have the following decomposition of the moment map $X(\bbmu)$ on the representation space of dimension $(n,1)$:
\begin{equation*}
 \sum_{p'=1}^p V_{p'}W_{p'}:\Rep(\kk \bar{Q},(n,1))\to \gl_n\,, \quad
 -\sum_{p'=1}^p W_{p'}V_{p'}:\Rep(\kk \bar{Q},(n,1))\to \kk^\times\,,
\end{equation*}
where we use the notation introduced in \eqref{Eq:SemiS-6}.
For $\zeta=\zeta_1e_1+\zeta_2e_2\in B$,
we obtain $\mathbb{P}_{\underline n}$ \eqref{Eq:Pn-red} by fixing the value of $X(\bbmu)$ to $(\zeta_1 \Id_n,\zeta_2)$;
the algebra $\mathbb{P}_{\underline n}$  is nontrivial provided that $n\zeta_1+\zeta_2= 0$ and $n\leq p$.
\end{example}

\subsection{Moment maps for DPVAs}

%\subsubsection{General results}
Consider a a double Poisson vertex $B$-algebra $\VV$. 
We set $\mathbb{K}=\ker(\mathrm{m}\circ (-)^\sigma \colon  \VV\otimes \VV\to \VV)$; 
this is a $\sigma$-twisted version of the noncommutative $1$-forms $\Omega^1(\VV)=\ker(\mathrm{m} \colon  \VV\otimes \VV\to \VV)$ of Cuntz-Quillen \cite{CQ}.

\begin{definition}\label{defmmapdpva} 
A moment map for $\VV$ is an element $\bbmu=(\bbmu_s)_s\in\oplus_se_s\VV e_s$ such that \begin{equation*}
\dgal{{\bbmu_s}_\lambda a}=ae_s\otimes e_s-e_s\otimes e_sa~\mod \lambda\mathbb{K}[\lambda]\end{equation*}
for all $a\in\VV$. This defines a category $\mathtt{DPVA}_\mu$, requiring morphisms to preserve moment maps.
\end{definition}

\begin{remark}\label{rem:DPVAmu-HOPV} Consider $\VV\in\mathtt{DPVA}_\mu$ and $\zeta\in B$. Set $\mathbb P=\VV/(\bbmu-\zeta)$. 
Note that Definition~\ref{defmmapdpva} implies that for every $a\in\VV$ we have $\dgal{a_\lambda\bbmu-\zeta}\in\ker \mathrm m$.
Therefore, we have a functor $\mathtt{DPVA}_\mu\to\HoPV$.
\end{remark}

Now, we study what happens to a moment map on representation algebras. 
Consider any $\VV\in\mathtt{DPVA}_\mu$ of finite type as a differential algebra, and a dimension vector $\underline n$.  
First, note that $X(\bbmu)$ is $\Gl_{\underline n}$-equivariant by definition of $X(-)$: we have $X(a)(g.x)=(g.x)(a)=g^{-1}x(a)g$ by factoring the action of $g\in \Gl_{\underline n}\subset \Gl_N$ through \eqref{Eq:ActRep}. 
Next, choose $i,j$ such that $(e_s)_{ii}=(e_s)_{jj}=1$, and $a\in\VV$. Then for all $1\leq u,v\leq N$ we have
\begin{align*}
\{{\bbmu_{ij}}_\lambda a_{uv}\}&=\sum_{n\ge0}\dgal{{\bbmu_s}_{(n)}a}'_{uj}\dgal{{\bbmu_s}_{(n)}a}''_{iv}\lambda^n 
= a_{uj} \delta_{iv} - \delta_{uj} a_{iv} + \sum_{n \geq 1} K'_{uj} K''_{iv} \lambda^n\,,
\end{align*}
for some $K\in \mathbb{K}$, thanks to Definition~\ref{defmmapdpva}. 
Looking at order $\lambda^0$, condition \eqref{Eq:MomapV-a} of Definition~\ref{defcommappva} is obtained as in~\cite[Proposition 7.11.1]{VdB}.  
We also note that $\{{\bbmu_{ij}}_\lambda \tr(a)\}=0$ since $(K''K')_{ij}=0$. 
Thus $\bbmu_{ij}{}_{(m)}\VV_{\underline n}^{\tr}=0$ for any $m\geq 0$. 
We claim that \eqref{Eq:MomapV-b} holds for any $\tilde{F}\in \VV_{\underline{n}}^{\Gl_{\underline n}}$. This follows from our previous computation if we can show that $\VV_{\underline n}^{\tr}= \VV_{\underline n}^{\Gl_{\underline n}}$.

Consider a set of generators $a_1,\dots,a_r$ of the differential algebra $V=\VV_{\underline n}$, and then a $\Gl_{\underline n}$-invariant element $\tilde{F}\in V$. Then $\tilde{F}$ belongs to the finitely generated (genuine) algebra $V_{(k)}=\langle\partial^\ell a_j\mid0\le\ell\le k, 1\le j\le r\rangle$ for some $k$ and, using~\cite[Remark 2.3]{CB}, we have that $\tilde{F}$ belongs to $V_{(k)} \cap \VV_{\underline n}^{\tr}$.

\begin{remark} 
Consider any $\VV\in\mathtt{DPVA}_\mu$ of finite type as a differential algebra, and a dimension vector $\underline n$. 
The previous observations yield the following vertex analog of~\cite[Proposition 7.11.1]{VdB}:
the dual of $X(\bbmu):\Rep(\VV,\underline n)\to\mathfrak g_{\underline n}$ is a comoment map in the sense of Definition~\ref{defcommappva} when we consider $\mathbf H=\Gl_{\underline n}$. 
\end{remark}

\subsection{Constructions for jets}
Let $\AA$ be a double Poisson algebra with moment map $\bbmu$ as in Definition~\ref{def:NCmomentmap}. Consider $\VV=\mathcal J_\infty \AA$ with its double Poisson vertex algebra structure given by Lemma~\ref{Lem:DPAtoDPVA}. Then $(\VV,\bbmu)\in\mathtt{DPVA}_\mu$. Indeed, for any $a\in \AA$, $s\in S$ and $\ell\ge0$ we have
\begin{align*}
\dgal{{\bbmu_s}_\lambda\partial^\ell a}
=(\partial+\lambda)^\ell(ae_s\otimes e_s-e_s\otimes e_sa) 
=\partial^\ell(a)e_s\otimes e_s-e_s\otimes e_s\partial^\ell(a)~\mod \lambda\mathbb{K}[\lambda].
\end{align*}
Working functorially, we get a commutative diagram
\begin{equation}
\begin{tikzcd}        
\mathtt{DPA}\arrow{r}{\mathtt J}&\mathtt{DPVA}\\
\mathtt{DPA}_\mu\arrow{r}{\mathtt J}\arrow{u}&\mathtt{DPVA}_\mu\arrow{u}\end{tikzcd}
\end{equation}
of categories, where vertical functors simply forget the moment maps.

Denote by $\Jinf\DPA_\mu$ the image of the bottom line functor. 
Consider $\Vect_\infty(\AA)\subseteq \Jinf(\AA)$ as in~\eqref{Eq:VectInf} for $(\AA,\bbmu)\in\DPA_\mu$, and $W$ its image by the projection $\Jinf \AA \to\Jinf(\AA)/(\bbmu-\zeta)$ for parameter $\zeta\in B$. Thanks to Example~\ref{Ex:Jet-Gen}, the codomain of this projection is isomorphic to $\Jinf(\AA/(\bbmu-\zeta))$, so that $W\simeq\Vect_\infty (\AA/(\bbmu-\zeta))$.
Using Remark~\ref{rem:DPVAmu-HOPV}, we can apply Proposition~\ref{Pr:HoP-HoPV} and Corollary~\ref{Cor:JDPA-HOPV} to get a functor
\begin{equation}\label{EqJHamRed}
\sharp_\mathtt{V}^\zeta:\Jinf\mathtt{DPA}_\mu\longrightarrow
\mathtt{V}_\infty\HoP
\end{equation}
of noncommutative Hamiltonian vertex reduction. Using Proposition~\ref{propdpahamred}, we get the next result.
 
 \begin{corollary}\label{corohamleft} The following diagram commutes:
 \begin{center}
    \begin{tikzpicture}
 \node  (TopLeft) at (-1,1) {$\DPA_\mu$};
 \node  (TopRight) at (2,1) {$\Jinf\mathtt{DPA}_\mu$};
 \node  (BotLeft) at (-1,-1) {$\HoP$};
 \node  (BotRight) at (2,-1) {$\mathtt{V}_\infty\HoP$};
\path[->,>=angle 90,font=\small]  
   (TopLeft) edge node[above]{$\mathtt{J}$} (TopRight) ;
   \path[->,>=angle 90,font=\small]  
   (BotLeft) edge node[below]{$\Vect_\infty$}  (BotRight) ;
\path[->,>=angle 90,font=\small]  
   (TopLeft) edge node[left]{$\sharp^\zeta$} (BotLeft) ;
\path[->,>=angle 90,font=\small]  
   (TopRight) edge node[right]{$\sharp_\mathtt{V}^\zeta$}  (BotRight) ;
   \end{tikzpicture}
\end{center} 
\end{corollary}

Also note that, thanks to Proposition~\ref{propdpamurep} and Lemma~\ref{Lem:MomapJets}, we get the following.

 \begin{theorem}\label{thmhamtop} The following diagram commutes: 
 \begin{center}
    \begin{tikzpicture}
 \node  (TopLeft) at (-1.5,1) {$\DPA_\mu$};
 \node  (TopRight) at (1.5,1) {$\Jinf\mathtt{DPA}_\mu$};
 \node  (BotLeft) at (-1.5,-1) {$\PA_\mu^{\Gl_{\underline n}}$};
 \node  (BotRight) at (1.5,-1) {$\Jinf\PA_\mu^{\Gl_{\underline n}}$};
\path[->,>=angle 90,font=\small]  
   (TopLeft) edge node[above]{$\mathtt{J}$} (TopRight) ;
   \path[->,>=angle 90,font=\small]  
   (BotLeft) edge node[below]{$\mathtt{J}$}  (BotRight) ;
\path[->,>=angle 90,font=\small]  
   (TopLeft) edge node[left]{$(-)_{\underline n}$} (BotLeft) ;
\path[->,>=angle 90,font=\small]  
   (TopRight) edge node[right]{$(-)_{\underline n}$} (BotRight) ;
   \end{tikzpicture}
\end{center} 
\end{theorem}
\begin{proof}
 The statement is direct if we prove that the functor $(-)_{\underline n}: \Jinf\mathtt{DPA}_\mu \to \Jinf\PA_\mu^{\Gl_{\underline n}}$ exists. This follows if we can show that the function $X(\bbmu)$ obtained from a moment map $\bbmu$ satisfies Definition~\ref{defcommappva}. From the discussion presented after Remark \ref{rem:DPVAmu-HOPV} with $\VV=\Jinf(\AA)$ we are left to show that $X(\bbmu)$ is not only $\Gl_{\underline{n}}$-equivariant, 
 but that it is $\Jinf(\Gl_{\underline{n}})$-equivariant. This is the case because $X(\bbmu)$ is the jet of a  
 $\Gl_{\underline{n}}$-equivariant morphism by Lemma \ref{Lem:JPAGmu}. 
\end{proof}

\begin{theorem} 
\label{Th:cube_Hamiltonian}
The cube depicted in Figure \ref{Fig:HamRed} is commutative.  
\end{theorem}

\begin{proof}
The bottom, right, front, left and top faces are being treated by 
Theorem~\ref{subsubbottom} (bottom face), 
Proposition~\ref{Pr:HamRed_commute}, 
Corollary~\ref{corohamfront}, 
Corollary~\ref{corohamleft} and 
Theorem~\ref{thmhamtop} respectively. Since functors to the back are essentially surjective and full, the back face is also commutative and we are done. 
\end{proof}

\subsection{Poisson reduction from Hamiltonian reduction} 
Let $(\AA,\dgal{-,-},\bbmu)$ be any double Poisson algebra with moment map over $B=\oplus_{s\in S}\kk e_s$.
If we forget about $\bbmu$, we can perform the Poisson reduction presented in Theorem \ref{Thm:rel-ComPoissonRed} and get commutativity of the cube depicted in Figure \ref{Fig:PoiRed} in its $B$-relative form. In particular, the lower face of the cube corresponds to the following commutative diagram: 
\begin{center}
    \begin{tikzpicture}
 \node  (TopLeft) at (-1,1) {$H_0(\AA)$};
 \node  (TopRight) at (2,1) {$H_0(\Vect_\infty(\AA))$};
 \node  (BotLeft) at (-1,-0.6) {$\AA_{\underline{n}}^{\Gl_{\underline n}}$};
 \node  (BotRight) at (2,-0.6) {$\Jinf(\AA_{\underline{n}})^{\Jinf(\Gl_{\underline n})}$};
\path[->,>=angle 90,font=\small]
   (TopLeft) edge (TopRight) ;
   \path[->,>=angle 90,font=\small]
   (BotLeft) edge  (BotRight) ;
\path[->,>=angle 90,font=\small]
   (TopLeft) edge (BotLeft) ;
\path[->,>=angle 90,font=\small]
   (TopRight) edge (BotRight) ; 
   \end{tikzpicture}
\end{center}
This diagram can be obtained by applying Hamiltonian reduction as follows.
We form the algebra $\widetilde{\AA}$ by adding to $\AA$ the generators $(z_s)_{s\in S}$ satisfying $z_s=e_s z_s e_s$.
(If $\AA$ is represented as a quiver with relations, this amounts to add a loop $z_s$ at each vertex $s\in S$.)
We can naturally embed $\AA$ into $\widetilde{\AA}$ and obtain a unique double bracket $\dgal{-,-}^{\sim}$ on  $\widetilde{\AA}$ satisfying
\begin{equation}
\dgal{a,b}^{\sim}=\dgal{a,b}, \quad
 \dgal{a,z_s}^{\sim}=0,\quad \dgal{z_r,z_s}^{\sim}=\delta_{rs} (z_s\otimes e_s - e_s \otimes z_s)\,,
\end{equation}
for any $a,b\in \AA$, $r,s\in S$.
This double bracket is Poisson and it admits $\widetilde{\bbmu}:=\bbmu+\sum_{s\in S}z_s$ as its moment map.
For any $\zeta\in B$, we have an isomorphism $\widetilde{\AA}/(\widetilde{\bbmu}-\zeta)\simeq \AA$ since the ideal is generated by the $|S|$ relations $z_s=\zeta_s-\bbmu_s$.
Therefore, it is clear that applying Hamiltonian reduction to $\widetilde{\AA}$ (for any $\zeta$!) according to Theorem \ref{Th:cube_Hamiltonian} gives the above diagram as the lower face of the cube in Figure~\ref{Fig:HamRed}.

\begin{remark}
It is not possible to deduce Theorem \ref{Thm:ComPoissonRed} from Theorem \ref{Th:cube_Hamiltonian} because an arbitrary double Poisson algebra may not admit a moment map. 
Indeed, any associative $B$-algebra $\AA$ can be endowed with the zero double Poisson bracket; however the existence of a moment map $\bbmu$ entails by \eqref{Eq:NCmomentmap} that $\dgal{\bbmu,a}\neq 0$ for any $a\in \AA\setminus B$. 
\end{remark}

%%%%%%%%%%%%%%%%%% NEW SECTION %%%%%%%%%%%%%%%%%%%%%
%%%%%%%%%%%%%%%%%% NEW SECTION %%%%%%%%%%%%%%%%%%%%%
%%%%%%%%%%%%%%%%%% NEW SECTION %%%%%%%%%%%%%%%%%%%%%
%%%%%%%%%%%%%%%%%% NEW SECTION %%%%%%%%%%%%%%%%%%%%%

\section{Connections with vertex algebras and open problems}
\label{Sec:Open}
We address in this section 
the problem of {\em double versions} for vertex algebras, 
and we discuss  
connections with the Poisson and Hamiltonian reductions. 

\subsection{Double versions of vertex algebras} 
Just like Poisson structures coming from double ones are easier to compute, 
we expect that 
a double version for vertex algebras can be used to compute the OPEs 
that govern the algebraic law in vertex algebras. 

In more detail, we expect a category, let us say $\mathtt{DVA}$, 
of {\em double vertex algebras} 
whose objects should be analogs of 
vertex algebras, 
%and whose axioms have to be determined, 
and the following functors:
 $(-)_N:\mathtt{DVA}\to \mathtt{VA}$ similar to the representation functor for each $N\in \Z_{\geq 0}$,
and $\mathtt{G}:\mathtt{DVA}\to \mathtt{DPVA}$ which is a counterpart to the functor
from $\mathtt{VA}$ to $\mathtt{PVA}$ sending a vertex algebra 
to its graded vertex algebra with respect to Li's filtration. 
Then the composition of this functor with the functor 
$\mathtt{Q}\colon \mathtt{DPVA} \to \mathtt{DPA}$ might be an analog of Zhu's $C_2$-functor 
sending a vertex algebra $\mathcal{V}$ 
to its Zhu's $C_2$ algebra $R_{\mathcal{V}}$. 
%We also expect that this functor is compatible with $(-)_N$ 
%in the sense that there should exist functors $F_N$ from $\mathtt{DVA}$ 
%to $\mathtt{PVA}$ 
We formulate the following problem. 

\begin{problem}
\label{Pb:vertex} 
Assume that there is a family of 
vertex algebras $\mathcal{V}_N$ 
indexed by $N\in \Z_{\geq 0}$
and a double Poisson algebra 
$\AA$ such that 
for all $N \geq 0$, 
we have 
$R_{\mathcal{V}_N} \cong \AA_N$. 
%Let $\mathcal{V}$ be a vertex algebra 
%such that the Poisson structure of the 
%Zhu's $C_2$-algebra $R_{\mathcal{V}}$ comes 
%from a double Poisson structure, 
%that is,
%$R_{\mathcal{V}} \cong \AA_N$
%for some double Poisson algebra $\AA$ 
%and some $N\geq 0$. 
Assume furthermore that the isomorphism \eqref{eq:arc_iso} holds 
for each $N \geq 0$, that is, 
${\rm gr}\,\mathcal{V}_N
\cong \Jinf R_{\mathcal{V}_N}.$ 
Under these hypothesis, is there a 
double analogue of the vertex algebras $\mathcal{V}_N$, 
that is, an object $\mathcal{A}$ of 
$\mathtt{DVA}$ such that $(-)_N$ sends $\mathcal{A}$ to $\mathcal{V}_N$ for each $N$, and  $(\mathtt{Q} \circ \mathtt{G})\mathcal{A}=\mathbb{A}$?
%corresponding to of $\mathcal{V}_N$ that 
%belongs to our expected category $\mathtt{DVA}$ 
%for these vertex algebras? 
%Are there double analogues for vertex algebras 
%satisfying these conditions? 
\end{problem}

The {\em Zhu's algebra} Zhu$(\mathcal{V})$  of a vertex algebra $\mathcal{V}$ 
is a certain quotient (\cite{Zhu}) which has a natural 
 structure of a filtered associative algebra 
whose graded algebra ${\rm gr}({\rm Zhu}\, \mathcal{V})$ is commutative.  
In addition to the above conditions, it is also reasonable to assume that 
${\rm gr}({\rm Zhu}\, \mathcal{V}_N) \cong R_{\mathcal{V}_N}$ 
for each $N \geq 0$. 
In general, for an arbitrary vertex algebra 
$\mathcal{V}$ we only have a surjective Poisson algebra morphism 
$R_{\mathcal{V}} \twoheadrightarrow {\rm gr}({\rm Zhu}\, \mathcal{V})$, 
\cite[Proposition 2.17(c)]{DSK} and \cite[Proposition 3.3]{ALY}.

We have already encountered examples of such vertex algebras in the introduction, 
cf.~\S\ref{ss:IntroVA}.  
Let us give some other examples. 
%\subsubsection{Universal affine vertex algebra}
Consider the {\em universal affine vertex algebra} 
\begin{align}
\label{eq:VOA}
\mathcal{V}^\kappa(\g) =U(\widehat{\g}) \otimes_{\g[t]\oplus \kk {\bf 1}} \kk
\end{align}
associated with a reductive Lie algebra $\g$ and an invariant symmetric bilinear
form $\kappa$,   
where $\widehat{\g}:=\g[t^\pm1] \oplus \kk {\bf 1}$ 
is the affine Kac-Moody algebra  
with Lie bracket 
\begin{align*}
& [xt^n,yt^m] = [x,y]t^{m+n} +  m\delta_{m+n,0} \kappa(x,y) {\bf 1}, 
& x,y \in \g, \, m,n \in \Z_{\geq 0},  
\end{align*}
the element ${\bf 1}$ being central. 
In \eqref{eq:VOA}, $\kk$ stands for the one-dimensional representation of 
$\g[t]\oplus \kk {\bf 1}$ on which $\g[t]$ acts
trivially and ${\bf 1}$ acts as the identity.  
It is well-known that the Zhu's algebra of $\mathcal{V}^\kappa(\g)$ is the enveloping algebra $U(\g)$ 
filtered by the  
PBW filtration and that, as Poisson algebras, we have: 
$${\rm gr}({\rm Zhu}\,\mathcal{V}^\kappa(\g)) \cong 
{\rm gr}\,U(\g) \cong \kk[\g^*]  \cong R_{\mathcal{V}^\kappa(\g)},$$
 with the Kirillov-Kostant-Souriau Poisson structure. 
 On the other hand,  
${\rm gr}\,\mathcal{V}^\kappa(\g) \cong S(t^{-1} \g [t^{-1}]),$ 
 the symmetric algebra of 
$t^{-1} \g [t^{-1}]:=\g\otimes t^{-1}  \kk[t^{-1}]$, which is canonically 
isomorphic to $\Jinf \kk[\g^*]$ through the mapping 
$xt^{-n-1} \mapsto \del^{n} x$, for $n\geq 0$ and $x\in \g \subset \kk[\g^*]$. 

Assume now that $\g=\gl_N$. 
Recall that 
$\kk[\gl_N^\ast] \cong \AA_N$, 
where $\AA=\kk[a]$ is the double Poisson algebra from Example \ref{Ex:ku} 
with $\beta=\gamma=0$. 
To sum up,  
the conditions of Problem~\ref{Pb:vertex} hold for the 
family 
of vertex algebras  
$\mathcal{V}^\kappa(\gl_N)$ 
with any $\kappa$. 
One can draw a picture similar to Figure~\ref{Fig:cdo}. 
Like the vertex algebra of chiral differential operators, one can wonder 
whether there are double versions of $U(\gl_N)$ 
and $\mathcal{V}^\kappa(\gl_N)$, respectively; cf.~Figure \ref{Fig:VOA}. 
Namely, we expect a category, let us say 
$\mathtt{NCDPA}$,  
of noncommutative double Poisson algebras together with a functor
from $\mathtt{DVA}$ to $\mathtt{NCDPA}$, 
similar to Zhu functor, and
another functor from $\mathtt{NCDPA}$ to 
$\mathtt{DPA}$, an analogue of the functor sending  
${\rm Zhu}(\mathcal{V})$ to ${\rm gr}\,{\rm Zhu}(\mathcal{V})$ 
for $V$ a vertex algebra, 
such that the 
double version of $U(\gl_N)$ is 
sent to $\kk[a]$ under this functor from $\mathtt{NCDPA}$ to $\mathtt{DPA}$.

\medskip

{\tiny
\begin{figure}
\centering
\begin{tikzpicture}[scale=.8]
 \node  (TopLeft) at (-5,2) {$\Jinf \kk[a]$};
 \node  (TopRight) at (3,2) {$\Jinf \kk[\gl_N^\ast]$};
 \node  (BotLeft) at (-5,-2) {$\kk[a]$}; 
 \node  (BotRight) at (3,-2) {$\kk[\gl_N^\ast]$};
% back
 \node  (BackTopLeft) at (-2,4) {$?$};
 \node  (BackTopRight) at (6,4) {$\mathcal{V}^\kappa(\gl_N)$};
 \node   (BackBotLeft) at (-2,0) {$?$}; 
 \node  (BackBotRight) at (6,0) {$U(\gl_N)$};
% front arrows 
\path[->,>=angle 90,font=\small]  
   (TopLeft) edge node[above] {\;\;{\tiny $(-)_N$}} (TopRight) ;
   \path[->,>=angle 90,font=\small]  
   (BotLeft) edge node[above] {\;\;{\tiny $(-)_N$}} (BotRight) ;
\path[<-,>=angle 90,font=\small]  
   (TopLeft) edge node[left,above] {\!\!\!\!\!\!\!{\tiny $\mathtt{J}$}} (BotLeft) ;
\path[<-,>=angle 90,font=\small]  
   (TopRight) edge node[right,above] {\quad{\tiny $\mathtt{J}$}} (BotRight) ;
% back arrows
\path[->,>=angle 90,dashed,font=\small]  
   (BackTopLeft) edge (BackTopRight) ;
   \path[->,>=angle 90,dashed,font=\small]  
   (BackBotLeft) edge (BackBotRight) ;
\path[->,>=angle 90,dashed,font=\small]  
   (BackTopLeft) edge (BackBotLeft) ;
\path[->,>=angle 90,font=\small]  
   (BackTopRight) edge node[right] {{\tiny $\text{Zhu}(-)$}} (BackBotRight) ;
% back to front arrows
\path[->,>=angle 90,dashed,font=\small]  
   (BackTopLeft) edge (TopLeft) ;
   \path[->,>=angle 90,dashed,font=\small]  
   (BackBotLeft) edge (BotLeft) ;
\path[->,>=angle 90,font=\small]  
   (BackTopRight) edge node[right] {\,\;{\tiny ${\rm gr}$}} (TopRight) ;
\path[->,>=angle 90,font=\small]  
   (BackBotRight) edge node[right] {\;\;{\tiny ${\rm gr}$}} (BotRight) ;
%%%% categories 
\node[gray,fill=gray!10,rounded corners, above=0cm of TopLeft] (CatTL) {$\DPVA$};
\node[gray,fill=gray!10,rounded corners, above=0cm of TopRight] (CatTR) {$\PVA$};
\node[gray,fill=gray!10,rounded corners, below=-0.1cm of BotLeft] (CatBL) {$\DPA$};
\node[gray,fill=gray!10,rounded corners, below=-0.1cm of BotRight] (CatBR) {$\PA$};
% back
\node[gray,fill=gray!10,rounded corners, above=0cm of BackTopLeft] (CatBTL) {$?$};
\node[gray,fill=gray!10,rounded corners, above=0cm of BackTopRight] (CatBTR) {$\mathtt{VA}$};
\node[gray,fill=gray!10,rounded corners, below=-0.1cm of BackBotLeft] (CatBBL) {$?$};
\node[gray,fill=gray!10,rounded corners, below=-0.1cm of BackBotRight] (CatBBR) {$\mathtt{AA}$};
   \end{tikzpicture}
   \caption{Universal affine vertex algebra and relative objects.} 
   \label{Fig:VOA}
\end{figure}
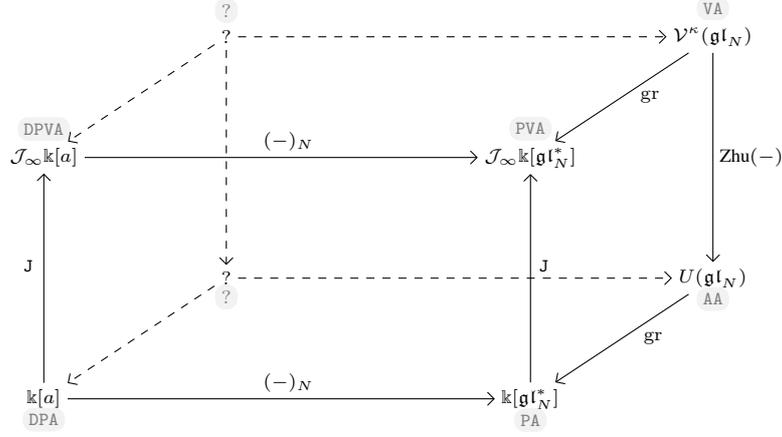
}

%\subsubsection{Feigin-Frenkel center} \label{sss:FF_center}
One can build another interesting example from the previous one. 
Let us consider the {\em Feigin-Frenkel center} (\cite{FF}) of $\mathcal{V}^{\kappa_c}(\g)$ 
for the {\em critical level} 
$$\kappa_c := - \dfrac{1}{2} \kappa_{\g},$$
where $\kappa_{\g}$ is the Killing form of $\g$.  
This is the vertex center of $\mathcal{V}^{\kappa_c}(\g)$:
$$\mathfrak{z}(\widehat{\g}) = \mathcal{V}^{\kappa_c}(\g)^{\g\[[t\]]}
:=\{u \in \mathcal{V}^{\kappa_c}(\g) \mid (xt^n).u=0 \text{ for all } x\in\g, \, n\in\Z_{\geq 0}\}.$$
Note that the vertex center of $\mathcal{V}^\kappa(\g)$ is trivial outside 
the critical level.
The Zhu's $C_2$-algebra of $\mathfrak{z}(\widehat{\g})$ 
is $\kk[\g^*]^{\GG} = \kk[\g/\!\!/\GG]$, 
the Poisson center of $\kk[\g^*]$, where $\GG$ is the adjoint group of $\g$. 
The Zhu's algebra of $\mathfrak{z}(\widehat{\g})$ 
is the center $Z(\g)$ of the enveloping algebra of $\g$,  
and we have: 
$${\rm gr}\, \mathfrak{z}(\widehat{\g}) \cong \Jinf  \kk[\g/\!\!/\GG]
\cong \Jinf \kk[\g]^{\Jinf(\GG)},$$
the second isomorphism resulting from \cite{RT,BD1}; see Example \ref{Ex:inv_adjoint}. 
In particular, the isomorphism \eqref{eq:arc_iso} holds for this example, too. 

For $\g=\gl_N$, this is compatible with the $H_0$-reduction on the double Poisson side.  
Indeed $H_0(\AA)=\kk[a]$, with $\AA=\kk[a]$ the double Poisson algebra as in the previous 
example, equipped with the trivial Lie structure ($\AA$ is commutative). 
All this is depicted in Figure \ref{Fig:FFcenter}. 

\medskip

{\tiny
\begin{figure}
\centering
\begin{tikzpicture}[scale=.8]
 \node  (TopLeft) at (-5,2) {$H_0({\rm Vect}_\infty \kk[a])$};
 \node  (TopRight) at (3,2) {$\Jinf \kk[\gl_N^*]^{\Jinf(\Gl_N)}$};
 \node  (BotLeft) at (-5,-2) {$\kk[a]=H_0(\kk[a])$}; 
 \node  (BotRight) at (3,-2) {$\kk[\gl_N^*/\!\!/\Gl_N]$};
% back
 \node  (BackTopLeft) at (-2,4) {$?$};
 \node  (BackTopRight) at (6,4) {$\mathfrak{z}(\widehat{\gl}_N)$};
 \node   (BackBotLeft) at (-2,0) {$?$}; 
 \node  (BackBotRight) at (6,0) {$Z(\gl_N)$}; 
\path[->,>=angle 90,font=\small]  
   (TopLeft) edge node[above] {\;\;{\tiny $\tr_N$}} (TopRight) ;
   \path[->,>=angle 90,font=\small]  
   (BotLeft) edge node[above] {\;\;{\tiny $\tr_N$}} (BotRight) ;
\path[<-,>=angle 90,font=\small]  
   (TopLeft) edge node[left] {\!{\tiny $\mathtt{Vect}_\infty$}} (BotLeft) ;
\path[<-,>=angle 90,font=\small]  
   (TopRight) edge node[right,above] {\quad{\tiny $\mathtt{J}$}} (BotRight) ;
% back arrows
\path[->,>=angle 90,dashed,font=\small]  
   (BackTopLeft) edge (BackTopRight) ;
   \path[->,>=angle 90,dashed,font=\small]  
   (BackBotLeft) edge (BackBotRight) ;
\path[->,>=angle 90,dashed,font=\small]  
   (BackTopLeft) edge (BackBotLeft) ;
\path[->,>=angle 90,font=\small]  
   (BackTopRight) edge node[right] {{\tiny $\text{Zhu}(-)$}} (BackBotRight) ;
% back to front arrows
\path[->,>=angle 90,dashed,font=\small]  
   (BackTopLeft) edge (TopLeft) ;
   \path[->,>=angle 90,dashed,font=\small]  
   (BackBotLeft) edge (BotLeft) ;
\path[->,>=angle 90,font=\small]  
   (BackTopRight) edge node[right] {\,\;{\tiny ${\rm gr}$}} (TopRight) ;
\path[->,>=angle 90,font=\small]  
   (BackBotRight) edge node[right] {\;\;{\tiny ${\rm gr}$}} (BotRight) ;
%%%% categories 
\node[gray,fill=gray!10,rounded corners, above left=0cm and -1.2cm of TopLeft] (CatTL) {$\mathtt{V}_\infty\HoP$};
\node[gray,fill=gray!10,rounded corners, above=0cm of TopRight] (CatTR) {$\PVA$};
\node[gray,fill=gray!10,rounded corners, below=-0.1cm of BotLeft] (CatBL) {$\HoP$};
\node[gray,fill=gray!10,rounded corners, below=-0.1cm of BotRight] (CatBR) {$\PA$};
% back
\node[gray,fill=gray!10,rounded corners, above=0cm of BackTopLeft] (CatBTL) {$?$};
\node[gray,fill=gray!10,rounded corners, above=0cm of BackTopRight] (CatBTR) {$\mathtt{VA}$};
\node[gray,fill=gray!10,rounded corners, below=-0.1cm of BackBotLeft] (CatBBL) {$?$};
\node[gray,fill=gray!10,rounded corners, below=-0.1cm of BackBotRight] (CatBBR) {$\mathtt{AA}$};
   \end{tikzpicture}
   \caption{Feigin-Frenkel center and relative objects.} 
   \label{Fig:FFcenter}
\end{figure}
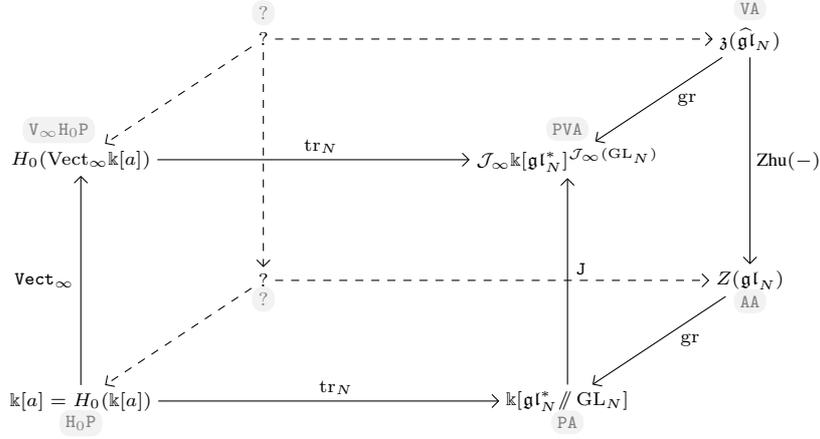
}

\subsection{Hamiltonian reduction in a more general setting}
\label{ss:Pb_red}
To motivate our second problem, let us consider an illustrating example.  
On the cotangent bundle $T^*\GG$  
 there are two commuting Hamiltonian 
 $\GG$-actions:  
\begin{align}\label{eq:G-action-on-T*G} 
g \cdot_L(h,x)=(hg^{-1}, ({\rm Ad}^*g).x),\quad 
g\cdot_R(h,x)=(gh,x), \quad g,h \in \GG, \,x\in \g^*,
\end{align} 
where ${\rm Ad}^*$ is the coadjoint action, 
with corresponding moment maps  
\begin{align*}  
\mu_L\colon T^* \GG \longrightarrow  \g^*, \, (h,x) \longmapsto x, 
\quad \mu_R\colon T^* \GG \longrightarrow  \g^*,\, (h,x) \longmapsto - ({\rm Ad}^*h).x. 
\end{align*} 

For $\GG=\Gl_N$, the Poisson structure on $\kk[T^*\Gl_N]$ 
comes from the double Poisson algebra 
$\AA=\kk\langle a,b^{\pm1}\rangle$ as in Example \ref{Exmp:SympGL}. 
This algebra is equipped with the noncommutative moment 
map $\bbmu=a-bab^{-1}$ which yields the moment 
map $X(\bbmu) = \mu_L+\mu_R$. 
It could be interesting to consider as well noncommutative 
moment maps for $\mu_L$ and $\mu_R$, respectively.  
In this example, natural candidates are easy: $\bbmu_L=a$ 
and $\bbmu_R=-bab^{-1}$, see Appendix~\ref{App:cot}. 
We notice that the embedding 
$\kk[a]\hookrightarrow \AA$ corresponds to the reduction by the left action $(\,\cdot_R)$ since the 
$\gl_N^\ast$ component is unchanged, 
and the embedding   $\kk[bab^{-1}]\hookrightarrow \AA$ 
corresponds to the reduction by the right action $(\,\cdot_L)$ since $h x h^{-1}$ 
generates $\kk[T^*\Gl_N]^{(\Gl_N,\cdot_L)}$. 

This leads us to the second problem. 

\begin{problem}
\label{Pb:reduction}
Are there analogues of $H_0$-Poisson structures 
and noncommutative moment maps that yield
more general Hamiltonian actions on representation spaces 
(not necessarily coming from the action by conjugation 
of $\Gl_N$)?
\end{problem}

Consider the vertex algebra of chiral differential operators 
$\mathscr{D}_{\GG,\kappa}^{ch}$ associated with the smooth scheme 
$\GG$ and the invariant symmetric bilinear form~$\kappa$ (\cite{MSV,BD2}): 
$$ \mathscr{D}_{\GG,\kappa}^{ch}=\mathcal{V}^\kappa(\g) 
\otimes_{\g\[[t\]]\oplus \kk{\bf 1}} \kk[\Jinf \GG],$$
where $\g\[[t\]]$ acts as left-invariant vector fields on $\kk[\Jinf \GG]$ 
and ${\bf 1}$ as the identity.  
This is a particular case of the examples in \S\ref{ss:IntroVA}. 
There are vertex algebra embeddings: 
\begin{align*}
\pi_L\colon \mathcal{V}^{\kappa^*}(\g) \longrightarrow \mathscr{D}_{\GG,\kappa}^{ch}, 
\quad \pi_R\colon \mathcal{V}^{\kappa}(\g) \longrightarrow \mathscr{D}_{\GG,\kappa}^{ch}, 
\end{align*}  
where 
$\kappa^*=-\kappa-\kappa_{\g}$ 
such that 
\begin{align}
\label{eq:iso_cdo}
(\mathscr{D}_{\GG,\kappa}^{ch})^{\pi_L(\g\[[t\]])}\cong \mathcal{V}^{\kappa^*}(\g), 
\quad (\mathscr{D}_{\GG,\kappa}^{ch})^{\pi_R(\g\[[t\]])}\cong \mathcal{V}^\kappa(\g).
\end{align} 
The vertex algebra morphisms $\pi_L$ and $\pi_R$  
are {\em chiral quantized comoment maps}  
of $\mu_L^*$ and $\mu_R^*$, respectively,  
in the sense that 
$${\rm gr}\,\pi_{L,R} =\Jinf (\mu_{L,R}) ^*
\colon \kk[\Jinf \g^*] \longrightarrow {\rm gr}\,\mathscr{D}_{\GG,\kappa}^{ch}.$$  
The Zhu's algebra of $\mathscr{D}_{\GG,\kappa}^{ch}$ 
is the algebra $\mathcal{D}_{\GG} \cong  \kk[\GG] \otimes U(\g)$ 
of global differential operators on $\GG$, 
and the subalgebra of left (or right) $\GG$-invariants differential operators 
on $\GG$ identifies with $U(\g)$; 
if we consider the invariants by two actions, we get the center $Z(\g)$ 
of $U(\g)$. 
For $\kappa=\kappa^*=\kappa_c$, then there is a vertex algebra 
morphism 
$$\pi_L\otimes \pi_R 
\colon \mathcal{V}^{\kappa_c}(\g)
\otimes \mathcal{V}^{\kappa_c}(\g)  \longrightarrow \mathscr{D}_{\GG,\kappa_c}^{ch}$$ 
such that (cf.~\cite{A18,AM}): 
$$(\mathscr{D}_{\GG,\kappa_c}^{ch})^{\g\[[t\]] \times \g\[[t\]] }
:= \pi_L( \mathcal{V}^{\kappa_c}(\g)) \cap \pi_R( \mathcal{V}^{\kappa_c}(\g)) 
\cong \mathfrak{z}(\widehat{\g}).$$ 

For $\GG=\Gl_N$, the Poisson structure on the Zhu $C_2$-algebra $R_{\mathscr{D}_{\Gl_N,\kappa}^{ch}} 
\cong \kk[T^*\Gl_N]$ comes from the double Poisson algebra 
$\AA=\kk\langle a,b^{\pm1}\rangle$ as in Example \ref{Exmp:SympGL}. 
This is a particular case of the examples in \S\ref{ss:IntroVA}.
We illustrate the reduction for $\mathscr{D}_{\Gl_N,\kappa}^{ch}$ and related objects 
with respect to the left action 
 in Figure \ref{Fig:inv_cdo}. 

{\tiny
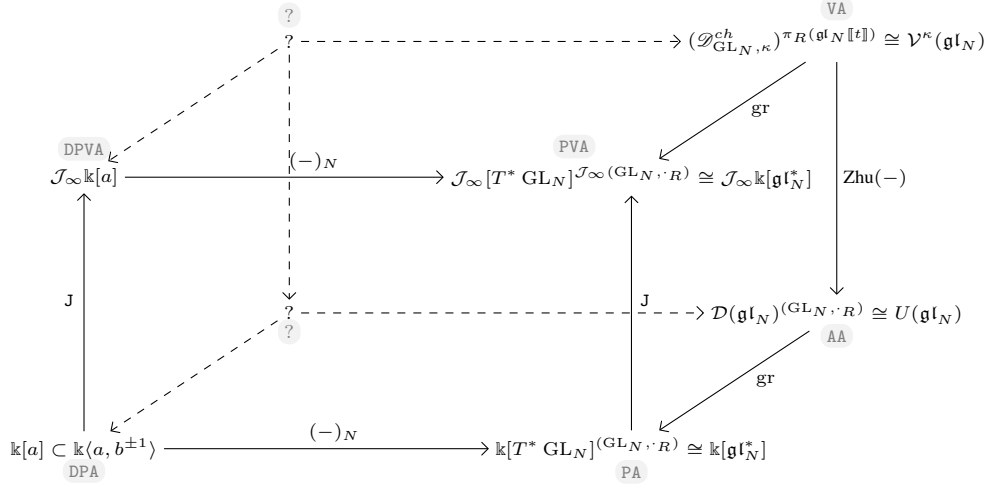
\begin{figure}
\centering
\begin{tikzpicture}[scale=.9]
 \node  (TopLeft) at (-5,2) {$\Jinf \kk[a]$};
 \node  (TopRight) at (3,2) {$\Jinf [T^*\Gl_N]^{{\Jinf(\Gl_N,\cdot_R)}} 
 \cong \Jinf \kk[\gl_N^\ast]$};
 \node  (BotLeft) at (-5,-2) {$\kk[a] \subset \kk\langle a,b^{\pm1}\rangle$}; 
 \node  (BotRight) at (3,-2) {$\kk[T^*\Gl_N]^{(\Gl_N,\cdot_R)} \cong \kk[\gl_N^\ast]$};
% back
 \node  (BackTopLeft) at (-2,4) {$?$};
 \node  (BackTopRight) at (6,4) {$(\mathscr{D}_{\Gl_N,\kappa}^{ch})^{\pi_R(\mathfrak{g\!l}_N\[[t\]])}
 \cong \mathcal{V}^{\kappa}(\gl_N)$};
 \node   (BackBotLeft) at (-2,0) {$?$}; 
 \node  (BackBotRight) at (6,0) {$\mathcal{D}(\gl_N)^{(\Gl_N,\cdot_R)}\cong U(\gl_N)$};
% front arrows 
\path[->,>=angle 90,font=\small]  
   (TopLeft) edge node[above] {\;\;\;\;\;\;\;\;{\tiny $(-)_N$}} (TopRight) ;
   \path[->,>=angle 90,font=\small]  
   (BotLeft) edge node[above] {\;\;{\tiny $(-)_N$}} (BotRight) ;
\path[<-,>=angle 90,font=\small]  
   (TopLeft) edge node[left,above] {\!\!\!\!\!\!\!{\tiny $\mathtt{J}$}} (BotLeft) ;
\path[<-,>=angle 90,font=\small]  
   (TopRight) edge node[right,above] {\quad{\tiny $\mathtt{J}$}} (BotRight) ;
% back arrows
\path[->,>=angle 90,dashed,font=\small]  
   (BackTopLeft) edge (BackTopRight) ;
   \path[->,>=angle 90,dashed,font=\small]  
   (BackBotLeft) edge (BackBotRight) ;
\path[->,>=angle 90,dashed,font=\small]  
   (BackTopLeft) edge (BackBotLeft) ;
\path[->,>=angle 90,font=\small]  
   (BackTopRight) edge node[right] {{\tiny $\text{Zhu}(-)$}} (BackBotRight) ;
% back to front arrows
\path[->,>=angle 90,dashed,font=\small]  
   (BackTopLeft) edge (TopLeft) ;
   \path[->,>=angle 90,dashed,font=\small]  
   (BackBotLeft) edge (BotLeft) ;
\path[->,>=angle 90,font=\small]  
   (BackTopRight) edge node[right] {\,\;{\tiny ${\rm gr}$}} (TopRight) ;
\path[->,>=angle 90,font=\small]  
   (BackBotRight) edge node[right] {\;\;{\tiny ${\rm gr}$}} (BotRight) ;
%%%% categories 
\node[gray,fill=gray!10,rounded corners, above=0cm of TopLeft] (CatTL) {$\DPVA$};
\node[gray,fill=gray!10,rounded corners, above left=0cm and -2cm of TopRight] (CatTR) {$\PVA$};
\node[gray,fill=gray!10,rounded corners, below=-0.1cm of BotLeft] (CatBL) {$\DPA$};
\node[gray,fill=gray!10,rounded corners, below=-0.1cm of BotRight] (CatBR) {$\PA$};
% back
\node[gray,fill=gray!10,rounded corners, above=0cm of BackTopLeft] (CatBTL) {$?$};
\node[gray,fill=gray!10,rounded corners, above=0cm of BackTopRight] (CatBTR) {$\mathtt{VA}$};
\node[gray,fill=gray!10,rounded corners, below=-0.1cm of BackBotLeft] (CatBBL) {$?$};
\node[gray,fill=gray!10,rounded corners, below=-0.1cm of BackBotRight] (CatBBR) {$\mathtt{AA}$};
   \end{tikzpicture}
   \caption{Poisson reduction for the vertex algebra 
   of chiral differential operators on $\Gl_N$ with respect to the left action and relative objects.} 
   \label{Fig:inv_cdo}
\end{figure}
}

This example 
%, together with the example of \S\ref{sss:FF_center} 
%(see Figure \ref{Fig:FFcenter}) 
%obtained by {\em reduction} from $\mathcal{V}^{\kappa_c}(\gl_N)$, 
suggests that an answer to Problem \ref{Pb:vertex} 
must be compatible with the reduction, possibly in a wider sense (see  
Problem~\ref{Pb:reduction}).

\subsection{Double Poisson version of Slodowy slices}
The previous notion of a chiral quantized comoment map 
is particularly useful when we wish to perform the BRST reduction 
(that is, a certain chiral quantized Hamiltonian reduction) used 
to define {\em affine $W$-algebras} 
which are obtained by the Drinfeld--Sokolov reduction 
from affine vertex algebras, see \cite{KRW,FF0}. 
The Zhu's $C_2$-algebras of the universal affine $W$-algebras 
are the coordinate rings of {\em Slodowy slices} associated with nilpotent elements,  
obtained by Hamiltonian reduction 
from a reductive Lie algebra $\g\cong \g^*$. 
For $\g=\gl_N$, it was proved by Maffei \cite{Ma} that Slodowy slices in a nilpotent orbit closure
can be described in term of quiver varieties. 
So the following problem is natural. 

\begin{problem}
\label{Pb:Slodowy}
Can we describe the Poisson structure of Slodowy slices 
in $\gl_N$ from double Poisson algebras? 
\end{problem}

Here, we have a quite clear idea of how to solve this problem. 
This will be the subject of a future work. 
Combing the three problems, our hope would be to describe the OPEs for 
the affine $W$-algebras associated with $\gl_N$ in terms of double Poisson 
vertex brackets. 
%Indeed, here again, we expect that there exist a double vertex analog %$\mathcal{W}_f$, 
%which does not depend on $N$,  
%for the family of the $W$-algebras associated with $\gl_N$. 
%and a given nilpotent element $f \in \gl_N$. 
In particular, for the regular $W$-algebras $\mathcal{W}^k(\gl_N)$,   
that is, those associated with a regular nilpotent element  
at level $k \in \kk$, 
there should be a double vertex analog $\mathcal{A}^k$ (independent of~$N$)
corresponding to the expected double Poisson algebra describing 
the Poisson structure of the regular Slodowy slices. 
This is coherent with the fact that the $W$-algebras  
$\mathcal{W}^k(\gl_N)$ are all quotients of the 
universal 2-parameter $W_{1+\infty}$-algebra, which is independent of $N$; cf.~\cite{Lin}. 
Thus the existence of such a double vertex analog 
should give a compact way of encoding the OPE's 
of $W_{1+\infty}$.\footnote{We thank the anonymous 
referee for this interesting observation.}

\subsection{Double vertex analogues of quasi Poisson structures} \label{ss:qPVA}

Recall from \S\ref{ss:QuivRep} that Van den Bergh defined the double Poisson bracket \eqref{Eq:dbr-quiver}
on the path algebra of an arbitrary double quiver $\overline{Q}$ so that it induces
the canonical Poisson structure of the corresponding quiver varieties $\Rep(\kk\overline{Q}/(\bbmu-\zeta),\underline{n})/\!\!/\Gl_{\underline{n}}$.
In fact, Van den Bergh's original construction was motivated by unveiling a Poisson structure on \emph{multiplicative quiver varieties} \cite{CBS}.
This required the introduction of \emph{double quasi-Poisson algebras} \cite{VdB}, which are non-commutative versions of the algebra of functions on a quasi-Poisson manifold \cite{AKSM}.
Roughly speaking, the definition of a (resp.~double) Poisson bracket is modified so that the right-hand side of the Jacobi identity \eqref{Eq:J1} (resp. \eqref{Eq:DJ}) takes a specific nontrivial form; in the geometric picture this is given by the infinitesimal action of the Cartan trivector $\phi\in \wedge^3 \mathrm{Lie}(\GG)$ of the algebraic group $\GG$ acting on the algebra of functions.
These notions are compatible and lead to the following analogue of Proposition \ref{Pr:ComFront}.

\begin{proposition} \label{Pr:ComFront-q}
Fix $N\geq 1$.
 The following diagram is commutative:
 \begin{center}
    \begin{tikzpicture}
 \node  (TopLeft) at (-2,1) {$q\DPA$};
\node  (TopRight) at (1.5,1) {$q\PA^{\Gl_N}$};
\node  (BotLeft) at (-2,-1) {$\HoP$};
\node  (BotRight) at (1.5,-1) {$\PA$};
\path[->,>=angle 90,font=\small]
   (TopLeft) edge node[above] {$(-)_N$} (TopRight) ;
\path[->,>=angle 90,font=\small]
    (BotLeft) edge node[above] {$\tr_N$} (BotRight) ;
\path[->,>=angle 90,font=\small]
    (TopLeft) edge node[left] {$\sharp$}  (BotLeft) ;
\path[->,>=angle 90,font=\small]
    (TopRight) edge node[right] {$\mathtt{R}$} (BotRight) ;
    %%% right diagram
 \node  (TopLeft2) at (4,1) {$\AA$};
\node  (TopRight2) at (7,1) {$\AA_N$};
\node  (BotLeft2) at (4,-1) {$H_0(\AA)$};
\node  (BotRight2) at (7,-1) {$\AA_N^{\Gl_N}$};
\path[->,>=angle 90,font=\small]
   (TopLeft2) edge (TopRight2) ;
\path[->,>=angle 90,font=\small]
    (BotLeft2) edge (BotRight2) ;
\path[->,>=angle 90,font=\small]
    (TopLeft2) edge (BotLeft2) ;
\path[->,>=angle 90,font=\small]
    (TopRight2) edge (BotRight2) ;
   \end{tikzpicture}
\end{center}
where $q\DPA$ (resp. $q\PA^{\Gl_N}$) in the category of double quasi-Poisson algebras (resp. quasi-Poisson algebras with a $\Gl_N$-action by quasi-Poisson automorphisms).
\end{proposition}
The proof is a straightforward adaptation of the Poisson case, which is again based on \cite{CB,F22,VdB}. (This holds in the $B$-relative case, too.)
In fact, we can derive an analogue of Corollary \ref{corohamfront} for quasi-Hamiltonian reduction since there is a notion of \emph{multiplicative moment maps} in the quasi-Poisson case, cf. \cite{AKSM,VdB}.

\begin{problem}
\label{Pb:Quasi}
What is the vertex analogue of the categories $q\DPA$ and $q\PA^{\Gl_N}$?
Are there properties that completely characterize the vertex analogues of multiplicative moment maps?
\end{problem}
The crux of this problem is to understand what type of structure is obtained under application of the jet functor $\mathtt{J} \colon (\mathtt{Comm})\mathtt{Alg}\to (\mathtt{Comm})\mathtt{DiffAlg}$ to a (double) quasi-Poisson algebra.
This issue is far from being trivial since it requires to understand how to modify the Jacobi identities \eqref{Eq:A3} and \eqref{Eq:DA3} in the vertex case.
Assuming that this problem can be solved, it raises the question of obtaining the commutative cubes from Figures \ref{Fig:PoiRed} and \ref{Fig:HamRed} in the quasi-Poisson setting.
The latter case shall lead to Poisson vertex algebras associated with multiplicative quiver varieties that may be of independent interest.

\medskip

\appendix 

\section{The moment maps on the cotangent bundle}
\label{App:cot} 
We view $\gl_N^\ast$ as $\Mat_N$ and consider `entry functions' $a_{ij}\in \kk[\gl_N^\ast]$, $\xi \mapsto \xi_{ij}$, $1\leq i,j\leq N$.
An element $x\in \gl_N$ is seen as a linear function on $\gl_N^\ast$ by $x(\xi):=\langle \xi,x \rangle$. 
We identify $\gl_N\simeq \gl_N^\ast =\Mat_N$ so that $\langle -,-\rangle$ is the trace pairing. 
Denote by $E_{ij}\in \gl_N$ the elementary matrix having $+1$ as $(i,j)$ entry and $0$ for all other entries.  
The identification yields $E_{ij}=\tr(E_{ij}\xi)=\xi_{ji}$ so that $E_{ij}(-)=a_{ji}$. 
The Poisson bracket on $\kk[\gl_N^\ast]$ then satisfies 
\begin{equation}
 \br{a_{ij},a_{kl}} = \br{E_{ji},E_{lk}}=[E_{ji},E_{lk}] = a_{kj} \delta_{il} - \delta_{kj} a_{il}\,.
\end{equation}

For $f\in \kk[\Gl_N]$, the left- and right-invariant vector field associated with $x\in \gl_N$ are given by  (in their analytic form)
\begin{equation*}
 x_L(f)(h):=\frac{d}{dt}\Big|_{t=0} f(h e^{tx})\,, \quad 
 x_R(f)(h):=\frac{d}{dt}\Big|_{t=0} f(e^{tx} h)\,.
\end{equation*} 
On the coordinate functions $b_{ij}\in \kk[\Gl_N]$, $h\mapsto h_{ij}$, we find that 
$x_L(b_{kl})=(bx)_{kl}$ and $x_R(b_{kl})=(xb)_{kl}$ after forming the matrix-valued function $b=(b_{kl})_{1\leq k,l\leq N}$. 
Thus the Poisson bracket on $T^\ast\Gl_N$ is such that (cf.~\S\ref{ss:Pb_red}) 
\begin{equation}
 \br{a_{ij},b_{kl}}=(E_{ji})_L(b_{kl}) =(bE_{ji})_{kl}=b_{kj} \delta_{il}\,.
\end{equation}
The Poisson bracket is zero for 2 functions depending only on the base, hence $\br{b_{ij},b_{kl}}=0$. 
Writing these identities according to \eqref{Eq:relPA} after realising $\kk[T^\ast\Gl_N]\simeq \kk\langle a,b^{\pm1}\rangle_N$ in the obvious way, we precisely get the double bracket from Example \ref{Exmp:SympGL}.

The actions \eqref{eq:G-action-on-T*G} of $\Gl_N$ give at the level of functions  
\begin{equation*}
 g\cdot_L(b_{ij},a_{ij})=((bg)_{ij},(g^{-1}ag)_{ij})\,, \quad 
 g\cdot_R(b_{ij},a_{ij})=((g^{-1}b)_{ij},a_{ij})\,, \quad g\in \Gl_N\,,
\end{equation*}
from which we get the following infinitesimal action 
\begin{equation*}  
 x\cdot_L(b_{ij},a_{ij})=((bx)_{ij},(ax-xa)_{ij})\,, \quad 
 x\cdot_R(b_{ij},a_{ij})=(-(xb)_{ij},0)\,, \quad x\in \gl_N\,.
\end{equation*}
The corresponding moment maps $\mu_{L,R}:T^\ast\Gl_N\to \gl_N^\ast$ must satisfy 
$$x\cdot_{L,R}f=\br{\langle \mu_{L,R}(-),x\rangle, f}\,, \quad \text{ for }x\in \gl_N,\,\, f\in \kk[T^\ast \Gl_N]\,.$$  
It suffices to look at these identities for the functions $a_{kl},b_{kl}$. 
We have on the one hand, 
\begin{equation}
 \br{(\mu_L)_{ij},b_{kl}}=E_{ji}\cdot_L b_{kl}=b_{kj} \delta_{il}\,, \quad 
  \br{(\mu_L)_{ij},a_{kl}}=E_{ji}\cdot_L a_{kl}=a_{kj} \delta_{il}- \delta_{kj} a_{il}\,. 
\end{equation}
Using \eqref{Eq:relPA}, this means $\dgal{\bbmu_L,b}=b\otimes 1$ and $\dgal{\bbmu_L,a}=a\otimes 1 - 1 \otimes a$ 
for some $\bbmu_L$ such that $\mu_L=X(\bbmu_L)$. 
Looking at the double bracket from Example \ref{Exmp:SympGL}, 
we should take $\bbmu_L=a$ and the moment map is $\mu_L\colon (h,\xi)\mapsto \xi$. 
On the other hand 
\begin{equation}
 \br{(\mu_R)_{ij},b_{kl}}=E_{ji}\cdot_R b_{kl}=-\delta_{kj} b_{il}\,, \quad 
  \br{(\mu_R)_{ij},a_{kl}}=E_{ji}\cdot_R a_{kl}=0\,, 
\end{equation}
or, using \eqref{Eq:relPA}, $\dgal{\bbmu_R,b}=-1\otimes b$ and $\dgal{\bbmu_R,a}=0$ 
with $\bbmu_R$ satisfying $\mu_R=X(\bbmu_R)$. 
It is a simple computation that the double bracket from Example \ref{Exmp:SympGL} yields 
these equalities for $\bbmu_R=-bab^{-1}$, i.e. the moment map $\mu_R \colon (h,\xi)\mapsto -h\xi h^{-1}$. 
E.g. 
\begin{equation*} 
\dgal{-bab^{-1},b}=-b\ast \dgal{a,b} \ast b^{-1} = - bb^{-1}\otimes b1=-1\otimes b\,. 
\end{equation*}

%%%%%%%%%%%%%%%%%% BIBLIOG %%%%%%%%%%%%%%%%%%%%%
%%%%%%%%%%%%%%%%%% BIBLIOG %%%%%%%%%%%%%%%%%%%%%
%%%%%%%%%%%%%%%%%% BIBLIOG %%%%%%%%%%%%%%%%%%%%%
%%%%%%%%%%%%%%%%%% BIBLIOG %%%%%%%%%%%%%%%%%%%%%

\end{document}